\documentclass{article}
\usepackage{amsmath, amsthm, amssymb, mathtools, leftindex, tikz-cd, verbatim, bbm}
\usepackage{enumitem}
\usepackage[parfill]{parskip}
\usepackage[a4paper, total={7in, 9in}]{geometry}
\usepackage[hidelinks]{hyperref}

\usepackage[maxbibnames=99, backend= biber, style = alphabetic]{biblatex}
\mathtoolsset{showonlyrefs}

\newtheorem{thm}{Theorem}[section]
\newtheorem{lem}[thm]{Lemma}
\newtheorem{prop}[thm]{Proposition}

\newtheorem{defi}[thm]{Definition}
\newtheorem{cor}[thm]{Corollary}

\newtheorem{remark}[thm]{Remark}

\newtheoremstyle{named}{}{}{\itshape}{}{\bfseries}{.}{.5em}{\thmnote{#3}}

\theoremstyle{named}

\newcommand{\psum}{\sideset{}{^*}\sum}

\newcommand{\C}{\mathcal{C}}

\newcommand{\E}{\mathcal{E}}

\newcommand{\HH}{\mathcal{H}}

\newcommand{\Z}{\mathbb{Z}}
\renewcommand{\L}{\mathcal{L}}

\newcommand{\A}{\mathbb{A}}
\newcommand{\Q}{\mathbb{Q}}
\renewcommand{\S}{\mathcal{S}}

\newcommand{\q}{\mathfrak{q}}

\newcommand{\CC}{\mathbb{C}}
\newcommand{\Hyp}{\mathbb{H}}
\newcommand{\R}{\mathbb{R}}
\newcommand{\T}{\mathcal{T}}
\renewcommand{\a}{\mathfrak{a}}
\renewcommand{\b}{\mathfrak{b}}

\newcommand{\N}{\mathcal{N}}

\renewcommand{\O}{\mathcal{O}}

\renewcommand{\P}{\mathbb{P}}

\DeclareMathOperator{\spn}{span}
\DeclareMathOperator{\Gal}{Gal}
\DeclareMathOperator{\Frob}{Frob}

\DeclareMathOperator{\SL}{SL}
\DeclareMathOperator{\GL}{GL}

\DeclareMathOperator{\Ind}{Ind}
\DeclareMathOperator{\sgn}{sgn}
\DeclareMathOperator{\vol}{Vol}

\DeclareMathOperator{\Ad}{Ad}
\DeclareMathOperator{\Tr}{Tr}
\DeclareMathOperator{\modulo}{mod}

\DeclareMathOperator{\nrd}{nrd}
\DeclareMathOperator{\lcm}{lcm}

\begin{document}
\title{Chebotarev geodesic theorem: split case.}
\author{Alberto Acosta Reche\thanks{This work was supported by the Engineering and Physical Sciences Research Council [EP/S021590/1], via the
EPSRC Centre for Doctoral Training in Geometry and Number Theory (The London School of Geometry and Number
Theory), University College London.}}
\date{}

\maketitle

\begin{abstract}
We study the prime geodesic theorem (PGT) in congruence classes of $\SL_2(\Z)$. We generalize previous work of Luo and Sarnak \cite{luosarnak} and of Soundararajan and Young \cite{Soundararajan_2013}, and prove that the geodesic analogue of the Chebotarev density theorem holds with exponent $25/36 + \varepsilon$. In particular, we deduce that the PGT holds with exponent $25/36 + \varepsilon$ for any congruence subgroup of $\SL_2(\Z)$. 
\end{abstract}

\section{Introduction.}\label{sec:introduction}
\subsection{Prime geodesic theorem.}
Let $\{\pm I\} \subset \Gamma \subset \SL_2(\R)$ be a discrete cofinite subgroup. Let $\{P\}_{\Gamma}$ range over hyperbolic $\Gamma$-conjugacy classes and let $X \geq 1$ be a parameter, to be taken large. The prime geodesic theorem (PGT), is concerned with the asymptotics as $X \to \infty$ of the sum 
\begin{equation}\label{eq:defiofprimegeodesicsum}
    \Psi_{\Gamma}(X) := \sum_{\substack{\{P\}_{\Gamma}\\N(P)\leq X\\
    \Tr(P) > 2}}\frac{\log N(P)}{\nu_\Gamma(P)}.
\end{equation}
Here $\nu_\Gamma(P)$ is the positive integer such that $P = P_0^{\nu(P)}$ for a $\Gamma$-primitive hyperbolic element $P_0 \in \Gamma$. The quantities $\log N(P)$ are the lengths of the closed geodesics of $\Gamma\backslash \Hyp$, and the conjugacy classes $\{P\}_\Gamma$ can be considered as analogues of rational prime powers. In this sense, $\Psi_{\Gamma}(X)$ can be considered as a geodesic analogue of the second Chebyshev function $x \mapsto \sum_{p^k \leq x} \log p$. 

In analogy with the prime number theorem, the PGT asserts that 
\begin{equation}\label{eq:weakprimegeodesic}
\Psi_\Gamma(X) \sim X \quad \text{ as }\quad X \to \infty.
\end{equation}
When $\Gamma\backslash \Hyp$ is compact this result goes back to Huber, see \cite[Satz 9]{huberI}. We can be much more precise in \eqref{eq:weakprimegeodesic}. Recall that the small eigenvalues of the Laplacian of $\Gamma\backslash \Hyp$ are those eigenvalues which lie in $[0, 1/4)$. Let $\{\lambda_j\}_j$ denote the finitely-many small eigenvalues of $\Gamma\backslash \Hyp$ counted with multiplicity. In particular, $0$ is always an eigenvalue with multiplicity one. Then we can write 
\begin{equation}\label{eq:primegeodesictheorem}
    \Psi_{\Gamma}(X) = \sum_{j}\frac{X^{1/2 + \sqrt{1/4 - \lambda_j}}}{1/2 + \sqrt{1/4 - \lambda_j}} +  E_{\Gamma}(X).
\end{equation}
When $\Gamma\backslash \Hyp$ is compact Huber proved that $E_\Gamma(X) = O(X^{5/6})$, see \cite[Satz III]{huberII}. Selberg observed \cite[page 179]{selberg63} that his trace formula could be used to derive results about the distribution of lengths of prime geodesics, but he did not publish a proof of the PGT. Sarnak followed Selberg's idea and provided a direct proof that $E_\Gamma(X) = O(X^{3/4}(\log X)^3)$ for any cofinite discrete group $\Gamma$, see \cite[Theorem 3.4]{sarnakthesis}. Iwaniec later refined this proof and obtained $E_\Gamma(X) = O(X^{3/4})$, see \cite[Theorem 10.5]{iwaniec}. By a less direct argument using the theory of the Selberg zeta function, Hejhal treated the general case of a multiplier system of arbitrary real weight and proved that $E_\Gamma(X) = O(X^{3/4}(\log X)^{1/2})$ \cite[Theorem 3.4]{hejhal2} in that generality. 

We consider the estimate $E_\Gamma(X) = O(X^{3/4})$ to be the baseline bound for the PGT. For a general cofinite group $\Gamma$ this estimate is yet to be improved. By analogy with the prime number theorem, one is lead to expect that the bound $E_{\Gamma}(X) = O_\varepsilon(X^{1/2 + \varepsilon})$ holds for any $\varepsilon$. This is believed to be true, but a proof is far from reach at the moment. If proven, this estimate would be optimal up to an arbitrarily small power of $X$. Indeed, it is known that 
\begin{equation*}
    E_\Gamma(X) = \Omega_{\pm}(X^{1/2-\varepsilon})
\end{equation*}
for any $\varepsilon > 0$, see \cite[note 18 in page 503]{hejhal2}. 

The baseline bound on $E_\Gamma(X)$ has been improved in particular cases of $\Gamma$. The case that has attracted the most attention is $\Gamma = \SL_2(\Z)$ because of its connection with class numbers of quadratic orders. Indeed, Sarnak showed \cite{sarnak82} that
\begin{equation}\label{eq:connectiontoclassnumbers}
    \Psi_{\SL_2(\Z)}(X) = \sum_{2 < t \leq X^{1/2} + X^{-1/2}}\sum_{du^2 = t^2 - 4} h(d) \log(\epsilon_d^2) = 2\sum_{2 < t \leq X^{1/2} + X^{-1/2}} \sum_{du^2 = t^2 - 4}  \sqrt{d}L(1; \chi_d),
\end{equation}
where the inner sums are over positive integers $d$ dividing $t^2 - 4$ and such that the quotient is a square. Here, $h(d)$ and $\log \epsilon_d$ are the class number and the regulator, respectively, of the unique quadratic order of discriminant $d$, $\chi_d$ is the quadratic character (not necessarily primitive) associated to the discriminant $d$, and $L(s, \chi_d)$ is the corresponding Dirichlet series. The second equality follows from the class number formula. 

When $\Gamma = \SL_2(\Z)$, improvements to the baseline bound were given, chronologically, by Iwaniec, by Luo and Sarnak, and by Cai, see \cite{Iwaniec1984}, \cite{luosarnak} and \cite{Cai2002}, respectively. The bounds they proved for $E_{\SL_2(\Z)}(X)$ were $O_\varepsilon(X^{35/48 + \varepsilon}), O_\varepsilon(X^{7/10 + \varepsilon})$ and $O_\varepsilon(X^{71/102 + \varepsilon})$, respectively. 

The last improvement was due to Soundararajan and Young, see \cite[Theorem 1.1]{Soundararajan_2013}. Combining the work of Luo and Sarnak with a clever use of Equation \eqref{eq:connectiontoclassnumbers}, Soundararajan and Young proved that $E_{\SL_2(\Z)}(X) = O_\varepsilon(X^{2/3 + \theta/6 + \varepsilon})$. Here $\theta > 0$ is any subconvexity exponent in conductor aspect for quadratic Dirichlet characters. By work of Conrey and Iwaniec \cite[Corollary 1.5]{conreyiwaniec} we can take $\theta = 1/6 + \varepsilon$ for any $\varepsilon > 0$. Thus, Soundararajan and Young were able to prove that $E_{\SL_2(\Z)}(X) = O_\varepsilon(X^{25/36 + \varepsilon})$ for any $\varepsilon > 0$. This estimate remains the best unconditional bound for the error term in the PGT at the moment of writing this paper.

Since there are no nonzero small eigenvalues in the case of $\Gamma = \SL_2(\Z)$, the result of Soundararajan and Young simply says that 
\begin{equation*}
    \Psi_{\SL_2(\Z)} = X + O_\varepsilon(X^{25/36 + \varepsilon}), \quad \text{ for any }\varepsilon > 0.
\end{equation*}

Regarding other cofinite groups, Luo, Rudnick and Sarnak claimed in the last paragraph of \cite{sarnakluorudnick1995} that the arguments in \cite{luosarnak} could be generalized to prove that $E_\Gamma(X) = O(X^{7/10 + \varepsilon})$ holds for any congruence subgroup $\Gamma$, but they did not provide further details. 

Later on Cherubini, Wu and Zábrádi generalized the arguments of Soundararajan and Young to the case where $\Gamma = \Gamma(q)$ is a principal congruence subgroup of $\SL_2(\Z)$. Thus, assuming the claim of Luo, Rudnick and Sarnak, they proved that $E_{\Gamma(q)}(X) = O_{q, \varepsilon}(X^{2/3 + \theta/6 + \varepsilon})$ for any $q \geq 1$, see \cite[Theorem 1.10]{cherubini_wu}.

Finally, the baseline bound has also been improved for certain cocompact groups built from quaternion algebras. Let $\O$ be a maximal order in $B$, an indefinite quaternion division algebra over $\Q$. Let $\nrd: B^\times \rightarrow \Q^\times$ denote the reduced norm map, and let $\O^1 := \{x \in \O : \nrd(x) = 1\}$. Via the isomorphism $B \otimes_\Q \R \simeq M_2(\R)$ we can view $\O^1$ as a cocompact discrete group of $\SL_2(\R)$. Using the Jacquet-Langlands correspondence, Koyama deduced from the result of Luo, Rudnick and Sarnak that $E_{\O^1}(X) = O_\varepsilon(X^{7/10 + \varepsilon})$ holds for any $\varepsilon > 0$, see \cite{koyama}. More recently, Tang, Wu, Yang and Yang have generalized the bound of Koyama to principal congruence subgroups of $\O^1$, see \cite{tang2025}.

\subsection{Chebotarev geodesic theorem.}\label{subsec:chebotarevGL2}
Let $\{\pm I\} \subset \Gamma' \subset \Gamma$ be two cofinite discrete groups, such that $\Gamma'$ is normal in $\Gamma$. In his PhD thesis Sarnak compared the normal covering $\Gamma'\backslash \Hyp \rightarrow \Gamma\backslash \Hyp$ to a Galois extension of number fields $K/F$, see \cite[section 3]{sarnakthesis}. The analogy is summarized in the following table. 
\begin{center}
\begin{tabular}{|c|c|c|c|c|c|}
    \hline
Classical &$F$ & $K$ & $\Gal(K/F)$ & unr. prime $\mathfrak{p}$ & $\mathfrak{p} \mapsto \{\Frob_\mathfrak{p}\}_{\Gal(K/F)}$\\
\hline
Geodesic & $\Gamma \backslash \Hyp$ & $\Gamma' \backslash \Hyp$ & $\Gamma'\backslash \Gamma$ &  prim. hyp. $\{P_0\}_\Gamma$ & $\{P_0\}_\Gamma \mapsto \{P_0 \modulo \Gamma'\}_{\Gamma'\backslash \Gamma}$\\
\hline
\end{tabular}
\end{center}
In particular, the geodesic analogue of the Frobenius map simply sends a $\Gamma$-primitive hyperbolic conjugacy class $\{P\}_\Gamma \subset \Gamma$ to the conjugacy class of $P \modulo \Gamma'$ inside $\widetilde{\Gamma} := \Gamma'\backslash \Gamma$. Let $\mathcal{C}$ be a conjugacy class of $\widetilde{\Gamma}$. In view of the table above, the geodesic analogue of the Chebotarev density theorem is the asymptotic
\begin{equation}\label{eq:weakchebotarev}
    \sum_{\substack{\{P\}_{\Gamma}\\N(P)\leq X\\
    \Tr(P) > 2}}\frac{\log N(P)}{\nu_\Gamma(P)} \mathbbm{1}_{\C}(P \modulo \Gamma') \sim \frac{|\C|}{|\widetilde{\Gamma}|}X
\end{equation}
as $X \to \infty$. Following Sarnak we call a result like this a \emph{Chebotarev geodesic theorem} (CGT). As with the PGT, one can give a much more precise version of \eqref{eq:weakchebotarev}.

Let $f$ be a class function on $\widetilde{\Gamma}$ and introduce the twisted Chebyshev sum 
\begin{equation}\label{eq:twistedgeodesicsum}
    \Psi_{\Gamma}(X; f) := \sum_{\substack{\{P\}_{\Gamma}\\N(P)\leq X\\
    \Tr(P) > 2}}\frac{\log N(P)}{\nu_\Gamma(P)} f(P \modulo \Gamma').
\end{equation}
If we define the quantity
\begin{equation*}
    \langle f, \mathbbm{1}_{\widetilde{\Gamma}} \rangle_{\widetilde{\Gamma}} := \frac{1}{|\widetilde{\Gamma}|}\sum_{g \in \widetilde{\Gamma}} f(g),
\end{equation*}
then Equation \eqref{eq:primegeodesictheorem} can be generalized to 
\begin{equation}\label{eq:chebotarevgeodesictheorem}
    \Psi_{\Gamma}(X; f) = \langle f, \mathbbm{1}_{\widetilde{\Gamma}}\rangle_{\widetilde{\Gamma}} X+ \sum_j c_j(f) \frac{X^{1/2 + \sqrt{1/4 - \lambda_j}}}{1/2 + \sqrt{1/4 - \lambda_j}} + E_{\Gamma}(X; f),
\end{equation}
where the $j$-sum is over the nonzero small eigenvalues of $\Gamma'$, and each of the coefficients $c_j(f)$ can be expressed as a sum of inner products of $f$ with characters of representations of $\widetilde{\Gamma}$. Sarnak proved that $E_{\Gamma}(X; f) = O_f(X^{3/4}\log(X)^3)$, see \cite[Theorem 3.16]{sarnakthesis}. Following the proof of Theorem 10.5 in \cite{iwaniec} one can remove the factor $\log(X)^3$ without any difficulty. As with the PGT, we consider the estimate $E_\Gamma(X; f) = O_f(X^{3/4})$ to be the baseline bound for the CGT. It seems that even in the case of congruence subgroups of $\SL_2(\Z)$ the baseline bound has not been improved. 
\subsection{Results.}

In order to state our main result, we need to introduce a somewhat ad-hoc notion of subconvexity exponent.

\begin{defi}\label{defi:admissibleexponent}
We say that $\theta > 0$ is \emph{admissible} if for any $q \geq 1$ and any Dirichlet character $\chi_1$ of level $q$, there exists $A > 0$, possibly depending on $q$, such that 
\begin{equation*}
    |L(s; \chi_1 \chi_2)| \ll_{q, \theta} \q(\chi_2)^\theta |s|^A
\end{equation*}
holds for $\text{Re}(s) = 1/2$ and any quadratic Dirichlet character $\chi_2$.
\end{defi}
In this definition $\q(\chi_2)$ is the conductor of $\chi_2$. By the work of Petrow and Young we know that the exponent $\theta = 1/6 + \varepsilon$ is admissible for any $\varepsilon > 0$, see \cite{petrowyoung20} and \cite{petrowyoung}. The generalized Lindelöf hypothesis predicts that $\theta = \varepsilon$ is admissible for any $\varepsilon > 0$.  

Let $q$ be a positive integer, $\Gamma$ a congruence subgroup of level $q$, $\Gamma' = \Gamma(q)$ and $f$ a class function on $\widetilde{\Gamma} = \Gamma'\backslash \Gamma$. In this case the geodesic analogue of the Frobenius map is simply reduction modulo $q$. Note that, by work of Kim and Sarnak \cite[second appendix]{kimramakrishnansarnak2003} we know that the smallest nonzero eigenvalue of $\Gamma(q)$ satisfies $\lambda_1 \geq 1/4 - (7/64)^2$. Similarly as before define 
\begin{equation}\label{eq:chebyshevsumsgeneral}
    \Psi_{\Gamma}(X; f) := [\{\pm I\}\Gamma: \Gamma] \times \sum_{\substack{\{P\}_\Gamma\\
    N(P) \leq X\\
    \Tr(P) > 2}} \frac{\log N(P)}{\nu_\Gamma(P)}f(P \modulo q),
\end{equation}
where the factor $[\{\pm I\}\Gamma: \Gamma]$ accounts for the fact that the group $\Gamma$ does not necessarily contain $-I$. Our main result is the following theorem.
\begin{thm}\label{thm:maintheorem1}
For any congruence subgroup $\Gamma$, any class function $f$ on $\widetilde{\Gamma}$ and any admissible exponent $\theta > 0$ we have
\begin{equation*}
    \Psi_\Gamma(X; f) = \langle f, \mathbbm{1}_{\widetilde{\Gamma}}\rangle_{\widetilde{\Gamma}} X + O_{f, \theta, \varepsilon}(X^{2/3 + \theta/6 + \varepsilon}).
\end{equation*}
\end{thm}
Since $\theta = 1/6 + \varepsilon$ is admissible we deduce unconditionally that
\begin{equation*}
    \Psi_{\Gamma}(X; f) = \langle f, \mathbbm{1}_{\widetilde{\Gamma}}\rangle_{\widetilde{\Gamma}} X + O_{f, \varepsilon}(X^{25/36 + \varepsilon}).
\end{equation*}
This result extends the estimate of Soundararajan and Young \cite[Theorem 1.1]{Soundararajan_2013} to an arbitrary class function $f$ of $\widetilde{\Gamma}$, whereas the original result applied only to constant functions on $\SL_2(\Z/q\Z)$. To illustrate the generality of Theorem \ref{thm:maintheorem1}, let us consider two types of class functions.

Suppose that $\{\pm I\} \subset \Gamma$ and let $f = \mathbbm{1}_{\widetilde{\Gamma}}$. Then $\Psi_{\Gamma}(X; \mathbbm{1}_{\widetilde{\Gamma}}) = \Psi_\Gamma(X)$ as defined earlier in \eqref{eq:defiofprimegeodesicsum}.
\begin{cor}\label{cor:pgtforanycongruencesubgroup}
Let $\{\pm I\} \subset \Gamma \subset \SL_2(\Z)$ be a congruence subgroup. Then for any admissible $\theta>0$ we have 
\begin{equation*}
    \Psi_\Gamma(X) = X + O_{\Gamma, \theta, \varepsilon}(X^{2/3 + \theta/6 + \varepsilon}).
\end{equation*}
\end{cor}
Thus, the PGT holds with error term $E_\Gamma(X) = O_{\varepsilon}(X^{25/36 + \varepsilon})$ for an arbitrary congruence subgroup. This result is a strengthened version of \cite[Corollary 1.2]{sarnakluorudnick1995}.

Another interesting family of class functions $f$ is the one consisting of characteristic functions of fibers of the trace map. When $\Gamma = \SL_2(\Z)$, the resulting Chebyshev sums are particularly interesting, since they are averages of class numbers with congruence conditions. Indeed, from \eqref{eq:connectiontoclassnumbers} it follows that for any $a \in \Z/q\Z$ we have
\begin{equation}
    \Psi_{\SL_2(\Z)}(X; \mathbbm{1}_{\Tr^{-1}(a)}) = \sum_{\substack{2 < t \leq X^{1/2} + X^{-1/2}\\
    t = a \modulo q}}\sum_{du^2 = t^2 - 4} h(d) \log(\epsilon_d^2) = 2\sum_{\substack{2 < t \leq X^{1/2} + X^{-1/2}\\
    t = a \modulo q}}\ \sum_{du^2 = t^2 - 4}  \sqrt{d}L(1; \chi_d).
\end{equation}
Theorem \ref{thm:maintheorem1} allows us to evaluate these sums very accurately.

\begin{cor}\label{cor:geodesicdirichlet}
For any positive integer $q$, any congruence subgroup $\Gamma$ of level $q$, any $a \in \Z/q\Z$ and any admissible exponent $\theta>0$ we have 
\begin{equation*}
    \Psi_\Gamma(X; \mathbbm{1}_{\Tr^{-1}(a)}) = \frac{|\{\gamma \in \widetilde{\Gamma} \mid \Tr(\gamma) = a \modulo q\}|}{|\widetilde{\Gamma}|}X + O_{q, \theta, \varepsilon}(X^{2/3 + \theta/6 + \varepsilon}).
\end{equation*}
\end{cor}

Corollary \ref{cor:geodesicdirichlet} can be seen as the geodesic analogue of Dirichlet's theorem on arithmetic progressions. It generalizes and improves on the results of \cite{Chatzakos_2024} and \cite{deitmar2025distributionprimegeodesictraces} in the $X$-aspect. These two papers considered only the case $\Gamma = \SL_2(\Z)$ and obtained a worse bound for the error term in the $X$-aspect. 

\begin{remark}
If $q \geq 3$, then $\{\pm I\} \not \subset \Gamma(q)$. Thus, in the case of congruence subgroups of $\SL_2(\Z)$ our setup is a little bit more general than that of \cite[Theorem 3.16]{sarnakthesis}. The drawback of this small subtlety is that, in order to prove Theorem \ref{thm:maintheorem1}, we need to deal with Maass forms of both weights $0$ and $1$, whereas Sarnak worked only with Maass forms of weight $0$. The reward, however, is that we are able to deduce Corollary \ref{cor:geodesicdirichlet} as stated. If we had worked with $\Gamma'=\{\pm I\}\Gamma(q)$ instead of $\Gamma(q)$ we would only have been able to study the combined sums $\Psi_\Gamma(X; \mathbbm{1}_{\Tr^{-1}(a)}) + \Psi_\Gamma(X; \mathbbm{1}_{\Tr^{-1}(-a)})$, instead of each of them separately.
\end{remark}

Before stating other results it is convenient to introduce a definition.

\begin{defi}\label{defi:congruencepair}
We say that $(\Gamma, \chi)$ is a \emph{congruence character of level $q$} if $\{\pm I_2\}\Gamma(q)\subset \Gamma \subset \SL_2(\Z)$ and $\chi: \Gamma\rightarrow S^1$ is a character which is trivial on $\Gamma(q)$.
\end{defi}

By a formal argument, Theorem \ref{thm:maintheorem1} is reduced to the study of the sums $\Psi_\Gamma(X; \chi)$ where $(\Gamma, \chi)$ is a congruence character of level $q$, see Section \ref{sec:reductiontocongruencecharacters}. We let $k \in \{0, 1\}$ be defined implicitly by the equation $\chi(-I) = (-1)^k$. Let $\{u_j\}_{j\geq 1}$ be an orthonormal basis of Maass cusp forms of weight $k$ for $(\Gamma, \chi)$. Let the eigenvalue of $u_j$ be parametrized as $\lambda_j = 1/4 + t_j^2$, where either $t_j \geq 0$ or $\text{Im}(t_j) > 0$.

By adapting the proof of \cite[Lemma 1]{Iwaniec1984} we can show that
\begin{equation}\label{eq:iwaniecslemmaintro}
    \Psi_\Gamma(X; \chi) = \delta_{\chi = 1} X + \sum_{|t_j|\leq T}\left(\frac{X^{1/2 + it_j}}{1/2 + it_j} + \frac{X^{1/2 -it_j}}{1/2 - it_j}\right) + O_q\left(\frac{X}{T}\log^2 X\right)
\end{equation}
for parameters $X \geq 10$ and $1 \leq T \leq \sqrt{X}(\log X)^{-2}$, see Section \ref{sec:iwaniecslemma}. Thus, the Chebyshev sums $\Psi_{\Gamma}(X; \chi)$ are connected with the exponential sums 
\begin{equation}\label{eq:defspectralexponential}
    \S(T, X) := \sum_{0 \leq t_j \leq T} X^{it_j}.
\end{equation}
By a weak version of Weyl's law (see \eqref{eq:basicweyllaw} below) the trivial bound on $\S(T, X)$ is 
\begin{equation}\label{eq:trivialboundexponentialsum}
    \S(T, X) \ll_q T^2,
\end{equation}
which combined with \eqref{eq:iwaniecslemmaintro} recovers the baseline bound $E_\Gamma(X; \chi) = O_q(X^{3/4}\log X)$. 
Following the prediction of \cite[Conjecture 2.2]{petridisrisagerlocalaverage} we expect that for any $X, T \geq 1$ we have
\begin{equation*}
    \S(T, X) \ll_{q, \varepsilon}T(XT)^\varepsilon.
\end{equation*}
If true, this estimate would imply the Chebotarev geodesic theorem with the error term $O_{q, \varepsilon}(X^{1/2 + \varepsilon})$. Unfortunately, a bound of such strength is far from reach at the moment. On the positive side, we do have two very useful estimates on $\S(T, X)$. The first of them generalizes Equation (58) in \cite{luosarnak}. 

\begin{thm}\label{thm:luosarnakbound}
Let $(\Gamma, \chi)$ be a congruence character of level $q$. Then, for $1 \leq T \leq X^{1/2}$ we have 
\begin{equation}\label{eq:luosarnakbound}
   \S(T, X) \ll_{\varepsilon, q} T^{5/4}X^{1/8} (XT)^\varepsilon.
\end{equation}
\end{thm}

The second estimate generalizes \cite[Theorem 1.2]{balkanova_frolenkov}. 

\begin{thm}\label{thm:balkanovafrolenkovbound}
Let $(\Gamma, \chi)$ be a congruence character of level $q$. Let $\theta > 0$ be an admissible exponent. Then, for $1 \leq T \leq X^{1/2}$ we have 
\begin{equation}\label{eq:balkanovafrolenkovbound}
\S(T, X) \ll_{\varepsilon, q,\theta} X^{1/4 + \theta/2}T^{1/2} (XT)^\varepsilon.
\end{equation}
\end{thm}

This result is closely connected to \cite[Theorem 3.2]{Soundararajan_2013}, which we are also able to generalize.
\begin{thm}\label{thm:soundyoungshortintervals}
Let $\Gamma$ be a congruence subgroup of level $q$. Let $f$ be a class function on $\widetilde{\Gamma}:= \Gamma(q)\backslash \Gamma$. Let $\theta > 0$ be an admissible exponent. Then, for any $u \geq 1$ we have
\begin{equation*}
    \Psi_\Gamma(X + u; f) - \Psi_\Gamma(X; f) = \langle f, \mathbbm{1}_{\widetilde{\Gamma}}\rangle_{\widetilde{\Gamma}} u + O_{f, \theta, \varepsilon}(u^{1/2} X^{1/4 + \theta/2 + \varepsilon}).
\end{equation*}
\end{thm}
\begin{proof}[Proof of Theorem \ref{thm:maintheorem1} assuming Theorems \ref{thm:luosarnakbound} and \ref{thm:balkanovafrolenkovbound}.]

Interpolating between bounds \eqref{eq:luosarnakbound} and \eqref{eq:balkanovafrolenkovbound} we deduce that 
\begin{equation*}
    \S(T, X) \ll_{q, \theta, \varepsilon} T X^{1/6 + \theta/6 + \varepsilon}
\end{equation*}
when $1 \leq T \leq X^{1/2}$. Integration by parts then leads to 
\begin{equation*}
    \sum_{|t_j|\leq T}\left(\frac{X^{1/2 + it_j}}{1/2 + it_j} + \frac{X^{1/2 -it_j}}{1/2 - it_j}\right) \ll_{q, \theta, \varepsilon} X^{2/3 + \theta/6 + \varepsilon}.
\end{equation*}
Therefore, letting $T = X^{1/3 - \theta/6}$ in \eqref{eq:iwaniecslemmaintro} we obtain  
\begin{equation*}
    \Psi_\Gamma(X; \chi) = \delta_{\chi = 1}X + O_{q, \theta, \varepsilon}(X^{2/3 + \theta/6 + \varepsilon}),
\end{equation*}   
which finishes the proof of Theorem \ref{thm:maintheorem1}. 

\end{proof}

\begin{remark}
By following the discussion in \cite[Section 3]{Soundararajan_2013}, one could also give a proof of Theorem \ref{thm:maintheorem1} that uses Theorem \ref{thm:soundyoungshortintervals} instead of Theorem \ref{thm:balkanovafrolenkovbound}. 
\end{remark}

We now sketch the proofs of Theorem \ref{thm:luosarnakbound} and Theorem \ref{thm:balkanovafrolenkovbound}. By an observation due to Iwaniec it is enough to prove the same bounds for the smoothed sums 
\begin{equation*}
    \S_0(T, X) := \sum_{t_j \geq 0}X^{it_j}\exp(-t_j/T),
\end{equation*}
see Section \ref{subsec:smoothingtheexponentialsum}. For the proof of Theorem \ref{thm:luosarnakbound} we follow the strategy of Luo and Sarnak from the appendix to \cite{luosarnak}. After applying the Bruggeman--Kuznetsov trace formula, using Weil's bound on Kloosterman sums and removing the harmonic weights we arrive at the estimate 
\begin{equation}\label{eq:estimateforS0afterapplyingkuznetsov}
\begin{aligned}
    \S_0(T, X) & \ll_\varepsilon N^{1/2}T^{1/2}X^{1/4}(NX)^\varepsilon\\
    & +  N^{-1/2}\int_{(1/2)} |s|^{-A} \left(\sum_{t_j \geq 0} |\L(s; |u_j|^2)| \omega(t_j) \exp(-|t_j|/T) \right)\, d|s|,
\end{aligned}
\end{equation}
where $A > 0$ is arbitrary, see Section \ref{sec:proofofluosarnakbound}. Here $N > 0$ is a parameter to be chosen later, $\omega(t_j)$ is a spectral weight satisfying $\omega(t_j) \asymp (1 + |t_j|)^k \exp(-|t_j|/T)$ and
\begin{equation}\label{eq:naiverankinselbergintro}
    \L(s; |u_j|^2) := \sum_{n \in \frac{1}{q}\Z_{\geq 1}} \frac{|\rho_j(n)|^2}{n^s}
\end{equation}
is the \emph{naive Rankin--Selberg series} studied in Section \ref{sec:rankinselbergandadjointlfunction}. In this definition the numbers $\rho_j(n)$ are Fourier coefficients of $u_j$ defined in \eqref{eq:generalfourierexpansion}.

The task now is to give, for some choice of orthonormal basis $\{u_j\}_j$, a good estimate for the spectral average of $|\L(s; |u_j|^2)|$ on the critical line. In the case $\Gamma = \SL_2(\Z)$, which is the setup in which Luo and Sarnak worked, we can assume that the $u_j$ are eigenforms of all the Hecke operators. Thus, for each $j$ there exists an automorphic cuspidal representation $\pi_j$ of $GL(2)$, unramified everywhere, such that 
\begin{equation*}
    \L(s; |u_j|^2) = |\rho_j(1)|^2 \frac{L(s; \pi_j \otimes \overline{\pi_j})}{\zeta(2s)},
\end{equation*}
where $\rho_j(1)$ is the first Fourier coefficient of $u_j$,  $\zeta(s)$ is the Riemann zeta function and $L(s; \pi_j \otimes \overline{\pi_j})$ is the Rankin--Selberg $L$-function of $\pi_j$ and $\overline{\pi}_j$ in the sense of \cite{jacquetrankinselberg}. By work of Iwaniec \cite{iwaniecsmalleigenvalues} and of Hoffstein and Lockhart \cite{hoffsteinlockhart} we know that 
\begin{equation}\label{eq:estimationfirstfouriercoefficientintro}
    (1 + |t_j|)^{-\varepsilon} \ll_{\varepsilon} \omega(t_j)|\rho_j(1)|^2 \ll_\varepsilon (1 + |t_j|)^\varepsilon.
\end{equation}
By work of Gelbart and Jacquet \cite{gelbartjacquetlift}, who generalized previous work of Shimura \cite{shimuraholomorphy}, we can write 
\begin{equation*}
    L(s; \pi_j \otimes \overline{\pi_j}) = \zeta(s)L(s; \Ad(\pi_j)),
\end{equation*}
where $L(s; \Ad(\pi_j))$ is the standard $L$-function associated to $\Ad(\pi_j)$, a cuspidal automorphic form of $\GL(3)$. The spectral average of $L(s; \Ad(\pi_j))$ can be estimated by using the spectral large sieve inequality of Deshouillers and Iwaniec \cite[Theorem 2]{deshouillersiwaniec82} together with \eqref{eq:estimationfirstfouriercoefficientintro}, giving
\begin{equation}\label{eq:lindelofonaverageforadjointLfunctionsintro}
    \sum_{|t_j| \leq T} |L(s; \Ad(\pi_j))|^2 \ll_\varepsilon T^{2 + \varepsilon} |s|^A
\end{equation}
for some $A > 0$ whose value is not important for this discussion. By Weyl's law, this estimate is as strong as the Lindelöf hypothesis on average. Another application of \eqref{eq:estimationfirstfouriercoefficientintro} and the previous discussion gives
\begin{equation}\label{eq:lindelofonaverageRankinselbergintro}
    \sum_{|t_j| \leq T} \omega(t_j)^2|\L(s; |u_j|^2)|^2 \ll_\varepsilon T^{2 + \varepsilon} |s|^A,
\end{equation}
where the value of $A >0$ may be larger. Luo and Sarnak obtained this estimate in the case when $\Gamma = \SL_2(\Z)$. Inserting it in \eqref{eq:estimateforS0afterapplyingkuznetsov} and optimizing the value of $N$ finishes the proof of Theorem \ref{thm:maintheorem1} in this case.

The main obstacle in generalizing the argument of Luo and Sarnak is the following: given an arbitrary congruence pair $(\Gamma, \chi)$, we need to construct an orthonormal basis $\{u_j\}_j$ whose Rankin--Selberg series are expressible in terms of $L$-functions. We solve this problem in Section \ref{sec:orthonormalbasis}, where we show that the orthonormal basis $\{u_j\}_j$ can be chosen in such a way that for each $j\geq 1$ there is a cuspidal automorphic form $\pi_j$ of $\GL(2)$ such that 
\begin{equation}\label{eq:keyapproxidentityrankinselberg}
    \omega(t_j)\L(s; |u_j|^2) \approx \sum_{\psi \modulo q} L(s; \pi_j \otimes \overline{\pi_j} \otimes \psi),
\end{equation} 
see \eqref{eq:naiverankinselbergdecomposition} and \eqref{eq:gelbartjacquetlift} below. After this structural identity is established we can adapt the argument of Luo and Sarnak and prove \eqref{eq:lindelofonaverageRankinselbergintro} for an arbitrary congruence character. Inserting this estimate in \eqref{eq:estimateforS0afterapplyingkuznetsov} and optimizing the value of $N$ we can finish the proof of the general case of Theorem \ref{thm:luosarnakbound}, see Section \ref{sec:proofofluosarnakbound}.

The proofs of Theorem \ref{thm:balkanovafrolenkovbound} and Theorem \ref{thm:soundyoungshortintervals} are quite similar to each other, and they are both based on the arguments of \cite{Soundararajan_2013}. In Section \ref{sec:bykovskiizagierseries} we define the \emph{Bykovskii--Zagier} series $Z(s; t)$, where $|t| \geq 3$, see \eqref{eq:defiofzagierbykovskiizeta} for the definition. The Dirichlet series $Z(s; t)$ has analytic continuation past a simple pole at $s = 1$ with residue
\begin{equation*}
    \text{Res}_{s = 1}Z(s; t) = \sum_{\substack{\{P\}_\Gamma\\
    \Tr(P) = t}} \frac{2 \log N(P_0) \chi(P)}{\pi \vol(\Gamma\backslash \Hyp)(N(P)^{1/2} - N(P)^{-1/2})} ,
\end{equation*} 
Thus, the hyperbolic contribution to the Selberg trace formula can be expressed as a weighted sum over residues of $Z(s; t)$. This observation goes back to Zagier \cite{zagierrankinselberg}. We can express the Bykovskii--Zagier series in terms of Dirichlet $L$-functions, roughly as
\begin{equation*}
    Z(s; t) \approx \sum_{\psi \modulo q} \frac{L(s, \psi)}{L(2s, \psi)}L(s, \psi_D \psi),
\end{equation*}
see the proof of Proposition \ref{prop:boundzagierzetafunction}. In this expression $t^2- 4 = Df^2$ for a fundamental discriminant $D >0$ and $\psi_D$ is the quadratic Dirichlet associated to the discrimimant $D$. By the definition of admissible exponent we have the bound 
\begin{equation}\label{eq:boundonzagierintro}
    Z(s; t) \ll_{q, \varepsilon} |t|^{2\theta + \varepsilon} |s|^A
\end{equation}
on $\text{Re}(s) = 1/2$. Using this bound we can adapt the original argument of Soundararajan and Young and prove Theorems \ref{thm:balkanovafrolenkovbound} and \ref{thm:soundyoungshortintervals} without much difficulty, see Section \ref{sec:shortintervals}.

\begin{remark}
Our proof of Theorem \ref{thm:balkanovafrolenkovbound} differs significantly from the arguments of \cite{balkanova_frolenkov}. Instead of following their method, we have adapted more directly the ideas of Soundararajan and Young from \cite{Soundararajan_2013}. As a result, our proof is shorter and simpler, since by appealing to the Selberg trace formula instead of the Bruggeman--Kuznetsov trace formula we arrive at the Bykovskii--Zagier functions more quickly than in \cite{balkanova_frolenkov}. In addition, the treatment of special functions is straightforward in our proof.
\end{remark}

\begin{remark}
In a follow-up paper we generalize the previous results to cocompact Fuchsian groups arising from indefinite quaternion algebras. 
\end{remark}

\subsection{Outline of the paper.}
In Section \ref{sec:reductiontocongruencecharacters} we show that the Chebyshev sums $\Psi(X; \cdot)$ are compatible with induction of class functions, see Lemma \ref{lem:compatibilitywithinduction}. From this observation and Artin's theorem on induced characters \cite[Corollary 9.2]{serrelinearreps} it follows that it is enough to prove Theorem \ref{thm:maintheorem1} and Theorem \ref{thm:soundyoungshortintervals} in the case of congruence characters.

In Section \ref{sec:spectralbackground} we state some background results on spectral theory of Maass forms of weight $0$ and $1$, including the Selberg trace formula and the Bruggeman--Kuznetsov trace formula. For the derivation of the Bruggeman--Kuznetsov we cite Proskurin \cite{proskurinarbitraryweightkuznetsov}, which treats the general case of arbitrary multiplier system. We face a small technical problem because the hypothesis that \cite{proskurinarbitraryweightkuznetsov} imposes on test functions are too strict for our application. Since we are dealing with the special case of congruence characters, we can take advantage of Weil's bound on Kloosterman sum to relax these conditions. We dedicate Section \ref{subsec:kuznetsovtraceformula} to carry out this argument in detail.

In Section \ref{sec:fouriercoefficientseisensteinseries} we bound the contribution from the continuous spectrum to the Bruggeman--Kuznetsov trace formula. The main idea is to reduce the estimate to the case of Hecke congruence subgroups and then use the explicit computations of \cite{youngeisensteinseries}. 

In Section \ref{sec:iwaniecslemma} we prove the identity in \eqref{eq:iwaniecslemmaintro}, generalizing an explicit formula of Iwaniec \cite[Lemma 1]{Iwaniec1984}. The original proof works verbatim after a few preliminary facts are collected from \cite{hejhal2}. 

In Section \ref{sec:orthonormalbasis} we prove the existence of an orthonormal basis of Maass cuspidal forms which are $(\Gamma, \chi)$-automorphic and for which an identity like \eqref{eq:keyapproxidentityrankinselberg} holds. Motivated by the results in \cite[Chapter 3]{shimurabook}, we study the action of a certain ring of Hecke operators $R(\Gamma(q), \Delta(q))$ on the larger space of $\Gamma(q)$-automorphic Maass forms, see \eqref{eq:defiofsemigroups}. We prove that the isotypic components for this action are the direct sums of oldspaces of newforms in the same twist class. Since the action of $R(\Gamma(q), \Delta(q))$ commutes with the action of $\SL_2(\Z)$, the subspace of $(\Gamma, \chi)$-automorphic forms naturally inherits a decomposition into $R(\Gamma(q), \Delta(q))$-isotypic components, see Corollary \ref{cor:basisaslinearcombinationoftwists}. 

In Section \ref{sec:rankinselbergandadjointlfunction} we establish the properties of the series $\L(s; |u_j|^2)$ which feature in the proof of Theorem \ref{thm:luosarnakbound}, including the crucial estimate in \eqref{eq:lindelofonaverageRankinselbergintro}, see Theorem \ref{thm:informationaboutrankinselberg}. Using the orthonormal basis from the previous section, the proof of \eqref{eq:lindelofonaverageRankinselbergintro} naturally divides into two parts. On the one hand we need to bound the contribution from ``bad primes'' which divide the level $q$. In this direction we use the explicit orthonormal basis for oldspaces constructed by Schulze-Pillot and Yenirce \cite{Schulze_Pillot_Petersson} together with partial progress towards the Ramanujan conjecture. On the other hand, we can deal with the contribution from ``good primes'' by generalizing the original argument of Luo and Sarnak without much difficulty.

In Section \ref{sec:proofofluosarnakbound} we give the proof of Theorem \ref{thm:luosarnakbound}, following \cite{luosarnak} very closely. The preliminaries for this proof are the Bruggeman--Kuznetsov trace formula, Weil's bound on Kloosterman sums and Theorem \ref{thm:informationaboutrankinselberg}.

In Section \ref{sec:bykovskiizagierseries} we introduce the Bykovskii--Zagier zeta function $Z(s; t)$ and establish some of its properties, generalizing the discussion from \cite{Soundararajan_2013}. Finally, the proofs of Theorems \ref{thm:balkanovafrolenkovbound} and \ref{thm:soundyoungshortintervals} are found in Section \ref{sec:shortintervals}. As we explained earlier in this introduction, Theorem \ref{thm:maintheorem1}, the main result of this paper, follows immediately from Equation \eqref{eq:iwaniecslemmaintro}, Theorem \ref{thm:luosarnakbound} and Theorem \ref{thm:balkanovafrolenkovbound}.

\subsection{Notation.}
\begin{itemize}
    \item $X \geq 1$, $T \geq 1$ are large parameters.
    \item $\Gamma \subset \SL_2(\R)$ is a cofinite discrete group, $\{P\}_\Gamma$ a hyperbolic conjugacy class.
    \item Given $P \in \Gamma$, we let $C_\Gamma(P):= \{\gamma \in \Gamma: \gamma P = P \gamma\}$.
    \item Given a cusp $\a \in \R \cup \{\infty\}$ we let $\Gamma_\a := \{\gamma \in \Gamma : \gamma \a = \a\}$.
    \item $\Psi_\Gamma(X)$, $\Psi_\Gamma(X; f)$ are Chebyshev-type sums over hyperbolic conjugacy classes, see \eqref{eq:defiofprimegeodesicsum} and \eqref{eq:twistedgeodesicsum}.
    \item $\langle \cdot, \cdot \rangle_G$ is the inner product with respect to probability measure on the finite group $G$. 
    \item $\mathbbm{1}_{A}(\cdot)$ is the characteristic function of a subset $A \subset Y$, where the ambient set $Y$ is always well-understood from the context.
    \item $\theta > 0$ is any admissible subconvexity exponent, see Definition \ref{defi:admissibleexponent}.
    \item From Section \ref{sec:spectralbackground} onwards $(\Gamma, \chi)$ is a fixed congruence character of level $q$, see Definition \ref{defi:congruencepair}. We let $k \in \{0, 1\}$ be such that $\chi(-I) = (-1)^k$, $\ell \geq 1$ is such that $\Gamma_\infty = \{\pm \begin{psmallmatrix}
        1 & \ell \Z\\
        0 & 1
    \end{psmallmatrix}\}$ and $\alpha \in [0, 1)$ is defined by $\chi(\begin{psmallmatrix}
        1 & \ell\\
        0 & 1
    \end{psmallmatrix}) = e(-\alpha)$.
    \item $\Gamma_0(q), \Gamma_1(q), \Gamma_d(q), \Gamma(q)$ denote the following classical congruence subgroups of $\SL_2(\Z)$:
    \begin{equation*}
    \begin{aligned}
        \Gamma_0(q) & := \{\begin{psmallmatrix}
            a & b\\
            c & d
        \end{psmallmatrix} \in \SL_2(\Z) : c = 0 \modulo q\},\quad \Gamma_1(q) := \{\begin{psmallmatrix}
            a & b\\
            c & d
        \end{psmallmatrix} \in \Gamma_0(q) : a = d = 1 \modulo q\},\\
        \Gamma_d(q) & := \{\begin{psmallmatrix}
            a & b\\
            c & d
        \end{psmallmatrix} \in \Gamma_0(q): b = 0 \modulo q\}, \quad \, \, \, \, \Gamma(q) := \{\begin{psmallmatrix}
            a & b\\
            c & d
        \end{psmallmatrix} \in \Gamma_1(q): b = 0 \modulo q\}.  
    \end{aligned}
    \end{equation*}
    \item Given $\gamma \in M_2(\R)$ we write $a(\gamma), b(\gamma), c(\gamma), d(\gamma)$ for its entries, organized in such a way that $\gamma = \begin{psmallmatrix}
    a(\gamma) & b(\gamma)\\
    c(\gamma) & d(\gamma)
\end{psmallmatrix}$. We let $\Tr(\gamma)$ denote the trace of $\gamma$.
    \item A Dirichlet character $\psi$ modulo $q$ gives rise to a character on $\Gamma_0(q)$ (and also on $\Gamma_d(q)$) given by $\gamma \mapsto \psi(d(\gamma))$. When no confussion is possible we will use the same letter to refer to the Dirichlet character, and the character it induces on $\Gamma_0(q)$ and on $\Gamma_d(q)$.
    \item The cocycle $j_\gamma(z)$ is defined by Equation \eqref{eq:deficocyclej}.
    \item For the definition of the slash operator $\cdot |_k \cdot$ see \eqref{eq:slashoperator} and also \eqref{eq:slashoperatordoublecosets}. 
    \item $P$ is the hyperbolic Laplacian operator of weight $k$, see \eqref{eq:laplacianoperatorweightk}. 
    \item $\rho(u; n)$ denotes the $n$-th Fourier coefficients of the $P$-eigenfunction $u$, see \eqref{eq:generalfourierexpansion}.
    \item $E_\a(z; s; \Gamma, \chi)$, or more briefly $E_\a(z; s)$, is the Eisenstein series attached to the singular cusp $\a$ of $(\Gamma, \chi)$, see \eqref{eq:defofeisensteinseries}.
    \item $W_{\alpha, \beta}(z)$ is the Whittaker function of parameters $\alpha, \beta$, as in \cite[Equation (4.19)]{DFIsubconvexityartin}.
    \item $L_0^2(\Gamma, \chi)$ is the cuspidal part of $L^2(\Gamma, \chi)$. We denote an orthonormal basis of cuspidal Maass forms by $\{u_j\}_{j \geq 1}$, let the corresponding eigenvalues of $P$ be $\{\lambda_j\}_j$, which we parametrize as $\lambda_j = 1/4 + t_j^2$, see the discussion around \eqref{eq:basicorthonormalbasis}.
    \item $\S(T, X)$ and $\S_0(T, X)$ are certain exponential sums over cuspidal eigenvalues, see \eqref{eq:defspectralexponential} and \eqref{eq:smoothversionofS}, respectively.
    \item $S_{(\Gamma, \chi)}(m, n; c)$ are Kloosterman sums for $(\Gamma, \chi)$, see \eqref{eq:generalkloostermansum} for the definition.
    \item $\C_{it}(\Gamma, \chi)$ is the space of cuspidal Maass forms of congruence character $(\Gamma, \chi)$, of Laplace eigenvalue $1/4 + t^2$ and of weight $k \in \{0, 1\}$ such that $\chi(-I) = (-1)^k$, see \eqref{eq:cuspidalspacefixedspectraldata}.
    \item $\HH_{it}(M, \psi)$ is the set of Hecke-normalized Hecke--Maass cuspidal newforms of conductor $M$, nebentypus $\psi$, weight $k \in \{0, 1\}$ such that $\psi(-1) = (-1)^k$, and Laplace eigenvalue $1/4 + t^2$. 
    \item $\HH'_{it}(q)$ is defined as 
    \begin{equation*}
        \HH'_{it}(q) := \bigsqcup_{\substack{\psi \modulo q\\
\psi(-1) = (-1)^k}} \bigsqcup_{d \mid q^2} \HH_{it}(d, \psi).
    \end{equation*}
    \item For the definitions of $\HH'_{it}(q)/\sim$ and of $\T_f(q)$, see \eqref{eq:equivalencerelationtwist}.  
    \item $V_f(\Gamma_0(M))$ is the classical oldspace for a newform $f$ whose conductor divides $M$, see \eqref{eq:oldspaceGamma0}.
    \item $V_f(\Gamma_d(q))$ is an oldspace for $f \in \HH'_{it}(q)$, and $V_{[f]}(\Gamma_d(q))$ is a direct sum of such oldspaces for twists of $f$ modulo $q$, see \eqref{eq:oldspaceGammad} and \eqref{eq:oldspacetwists} respectively.
    \item $\Delta(q)$ and $\Delta'(q)$ are semigroups of $\GL_2^+(\Q)$ defined in \eqref{eq:defiofsemigroups}.
    \item $\L(s; |u_j|)^2$ is the naive Rankin--Selberg Dirichlet series attached to $u_j$, see \eqref{eq:naiverankinselberg}.
    \item $L(s; f_1 \otimes f_2)$ is the Rankin--Selberg $L$-function of the pair of newforms $f_1, f_2$.
    \item $L(s; \Ad(f)\otimes \psi)$ is the $\psi$-twist of the adjoint $L$-function attached to the newform $f$, see Equation \eqref{eq:gelbartjacquetlift}.
    \item $Z(s; t)$ is the Bykovskii--Zagier zeta function, see \eqref{eq:defiofzagierbykovskiizeta}.  
\end{itemize}

\subsection{Acknowledgements.}
The author would like to express his gratitude to his advisor Ian Petrow for helpful discussions and for carefully reading the present work. The author would also like to thank Yiannis Petridis for suggesting the problem to the author and for his guidance in the initial stages of the project.

\section{Reduction to congruence characters.}\label{sec:reductiontocongruencecharacters}

Let $\Gamma_1 \subset \Gamma_2 \subset \SL_2(\R)$ be cofinite discrete groups. It may be the case that $-I \notin \Gamma_2$. Note that $[\Gamma_2 : \Gamma_1] < \infty$ since both groups are cofinite. A class function $f$ on $\Gamma_1$ can be induced to a class function on $\Gamma_2$ by defining
\begin{equation*}
    \Ind_{\Gamma_1}^{\Gamma_2}(f)(\delta) := \sum_{\substack{\gamma \in \Gamma_1 \backslash \Gamma_2\\\gamma \delta \gamma^{-1} \in \Gamma_1}}f(\gamma \delta \gamma^{-1}).
\end{equation*}
The next lemma shows that the sums $\Psi(X; \cdot)$ introduced in \eqref{eq:chebyshevsumsgeneral} are compatible with induction of class functions.
\begin{lem}\label{lem:compatibilitywithinduction}
For all $X\geq 1$ we have 
\begin{equation}\label{eq:compatibilitywithinduction}
\Psi_{\Gamma_2}(X; \Ind_{\Gamma_1}^{\Gamma_2}(f)) = \Psi_{\Gamma_1}(X; f).
\end{equation}
\end{lem}
This lemma generalizes Theorem 2.2 from \cite{venkov}, where it was assumed that $-I \in \Gamma$. We recall the proof for the convenience of the reader. 
\begin{proof}
For brevity we write $g = \Ind_{\Gamma_1}^{\Gamma_2}(f)$ in this proof. For a hyperbolic element $P \in \Gamma_1$ we let $\nu_1(P)$ be the positive integer such that $P = P_0^{\nu_1(P)}$ for a primitive element $P_0$ of $\Gamma_1$. Let $C_{\Gamma_1}(P)$ be the centralizer of $P$ in $\Gamma_1$. Since we have 
\begin{equation*}
    C_{\Gamma_1}(P) = \begin{dcases}
        \{\pm P_0^n : n \in \Z\}, & \text{ if }-I \in \Gamma_1,\\
        \{P_0^n : n \in \Z\}, & \text{ otherwise},
    \end{dcases}
\end{equation*}
it follows that
\begin{equation*}
\nu_1(P) = \frac{[\{\pm I\}\Gamma_1: \Gamma_1]}{2} \times [C_{\Gamma_1}(P) : \langle P \rangle].
\end{equation*} 
Similarly, if $Q \in \Gamma_2$ is hyperbolic and we write $Q = Q_0^{\nu_2(Q)}$ for a primitive hyperbolic element $Q_0$ of $\Gamma_2$, then 
\begin{equation*}
    \nu_2(Q) = \frac{[\{\pm I\}\Gamma_2: \Gamma_2]}{2} \times [C_{\Gamma_2}(Q) : \langle Q \rangle].
\end{equation*}
In particular, any hyperbolic $P \in \Gamma_1$ can be considered an element of both $\Gamma_1$ and $\Gamma_2$, and we have 
\begin{equation}\label{eq:sec2:auxeq1}
    \nu_2(P) = \nu_1(P) \cdot [C_{\Gamma_2}(P): C_{\Gamma_1}(P)] \cdot \frac{[\{\pm I\}\Gamma_2: \Gamma_2]}{[\{\pm I\}\Gamma_1: \Gamma_1]}.
\end{equation}

The identity in \eqref{eq:compatibilitywithinduction} is equivalent to
\begin{equation}\label{eq:sec2:auxeq2}
    [\{\pm I\}\Gamma_2: \Gamma_2]\times \sum_{\substack{\text{hyp. }\{Q\}_{\Gamma_2}\\N(Q) = X\\
    \Tr(P) > 2}} \frac{g(Q)}{\nu_2(Q)} = [\{\pm I\}\Gamma_1: \Gamma_1]\times\sum_{\substack{\text{hyp. }\{P\}_{\Gamma_1}\\
    N(P) = X\\
    \Tr(P) > 2}} \frac{f(P)}{\nu_1(P)}
\end{equation}
for every $X > 1$. Using the definition of $g$ we find that
\begin{equation}\label{eq:sec2:auxeq3}
\begin{aligned}
     \sum_{\substack{\text{hyp. }\{Q\}_{\Gamma_2}\\N(Q) = X\\
    \Tr(Q) > 2}} \frac{g(Q)}{\nu_2(Q)} & =  \sum_{\substack{\text{hyp. }\{Q\}_{\Gamma_2}\\N(Q) = X\\
    \Tr(Q) > 2}} \frac{1}{\nu_2(Q)} \sum_{\substack{\gamma \in \Gamma_2 /\Gamma_1\\ \gamma^{-1}Q\gamma \in \Gamma_1}}f(\gamma^{-1}Q\gamma)\\
    & = \sum_{\substack{\text{hyp. }\{P\}_{\Gamma_1}\\N(P) = X\\
    \Tr(P) > 2}} f(P) \cdot \frac{\#\{\gamma \in \Gamma_1 \backslash \Gamma_2 \mid \gamma P \gamma^{-1} \in \Gamma_1 \text{ and } \{\gamma P \gamma^{-1}\}_{\Gamma_1} = \{P\}_{\Gamma_1}\}}{\nu_2(P)}.
\end{aligned}
\end{equation}
It is easy to check that the inclusion $C_{\Gamma_2}(P) \hookrightarrow \Gamma_2$ induces a bijection
\begin{equation*}
    C_{\Gamma_1}(P) \backslash C_{\Gamma_2}(P) \longleftrightarrow \{\gamma \in \Gamma_1 \backslash \Gamma_2 \mid \gamma P \gamma^{-1} \in \Gamma_1 \text{ and } \{\gamma P \gamma^{-1}\}_{\Gamma_1} = \{P\}_{\Gamma_1}\},
\end{equation*}
so that \eqref{eq:sec2:auxeq3} can be simplified to
\begin{equation*}
    \sum_{\substack{\text{hyp. }\{Q\}_{\Gamma_2}\\N(Q) = X\\
    \Tr(Q) > 2}} \frac{g(Q)}{\nu_2(Q)} =\sum_{\substack{\text{hyp. }\{P\}_{\Gamma_1}\\N(P) = X\\
    \Tr(P) > 2}} f(P) \cdot \frac{[C_{\Gamma_2}(P): C_{\Gamma_1}(P)]}{\nu_2(P)}. 
\end{equation*}
If we combine this identity with \eqref{eq:sec2:auxeq1} we deduce \eqref{eq:sec2:auxeq2}, finishing the proof of the lemma.
\end{proof}

Suppose now that $q$ is a positive integer, $\Gamma$ is a congruence subgroup of level $q$, and $f$ is a class function on $\widetilde{\Gamma} = \Gamma(q)\backslash \Gamma$. We recall Artin's theorem on induced characters in the next lemma, see \cite[Section 9.2]{serrelinearreps}.
\begin{lem}[Artin]\label{lem:Artinstheorem}
Let $G$ be a finite group. Every character in $G$ is a linear combination with rational coefficients of characters of representations induced from one-dimensional representations of cyclic subgroups of $G$.
\end{lem}
Since characters of representations of a finite group span the space of class functions we deduce the following: there exist subgroups $\widetilde{\Gamma_i} \subset \widetilde{\Gamma}$, characters $\chi_i$ on $\widetilde{\Gamma_i}$ and constants $c_i$ such that 
\begin{equation}
    f = \sum_{i} c_i \Ind_{\widetilde{\Gamma_i}}^{\widetilde{\Gamma}} \chi_i.
\end{equation}
By Frobenius reciprocity we have 
\begin{equation*}
    \langle f, \mathbbm{1}_{\widetilde{\Gamma}}\rangle_{\widetilde{\Gamma}} = \sum_{i}  c_i \delta_{\chi_i  = 1}.
\end{equation*}
Let $\Gamma_i$ be the pullback of $\widetilde{\Gamma_i}$ via the projection $\Gamma \rightarrow \widetilde{\Gamma}$. Using Lemma \ref{lem:compatibilitywithinduction} we deduce that
\begin{equation}\label{eq:sec2:auxeq4}
    \Psi_{\Gamma}(X; f) = \sum_{i} c_i \Psi_{\Gamma_i}(X; \chi_i).
\end{equation}
If for some $i$ we have $\{\pm I\}\nsubseteq \Gamma_i$, we can consider the multiplicative characters $\chi_i^+$ and $\chi_i^-$ defined on $\{\pm I\}\Gamma$ by
\begin{equation*}
    \chi_i^+(\pm \gamma) = \chi_i(\gamma), \quad \chi_i^-(\pm \gamma) = \pm \chi_i(\gamma),
\end{equation*}
where $\gamma \in \Gamma_i$ is arbitrary. It is straightforward to check that
\begin{equation*}
    \Ind_{\Gamma_i}^{\{\pm I\}\Gamma_i}(\chi_i) = \chi_i^+ + \chi_i^-.
\end{equation*}
Therefore, using Lemma \ref{lem:compatibilitywithinduction} again we see that
\begin{equation*}
    \Psi_{\Gamma_i}(X; \chi_i) = \Psi_{\{\pm I\}\Gamma_i}(X; \chi^+_i) + \Psi_{\{\pm I\}\Gamma_i}(X; \chi^-_i).
\end{equation*}
From this observation it follows that we can assume that $\{\pm I\} \subset \Gamma_i$ for all the groups $\Gamma_i$ appearing in \eqref{eq:sec2:auxeq4}. According to Definition \ref{defi:congruencepair}, each $(\Gamma_i, \chi_i)$ is a congruence character of level $q$.

Putting together the arguments in this paragraph we arrive at the following conclusion.

\begin{prop}\label{prop:conclusionreductioncharacters}
Proving Theorem \ref{thm:maintheorem1} is equivalent to proving that
\begin{equation}\label{eq:estimatetobeprovenmaintheorem}
    \Psi_\Gamma(X; \chi) = \delta_{\chi = 1} X + O_{q, \theta,\varepsilon}(X^{2/3 + \theta/6 + \varepsilon})
\end{equation}
holds for all congruence characters $(\Gamma, \chi)$. Proving Theorem \ref{thm:soundyoungshortintervals} is equivalent to proving that  
\begin{equation}\label{eq:estimatetobeprovenshortintervals}
    \Psi_\Gamma(X + u; \chi) - \Psi_\Gamma(X; \chi) = \delta_{\chi = 1} u + O_{q, \theta, \varepsilon}(u^{1/2}X^{1/4 + \theta/2 + \varepsilon})
\end{equation}
holds for all congruence characters $(\Gamma, \chi)$ of level $q$.
\end{prop}

\section{Spectral theory.}\label{sec:spectralbackground}
From now on $(\Gamma, \chi)$ denotes a congruence character and $k \in \{0, 1\}$ is such that $\chi(-I) = (-1)^k$. 

\subsection{Background.}
Refenrences for the material in this section are \cite[Section 4]{DFIsubconvexityartin} and \cite[Chapter 9]{hejhal2}.

For $\gamma \in \GL_2^+(\R)$ and $z \in \Hyp$ we let
\begin{equation}\label{eq:deficocyclej}
    j_\gamma(z) = \frac{c(\gamma)z + d(\gamma)}{|c(\gamma)z + d(\gamma)|}.
\end{equation}
The slash operator $|_k$ is defined for
 $\gamma \in \GL_2^+(\R)$ and $h: \Hyp \rightarrow \CC$ by
\begin{equation}\label{eq:slashoperator}
    (h|_k \gamma)(z) := h(\gamma z) \left(\frac{c(\gamma)z + d(\gamma)}{|c(\gamma)z + d(\gamma)|}\right)^{-k} = h(\gamma z)j_\gamma(z)^{-k}.
\end{equation}
This formula provides a right action of $\GL_2^+(\R)$ on functions on $\Hyp$. The Laplacian of weight $k$ is denoted by $P$, so that in cartesian coordinates we have
\begin{equation}\label{eq:laplacianoperatorweightk}
    P := -y^2(\partial_x^2 + \partial_y^2) + iky\partial_x.
\end{equation}
Let $d\mu(z)$ denote the usual volume form on $\Hyp$, which in Cartesian coordinates is expressed as
\begin{equation*}
        d\mu(z) = \frac{dx\, dy}{y^2}.
\end{equation*}
We let $\vol(\Gamma\backslash \Hyp)$ denote the volume of a fundamental domain of $\Gamma$ with respect to this measure. We denote by $L^2(\Gamma, \chi)$ the Hilbert space of functions $h: \Hyp \rightarrow \CC$, 
which obey the transformation property
\begin{equation*}
    h|_k \gamma = \chi(\gamma)h, \text{ for all }\gamma \in \Gamma 
\end{equation*}
and are square integrable on $\Gamma\backslash \Hyp$, that is
\begin{equation*}
    \frac{1}{\vol(\Gamma\backslash \Hyp)}\int_{\Gamma\backslash \Hyp}|h(z)|^2 d\mu(z)
< \infty.
\end{equation*}
The inner product is computed with respect to probability measure, so that
\begin{equation}\label{eq:innerproductmaasforms}
    \langle h_1,h_2 \rangle := \frac{1}{\vol(\Gamma\backslash \Hyp)}\int_{\Gamma \backslash \Hyp} h_1(z)\overline{h_2(z)} d\mu(z).
\end{equation}
Note that this normalization differs from that in other sources. 
\begin{defi}
Given two congruence characters, we write $(\Gamma_1, \chi_1)\subset (\Gamma_2, \chi_2)$ when $\Gamma_1 \subset \Gamma_2$ and $\chi_2 \mid_{\Gamma_1} = \chi_1$.
\end{defi}
The normalization \eqref{eq:innerproductmaasforms} is convenient because whenever we have an inclusion of congruence characters $(\Gamma', \chi')\subset (\Gamma, \chi)$, then the natural inclusion 
\begin{equation*}
    L^2(\Gamma, \chi) \hookrightarrow L^2(\Gamma', \chi')
\end{equation*}
is an isometry.

\subsubsection{Fourier expansion.}
The $\Gamma$-stabilizer of the cusp at $\infty$ is of the form $\Gamma_\infty = \{\pm \begin{psmallmatrix}
    1 & x\\
    0 & 1 \end{psmallmatrix}: x \in \ell \Z\}$ for some $\ell = \ell(\Gamma) \geq 1$ which divides $q$. Since $\chi$ is trivial on $\Gamma(q)$ we have 
    \begin{equation}\label{eq:defialpha}
        \chi(\begin{psmallmatrix}
            1 & \ell\\
            0 & 1
        \end{psmallmatrix}) = e(-\alpha),
    \end{equation}
for a unique $\alpha \in \{0, \ell/q, 2\ell/q \ldots, 1 - \ell/q\}$. Suppose that $u \in C^2(\Hyp)$ has the following three properties:
\begin{itemize}
    \item $u|_k \gamma = \chi(\gamma)u$ for $\gamma \in \Gamma$.
    \item $Pu = (1/4 + t^2)u$ for some $t$. We can and will assume that, either $\text{Re}(t) > 0$ or else $\text{Im}(t) \geq 0$.
    \item $|u(z)| \leq y^A  + y^{-B}$ for all $z = x + iy \in \Hyp$ and some $A, B \geq 0$.
\end{itemize}
Then we have a Fourier expansion
\begin{equation}\label{eq:generalfourierexpansion}
    u(z) = \rho(u; 0^+)\frac{y^{1/2 + it} + y^{1/2 - it}}{2} + \rho(u; 0^-)\frac{y^{1/2 + it} - y^{1/2 - it}}{2it} + \sum_{n \in \Q - \{0\}} \frac{\rho(u; n)}{|n|^{1/2}}W_{\frac{k\sgn(n)}{2}, it}(4\pi |n|y)e(nx),
\end{equation}
where the function $W_{\alpha, \beta}(z)$ is the Whittaker function, as in \cite[Equation (4.19)]{DFIsubconvexityartin}, and where $(y^{1/2 + it} - y^{1/2 - it})/(2it)$ is interpreted as $y^{1/2}\log y$ when $t = 0$. The Fourier expansion \eqref{eq:generalfourierexpansion} has the following properties:
\begin{itemize}
    \item It converges absolutely, uniformly on compacta  of $\Hyp$.
    \item $\rho(u; n) = 0$ unless $n \in \frac{1}{\ell}\Z - \frac{\alpha}{\ell}$.
\end{itemize}
Note that the Fourier coefficients $\rho(u; n)$ do not depend on the congruence character $(\Gamma, \chi)$, only on the function $u$.

\subsubsection{Spectral decomposition.}
The Hilbert space $L^2(\Gamma, \chi)$ has an orthogonal decomposition
\begin{equation}\label{eq:orthogonaldecompositionspectrum}
    L^2(\Gamma, \chi) = L_0^2(\Gamma, \chi) \oplus L^2_\text{res}(\Gamma, \chi)\oplus L^2_{\text{cont}}(\Gamma, \chi),
\end{equation}
which arises from the spectral resolution of the operator $P$. The spaces $L_0^2$, $L^2_\text{res}$ and $L^2_\text{cont}$ are called, respectively, the cuspidal, the residual and the continuous parts of the spectrum of $P$. We will recall some information about each of these spaces.

\paragraph{The residual part.}
If $\chi$ is the trivial character, then $L^2_\text{res}(\Gamma, \chi)$ is one-dimensional space of constant functions. Otherwise, if $\chi$ is nontrivial, then $L^2_\text{res}(\Gamma, \chi) = \{0\}$ is the trivial zero-dimensional space. Note that $k = 1$ implies that $\chi$ is nontrivial. These facts depend on the fact that $(\Gamma, \chi)$ is a congruence character.

\paragraph{The continuous part.}
The continuous part of the spectrum is spanned by the Eisenstein series $z \mapsto E_\a(z; 1/2 + it)$ for $t \in \R$. We now recall how these are defined. 

Note that, since $\Gamma$ is a congruence subgroup of $\SL_2(\Z)$,  the cusps of $\Gamma$ are the points of $\P^1(\Q) = \Q \cup \{\infty\}$. We let $[\a]_\Gamma$ denote the $\Gamma$-orbit of cusps containing $\a$, that is 
\begin{equation*}
    [\a]_\Gamma := \{\b \in \P^1(\Q): \b = \gamma \a \text{ for some }\gamma \in \Gamma\}.
\end{equation*}
There are only finitely-many different $\Gamma$-orbits. For a cusp $\a$ of $\Gamma$ we let $\Gamma_\a := \{\gamma \in \Gamma: \gamma \a = \a\}$ be its stabilizer. This group is of the form $\Gamma_\a = \{\pm T_\a^n : n \in \Z\}$, where $T_\a$ is a parabolic element with positive trace. We let $\sigma_\a \in \SL_2(\R)$ be any matrix such that $\sigma_\a \infty = \a$ and
\begin{equation*}
    \sigma^{-1}_\a T_\a \sigma_\a = \begin{pmatrix}
        1 & 1\\
        0 & 1
    \end{pmatrix}.
\end{equation*}
We define $\alpha_\a \in [0, 1)$ by the equation
\begin{equation*}
    \chi(T_\a) = e(-\alpha_\a),
\end{equation*}
in particular $\alpha = \alpha_\infty$. We say that $\a$ is a \emph{singular cusp} of $(\Gamma, \chi)$ when $\chi(T_\a) = 1$. We let $\kappa$ denote the number of different $\Gamma$-orbits of singular cusps of $(\Gamma, \chi)$. We choose representatives $\a_1, \ldots, \a_{\kappa}$ of these orbits with the condition that, if $\infty$ is singular, then it is chosen as a representative of its class. Let $f_s: \Hyp \rightarrow \CC$ denote the function $f_s(z) = \text{Im}(z)^s$. Then, to a singular cusp we can attach the Eisenstein series $E_\a(z; s)$, depending on $(z, s)\in \Hyp \times \CC$, defined for $\text{Re}(s) > 1$ by the absolutely convergent series 
\begin{equation}\label{eq:defofeisensteinseries}
    E_\a(z; s) = \sum_{\gamma \in \Gamma_\a \backslash \Gamma} \overline{\chi(\gamma)}\left[f_s|_k (\sigma_\a^{-1}\gamma)\right](z).
\end{equation}
By construction, it follows that the identity 
\begin{equation*}
    E_\a(z, s)|_k \gamma = \chi(\gamma)E_\a(z, s)
\end{equation*}
holds for any $\gamma \in \Gamma$. Also, we note that the definition of $E_\a(z, s)$ does not depend on the choice of $\sigma_\a$. On the other hand, if $\b = \gamma \a$ for $\gamma \in \Gamma$, then we have 
\begin{equation}\label{eq:transformationdifferentrepresentativeeisenstein}
    E_\b(z;s) = \overline{\chi(\gamma)}E_\a(z; s),
\end{equation}
so that $E_\a(z;s)$ depends very mildly on the representative of the equivalence class $[\a]_\Gamma$. We will use the more complete notation $E_\a(z; s; \Gamma, \chi)$ when we want to emphasize the dependence on $(\Gamma, \chi)$. The Eisenstein series can be continued meromorphically to $\CC$. On the region $\text{Re}(s) \geq 1/2$ the only possible singularity can occur at $s = 1$ (we use here that $(\Gamma, \chi)$ is a congruence character). When $\chi$ is the trivial character there is a simple pole at $s = 1$ with residue
\begin{equation*}
    \text{Res}_{s = 1}E_\a(z; s) = \frac{1}{\vol(\Gamma \backslash \Hyp)}.
\end{equation*}
If $\chi$ is nontrivial, then $E_\a(z; \cdot)$ is holomorphic at $s = 1$. The Eisenstein series $E_\a(z; s)$ has a Fourier expansion at $\infty$ as in \eqref{eq:generalfourierexpansion}. For simplicity we will write $\rho_{\a, t}(n)$ instead of $\rho(E_\a(z; 1/2 + it); n)$.

We can arrange the Eisenstein series in a column vector $\E(z; s) := (E_{\a_i}(z; s))_i$ of size $\kappa\times 1$. Then, there exists a square matrix $\Phi(s)$ of size $\kappa \times \kappa$ such that the identity
\begin{equation}\label{eq:scatteringmatrixfunctionalequation}
    \E(z;s) = \Phi(s)\E(z; 1-s)
\end{equation}
holds away from possible poles. The matrix $\Phi(s)$ is called the \emph{scattering matrix} of $(\Gamma, \chi)$, and its entries can be expressed in terms of the constant term of the Fourier expansions of the different Eisenstein series at the different singular cusps. We know that $\Phi(s)$ is unitary on the critical line $\text{Re}(s) = 1/2$. The \emph{scattering determinant} is the meromorphic function
\begin{equation}\label{eq:scatteringdeterminant}
    \phi(s) := \det \Phi(s),
\end{equation}
which is nonzero and holomorphic on a neighbourhood of the line $\text{Re}(s) = 1/2$. The logarithmic derivative of this function expresses the contribution of the continuous part of the spectrum to the Selberg trace formula, see Lemma \ref{lem:selbergtraceformula} further down. Thus, for analytic applications it is important to control the size of $\phi'/\phi$. It turns out that for a congruence character $(\Gamma, \chi)$ the scattering determinant can be expressed explicitly in terms of completed Dirichlet $L$-functions. Recall that, if $\psi$ is a primitive Dirichlet character such that $\psi(-1) = (-1)^k$, the completed Dirichlet $L$-function of $\psi$ is given by   
\begin{equation*}
    \Lambda(s; \psi) := \pi^{-\frac{s + k}{2}}\Gamma\left(\frac{s + k}{2}\right)\sum_{n = 1}^\infty \frac{\psi(n)}{n^s}
\end{equation*}
when $\text{Re}(s) > 1$, and otherwise is defined by analytic continuation. 
\begin{lem}\label{lem:basicformscattering}
For each $i \in \{1, \ldots, \kappa\}$ there is a primitive Dirichlet character $\psi_i$ with $\psi_i(-1) = (-1)^k$, a finite set of primes $S_i$, and for each $p \in S_i$ a rational function $H_{i, p}(X)$ such that 
\begin{equation*}
    \phi(s) = \prod_{i = 1}^\kappa\left(\frac{\Lambda(2s-1; \psi_i)}{\Lambda(2s; \psi_i)} \prod_{p \in S_i} H_{i, p}(p^{-s})\right).
\end{equation*}
\end{lem}
\begin{proof}
The formula follows from the discussion in \cite[Paragraph 2.7]{reznikoveisensteinmatrix}.
\end{proof}

From Lemma \ref{lem:basicformscattering} it follows easily that
\begin{equation}\label{eq:meanvalueestimatescatteringdeterminant}
    \int_{|t-T| \leq 1} \left|\frac{\phi'(1/2 + it)}{\phi(1/2 + it)}\right| \, dt \ll_{q} \log T.
\end{equation}

\paragraph{The cuspidal part.}
We know that $L_0^2(\Gamma, \chi)$ decomposes discretely as an orthonormal sum
\begin{equation}\label{eq:basicorthonormalbasis}
    L_0^2(\Gamma, \chi) = \overline{\bigoplus_{j\geq 1}} \CC u_j,
\end{equation}
where $P u_j = \lambda_j u_j$ for some $\lambda_j > 0$ and where each eigenvalue $\lambda_j$ appears with finite multiplicity. As is customary, we write $\lambda_j = 1/4 + t_j^2$. When $\lambda_j \neq 1/4$ we need to make a choice of sign for $t_j$. We will assume that, either $\text{Im}(t_j) > 0$, or else $t_j > 0$.
We recall some facts about these eigenvalues.
\begin{enumerate}
    \item [1)] If $k = 0$, then since $\lambda_j > 0$ we have that $t_j$ is either real or purely imaginary with $|\text{Im}(t_j)|< 1/2$. The Selberg eigenvalue conjecture predicts that $t_j$ is always real. The best progress towards this conjecture is that $|\text{Im}(t_j)| \leq 7/64$, or, equivalently, that $\lambda_j \geq 1/4 - (7/64)^2$. The proof can be found in the second appendix to \cite{kimramakrishnansarnak2003}.
    \item [2)] If $k = 1$, then we know that $\lambda_j \geq 1/4$ \cite[Equation (4.40)]{DFIsubconvexityartin}, so we always have $t_j \in \R$. When $t_j = 0$ the function $y^{-1/2}u_j(z)$ is an holomorphic cusp form of weight $1$.
\end{enumerate}
Quantitatively, we need the following coarse version of Weyl's law
\begin{equation}\label{eq:basicweyllaw}
    \#\{t_j \in [T, T + 1]\} \asymp_q  T,
\end{equation}
which follows from Theorem 2.28 in \cite{hejhal2} and Equation \eqref{eq:meanvalueestimatescatteringdeterminant} above. At the cusp $\infty$ the cuspidal eigenfunction $u_j$ has a Fourier expansion as in \eqref{eq:generalfourierexpansion}, with the particularity that $\rho(u_j; 0^+) = \rho(u_j; 0^-) = 0$. If the choice of orthonormal basis $\{u_j\}_{j\geq 1}$ is understood from the context we will write $\rho_{j}(n)$ instead of $\rho(u_j; n)$.

\subsection{Smoothing the exponential sums.}\label{subsec:smoothingtheexponentialsum}
The exponential sums $S(T, X)$ with sharp cut-off are not directly accesible to trace formulae, since test functions on the spectral side are always smooth. To overcome this difficulty, Iwaniec introduced the smooth sums
\begin{equation}\label{eq:smoothversionofS}
    \S_0(T, X) := \sum_{t_j \geq 0} X^{it_j}\exp(-t_j/T). 
\end{equation}
Following the appendix to \cite{luosarnak}, we can write
\begin{equation*}
    \S(T, X) = \int_{-1}^1 \widehat{g}(\xi) \S_0(T, X\exp(-2\pi \xi)) \, d\xi + O(T \log T),
\end{equation*}
for a certain function $\widehat{g}(\xi)$ which satisfies $\int_{-1}^1 |\widehat{g}(\xi)| \, d\xi = O(\log T)$. In this argument the only property of the sequence $t_j$ that is used is the weak version of Weyl's law in \eqref{eq:basicweyllaw}. 

\begin{prop}\label{prop:conclusionsmoothingsum}
It is enough to prove the bounds of Theorem \ref{thm:luosarnakbound} and Theorem \ref{thm:balkanovafrolenkovbound} for $\S_0(T, X)$ instead of $\S(T, X)$. 
\end{prop}

\subsection{Selberg trace formula.}\label{subsec:selbergtraceformula}
We say that a function $h$, defined on an horizontal strip $|\text{Im}(z)|\leq 1/2 +\delta$ for some $\delta > 0$, is \emph{admissible} if it is even, holomorphic on the strip and for some $C > 0$ satisfies the bound $|h(z)| \leq C(1 + |z|)^{-2 - \delta}$ there.
Its inverse Fourier transform is denoted by $g$, so that 
\begin{equation}\label{eq:pairtestfunctionselbergtraceformula}
    g(r) = \frac{1}{2\pi}\int_\R e^{irt}h(t) \, dt, \quad \text{ and } \quad h(t) = \int_\R e^{irt} g(r) \, dr.
\end{equation}

Given $\theta \in \R$ we define $k(\theta) := \begin{psmallmatrix}
    \cos \theta & \sin \theta\\
    -\sin \theta & \cos \theta.
    \end{psmallmatrix}$. 
Given an elliptic element $R \in \SL_2(\R)$, there is a unique $\theta(R)$ satisfying $0 < |\theta(R)| < \pi$ and such that $R$ is conjugate to $k(\theta(R))$ in $\SL_2(\R)$. If $R \in \Gamma$ is elliptic, we let $m(R)$ denote the cardinality of the centralizer of $R$ in $\Gamma$. Note that $m(R) < \infty$. 

\begin{lem}[Selberg, Hejhal]\label{lem:selbergtraceformula}
Let $(\Gamma, \chi)$ be a congruence character, $k \in \{0, 1\}$ and $h$ an admissible function. Then the following equality holds
\begin{equation*}
\begin{aligned}
\Lambda_\text{Res}(h) + \Lambda_\text{Cusp}(h) & = \Lambda_\text{Id}(h) + \Lambda_{\text{Hyp}}(h) + \Lambda_\text{Ell}(h) + \Lambda_\text{Par-Cont}(h).
\end{aligned}
\end{equation*}
where each $h \mapsto \Lambda_*(h)$ is a distribution on admissible functions given as follows.\newline 
\begin{equation*}
    \Lambda_\text{Res}(h) = \delta_{\chi = 1}h(i/2), \quad      \Lambda_\text{Cusp}(h) = \sum_j h(t_j),
\end{equation*}
\begin{equation*}
    \Lambda_\text{Id}(h) = \begin{dcases}
    \frac{\vol(\Gamma\backslash \Hyp)}{4\pi}
        \int_\R h(r) r \tanh \pi r\, dr,  &\text{ if }k = 0,\\
    \frac{\vol(\Gamma \backslash \Hyp)}{4\pi}\int_\R h(r) r\coth \pi r\, dr, & \text{ if }k = 1.
    \end{dcases}
\end{equation*}
\begin{equation*}
    \Lambda_\text{Hyp}(h) = \sum_{\text{hyp.} \{P\}_\Gamma}\frac{\chi(P) \log N(P_0)}{N(P)^{1/2} - N(P)^{-1/2}} g(\log N(P)),
\end{equation*}
\begin{equation*}
    \Lambda_\text{Ell}(h) = \begin{dcases}
        \sum_{\substack{\text{ell.}\{R\}_\Gamma\\0 < \theta(R)< \pi}} \frac{i \chi(R)}{m(R) \sin \theta(R) }\int_\R \frac{g(u)}{e^{u/2 + i\theta(R)} - e^{-u/2 - i\theta(R)}}\, du, & \text{ if }k = 0,\\
        \sum_{\substack{\text{ell.}\{R\}_\Gamma\\0 < \theta(R)< \pi}} \frac{i\chi(R)}{2m(R) \sin \theta(R) }\int_\R g(u)\left(1 - \frac{i \sin 2\theta(R)}{\cosh u - \cos 2\theta(R)}\right)\, du, & \text{ if }k = 1.
    \end{dcases}
\end{equation*}
When $k = 0$ we have
\begin{equation*}
\begin{aligned}
    \Lambda_\text{Par-Cont}(h) & = -g(0) \sum_{\text{nonsing. } \a}\log|1 - e(\alpha_\a)|  + \frac{1}{4}h(0) \Tr(I - \Phi(1/2))\\
    & + \frac{1}{4\pi}\int_\R h(t) \frac{\phi'}{\phi}(1/2 + it)\, dt -\kappa\left( g(0) \log 2  + \frac{1}{2\pi}\int_\R h(r) \frac{\Gamma'}{\Gamma}(1 + ir)\, dr\right),
\end{aligned}
\end{equation*}
and when $k = 1$ we have  
\begin{equation*}
\begin{aligned}
    \Lambda_\text{Par-Cont}(h) & = -g(0) \sum_{\text{nonsing. } \a}\log|1 - e(\alpha_\a)| - \frac{h(0)}{2}\sum_{\text{nonsing. }\a}(1/2 - \alpha_\a) + \frac{h(0)}{4}\Tr(I - \Phi(1/2))\\
    & + \frac{1}{4\pi}\int_\R h(t) \frac{\phi'}{\phi}(1/2 + it) \, dt - \kappa\left(\frac{1}{2}\int_0^\infty g(u)\tanh(u/4)\, du + g(0)\log 2 + \frac{1}{2\pi}\int_\R h(r) \frac{\Gamma'}{\Gamma}(1 + ir)\, dr\right).
\end{aligned}
\end{equation*}
All sums and integrals converge absolutely when $h$ is admissible.
\end{lem}
\begin{proof}
After a few simple manipulations, the formula is a particular case of Theorem 6.3 in chapter 9 of \cite{hejhal2}.
\end{proof}
Note that $\Lambda_{\text{Res}}(h) = 0$ when $k = 1$, since $\chi \neq 1$ in that case.

\subsection{Kloosterman sums.}
Recall that $\Gamma_\infty = \{\pm \begin{psmallmatrix}
    1 & x\\
    0 & 1 \end{psmallmatrix}: x \in \ell\Z\}$, with $\chi(\begin{psmallmatrix}1 & \ell\\
        0 & 1
    \end{psmallmatrix}) = e(-\alpha)$, where $\alpha \in \{0, \ell/q, 2\ell/q \ldots, 1 - \ell/q\}$. For any $n,m \in \Q$ and $c \in \Z_{\geq 1}$, we define the Kloosterman sums
    \begin{equation}\label{eq:generalkloostermansum}
        S_{(\Gamma, \chi)}(m, n; c) := \begin{dcases}
        \sum_{\substack{[\gamma] \in \Gamma_\infty \backslash \Gamma/\Gamma_\infty\\
        c(\gamma) = c}}\overline{\chi(\gamma)}e\left(\frac{m a(\gamma) + n d(\gamma)}{c}\right), & \text{ if }m, n \in \frac{1}{\ell}\Z - \frac{\alpha}{\ell},\\
        0, & \text{ otherwise.}
        \end{dcases}
    \end{equation}
The trivial bound is $|S_{(\Gamma, \chi)}(m, n; c)| = O_q(c)$, but thanks to Weil's bound on classical Kloosterman sums we can do much better. 

\begin{lem}\label{lem:weilbound}
If $m, n \neq 0$, then for any $\varepsilon > 0$ we have
\begin{equation*}
 S_{(\Gamma, \chi)}(m, n; c) \ll_{q, \varepsilon} \gcd(qm, qn, c)^{1/2}c^{1/2 + \varepsilon}.
\end{equation*}
\end{lem}

\begin{proof}
Of course, we can assume that $m, n \in \ell^{-1}\Z - \alpha/\ell$. We let $m' := qm, n' := qn$, which are nonzero integers. After a small manipulation we find that 
\begin{equation}\label{auxeq:kloosterman:1}
    S_{(\Gamma, \chi)}(m, n; c)  = \frac{\ell^2}{q^2}\sum_{\tilde{\gamma} \in \tilde{\Gamma}}\overline{\chi(\tilde{\gamma})}\sum_{\substack{a, d \text{ mod }cq\\
    ad = 1 \modulo c\\
    \begin{psmallmatrix}
        a & (ad-1)/c\\
        c & d
    \end{psmallmatrix} = \tilde{\gamma} \text{ mod }q}} e\left(\frac{m'a + n'd}{cq}\right).
\end{equation}
We will estimate each inner sum separately. We may assume that $c = c(\widetilde{\gamma}) \modulo q$, since otherwise the inner sum vanishes. We can write $q = \prod_{i= 1}^t p_i^{r_i}$, where $t\geq 1$, and where each $p_i$ is prime and each is $r_i$ positive for $1 \leq i \leq t$. Write also $c = c_0 \prod_{i = 1}^t p_i^{s_i - r_i}$, where $(c_0, q) = 1$ and the product ranges over the primes dividing $q$. We have
$cq = c_0 \prod_i p_i^{s_i}$. We define
\begin{equation*}
    c_0' = c_0, \quad c'_i  = p_i^{s_i} \text{ for }1\leq i \leq t,
\end{equation*}
so that 
\begin{equation*}
    cq = \prod_{i=0}^t c_i', \quad (c_i', c_j') = 1 \text{ for } i\not = j.
\end{equation*}
We define also 
\begin{equation*}
    R_i = \prod_{j \neq i} c_j',
\end{equation*}
and let $\overline{R_i}$ be any integer which is the inverse of $R_i$ when considered modulo $c_i'$. From the Chinese remainder theorem we see that each inner sum in \eqref{auxeq:kloosterman:1} factors as 
\begin{equation*}
    \sum_{\substack{a, d \text{ mod }cq\\
    ad = 1 \modulo c\\
    \begin{psmallmatrix}
        a & (ad-1)/c\\
        c & d
    \end{psmallmatrix} = \tilde{\gamma} \text{ mod }q}} e\left(\frac{m'a + n'd}{cq}\right) = S^{\tilde{\gamma}}_0(m'\overline{R_0}, n' \overline{R_0}; c_0) \times \prod_{i = 1}^t S^{\tilde{\gamma}}_i(m'\overline{R_i}, n' \overline{R_i}; c; p_i^{s_i}),
\end{equation*}
where 
\begin{equation*}
S^{\tilde{\gamma}}_0(m'\overline{R_0}, n' \overline{R_0}; c_0) = \sum_{\substack{a, d \text{ mod }c_0\\
ad = 1 \text{ mod }c_0}} e\left(\frac{m'\overline{R_i}a + n'\overline{R_i}d}{c_0}\right)
\end{equation*} 
is the classical Kloosterman sum and for $1\leq i \leq t$ we have
\begin{equation*}
S^{\tilde{\gamma}}_i(m'\overline{R_i}, n' \overline{R_i}; c; p_i^{s_i}) = \sum_{\substack{a, d \text{ mod }p_i^{s_i}\\
(a, d) = (a(\tilde{\gamma}), d(\tilde{\gamma})) \text{ mod }p_i^{r_i}\\
ad = 1 + b(\tilde{\gamma})c \text{ mod }p_i^{s_i}}} e\left(\frac{m' \overline{R_i}a + n' \overline{R_i}d}{p_i^{s_i}}\right).
\end{equation*}
The Weil bound gives
\begin{equation*}
S^{\tilde{\gamma}}_0(m'\overline{R_0}, n' \overline{R_0}; c_0)     \ll_\varepsilon (m', n', c_0)^{1/2}c_0^{1/2 + \varepsilon}.
\end{equation*}
For $1 \leq i \leq t$ we write $(m', n', p_i^{s_i}) = p_i^{u_i}$ as well as $m' = m_i' p_i^{u_i}$, $n' = n_i' p_i^{u_i}$. If $s_i < 2r_i + u_i$, we can estimate the sum $S^{\tilde{\gamma}}_i(m'\overline{R_i}, n' \overline{R_i}; c; p_i^{s_i})$ trivially (our bounds are allowed to depend on $q$). Thus, we can assume that $s_i \geq 2r_i + u_i$ for the remainder of this proof. Under this assumption we have $p_i \mid c$, so that $w := 1 + b(\widetilde{\gamma})c$ is a unit modulo $p_i^{s_i}$, and both $a(\widetilde{\gamma})$ and $d(\widetilde{\gamma})$ are units modulo $p_i^{r_i}$. With this notation we can rewrite 
\begin{equation*}
    S^{\tilde{\gamma}}_i(m'\overline{R_i}, n' \overline{R_i}; c; p_i^{s_i}) = \sum_{\substack{a \modulo p_i^{s_i}\\
    a = a(\widetilde{\gamma}) \modulo p_i^{r_i}}} e\left(\frac{m_i' \overline{R_i} a + wn_i' \overline{R_i} \overline{a}}{p_i^{s_i - u_i}}\right) = p_i^{u_i} \sum_{\substack{a \modulo p_i^{s_i - u_i}\\
    a = a(\widetilde{\gamma}) \modulo p_i^{r_i}}} e\left(\frac{m_i' \overline{R_i} a + wn_i' \overline{R_i} \overline{a}}{p_i^{s_i - u_i}}\right).
\end{equation*}
If $v_i = \lceil (s_i - u_i)/2\rceil$ and $a_1 = a(1 + xp^{v_i})$, then its inverse modulo $p_i^{s_i - u_i}$ is given by $\overline{a_1} = \overline{a}(1 - xp^{v_i})$. It follows that 
\begin{equation}
\begin{aligned}
    S^{\tilde{\gamma}}_i(m'\overline{R_i}, n' \overline{R_i}; c; p_i^{s_i}) & = p^{u_i} \sum_{\substack{a \modulo p_i^{v_i}\\
    a = a(\widetilde{\gamma}) \modulo p_i^{r_i}}} e\left(\frac{m_i' \overline{R_i} a + wn_i' \overline{R_i} \overline{a}}{p_i^{s_i - u_i}}\right)\sum_{x \modulo p^{s_i - u_i - v_i}} e\left(\frac{x(m_i'\overline{R_i}a - wn_i'\overline{R_i}\overline{a})}{p_i^{s_i - u_i - v_i}}\right)\\
    & \leq p^{s_i - v_i} \#\{a \modulo p_i^{v_i} : a^2 m_i' = wn_i' \modulo p^{s_i - u_i - v_i}\}\\
    & \ll p_i^{v_i + u_i},
\end{aligned}
\end{equation}
where we have used Hensel's lemma in the last step. Putting together the bounds for $i = 0$ and $1 \leq i \leq t$ completes the proof of the lemma.
\end{proof}

\subsection{Preliminary version of the Bruggeman--Kuznetsov trace formula.}

Fix an orthonormal basis $\{u_j\}_{j \geq 1}$ of $L_0^2(\Gamma, \chi)$ as in the discussion after \eqref{eq:basicorthonormalbasis}, so that $P u_j = (1/4 + t_j^2)u_j$.  Fix also representatives $\a_1, \ldots, \a_\kappa$ of the equivalence classes of singular cusps of $(\Gamma, \chi)$. Recall that $\ell \geq 1$ is the width of the cusp $\infty$, so that $\Gamma_\infty = \{\pm \begin{psmallmatrix}
1 & x\\
0 & 1
\end{psmallmatrix} : x \in l\Z\}$: Recall also the Fourier coefficients introduced in \eqref{eq:generalfourierexpansion}. Note that we write $\rho_j(n)$ instead of $\rho(u_j; n)$, as well as $\rho_{\a, t}(n)$ instead of $\rho(E_\a(\cdot; 1/2 + it); n)$. 

The following preliminary version of the Bruggeman--Kuznetsov is already very convenient for some applications, like upper bounding the Fourier coefficients on average.
\begin{lem}\label{lem:preliminarykuznetsov}
For positive $m,n \in \frac{1}{\ell}\Z - \frac{\alpha}{\ell}$ and for any real number $r$ we have 
\begin{equation}\label{eq:prekuznetsov}
\begin{aligned}
    & \frac{\ell}{\vol(\Gamma\backslash \Hyp)}\sum_{j \geq 1} \frac{\overline{\rho_j(m)} \rho_j(n)}{\cosh \pi (r - t_j) \cosh \pi (r + t_j)} + \sum_{i = 1}^{\kappa}\frac{\ell}{4\pi}\int_\R \frac{\overline{\rho_{\a_i, t}(m)}\rho_{\a_i, t}(n)\, dt}{\cosh \pi (r - t)\cosh \pi (r + t)}\\
    & = \frac{|\Gamma\left(1 - \frac{k}{2} - ir\right)|^2}{4\pi^3}\left(\delta_{m =  n}+ \sum_{c = 1}^\infty \frac{S_{(\Gamma, \chi)}(m, n; c)}{c \ell }I\left(\frac{4\pi \sqrt{mn}}{c}; r\right)\right).
\end{aligned}
\end{equation}
In this sum we have
\begin{equation*}
    I(x; r) = -2x \int_{-i}^{i}(-i\zeta)^{k-1}K_{2ir}(\zeta x)\, d\zeta,
\end{equation*}
where the contour is a semicircle contained in the right half-plane. 
\end{lem}
\begin{proof}
The same proof as that of \cite[Proposition 5.2]{DFIsubconvexityartin} is valid in our context. We can also quote \cite[Lemma 3]{proskurinarbitraryweightkuznetsov}, specialized at $\sigma = 1$, which deals with the case of arbitrary real weight. Note that our normalization of the inner product of $L^2(\Gamma, \chi)$ as well as that of the Fourier coefficients differs from these two references.
\end{proof}

It is convenient to introduce the spectral weights
\begin{equation}\label{eq:weightskuznetsov}
    \omega(t) := \begin{dcases}
        \frac{\pi}{\cosh \pi t}, & \text{ if }k = 0,\\
        \frac{\pi t}{\sinh \pi t}, & \text{ if }k = 1.
    \end{dcases}
\end{equation}
Note that for $|\text{Im}(t)| \leq 7/64$ we have
\begin{equation*}
    |\omega(t)| \asymp (1 + |t|)^k \exp(-|t|/T).
\end{equation*}
We can use Lemma \ref{lem:preliminarykuznetsov} to upper bound averages of Fourier coefficients.

\begin{lem}\label{lem:upperboundfouriercoefficients}
For positive $m,n \in \frac{1}{\ell}\Z - \frac{\alpha}{\ell}$ and for $T \geq 1$ we have the following upper bound:
\begin{equation*}
    \sum_{|t_j - T|\leq 1}\omega(t_j)|\rho_j(n)\rho_j(m)| + \sum_{i = 1}^\kappa\int_{||t|-T|\leq 1}\omega(t) |\rho_{\a_i, t}(n, t)\rho_{\a_i, t}(m)| \, dt \ll T.
\end{equation*}
\end{lem}

\begin{proof}
Without loss of generality we can assume that $T \geq 10$, and by the Cauchy-Schwarz inequality we can assume that $n = m$, so that the summands and integrands on the left hand side are nonnegative. We are going to integrate equality \eqref{eq:prekuznetsov}, multiplied by $r^k e^{\pi r}$, over $r \in [T-2, T + 2]$.
Note that 
\begin{equation}\label{eq:auxiliarylowerboundweights}
    \int_{T-2}^{T + 2}\frac{r^k e^{\pi r}\, dr}{\cosh \pi (r - t)\cosh\pi (r + t)} \gg T^k e^{-\pi T}
\end{equation}
whenever $||t|-T| \leq 1$. Regarding the special function $I(x; r)$, we can use the trivial bound 
\begin{equation*}
    I(x; r) \ll x \int_{-i}^{i}K_0(x\text{Re}(\zeta))\, d\zeta \ll x |\log x|
\end{equation*}
where we used that $K_0(x) \ll |\log x|$ as $x \to 0^+$, see \cite[\href{https://dlmf.nist.gov/10.31.E2}{(10.31.2)}]{dlmf}. Using Lemma \ref{lem:weilbound} we get 
\begin{equation*}
\sum_{c = 1}^\infty \frac{S_{(\Gamma, \chi)}(m, n; c)}{c\ell}I\left(\frac{4\pi \sqrt{m n}}{c}; r\right) \ll_{q, m, n} 1.    
\end{equation*}
Thus, after multiplying equality \eqref{eq:prekuznetsov} with $m = n$ by $r^k e^{-\pi r}$ and integrating over $r \in[T-2, T + 2]$ we obtain
\begin{equation*}
T^k e^{-\pi T}\left(\sum_{|t_j - T|\leq 1}|\rho_j(m)|^2 + \sum_{i = 1}^\kappa \int_{||t|-T|\leq 1} |\rho_{\a_i, t}(m)|^2\right) \ll_{q, m, n} \int_{T - 2}^{T + 2}r^k e^{\pi r} |\Gamma(1 - k/2 -ir)|^2 \, dr \ll T,
\end{equation*}
where we have used Stirling's formula. The proof of the lemma is finished after noting that $\omega(t) \asymp T^k e^{-\pi T}$ when $||t|-T|\leq 1$.
\end{proof}

Lemma \ref{lem:preliminarykuznetsov} can be used to show that the Fourier coefficients of cusp forms enjoy very strong cancellation properties. It is useful to recast this cancellation as a large sieve inequality, as in the following lemma, which will be used in Section \ref{sec:rankinselbergandadjointlfunction}.

\begin{lem}\label{lem:largesieveinequality}
Let $\{v_j\}$ be an orthonormal basis of $L_0^2(\Gamma_0(M), \psi)$, with Fourier coefficients $\rho_j(n)$. Let $(a(n))_{n \geq 1}$ an arbitrary sequence of complex numbers. Then, for $T \geq 1$ and $N \geq 1$ we have
\begin{equation*}
    \sum_{|t_j| \leq T} \omega(t_j)\left|\sum_{n \leq N}a_n \rho_j(n)\right|^2 \ll_{M, \varepsilon} (T^2 + N^{1 + \varepsilon}) \sum_{n \leq N} |a_n|^2.
\end{equation*}
\end{lem}
\begin{proof}
For $k = 0$ and $\psi = 1$ this lemma is contained in \cite[Theorem 2]{deshouillersiwaniec82}. The result was later extended to arbitrary $k \in \{0, 1\}$ and arbitrary Dirichlet character $\psi$ in \cite[Proposition 4.7]{sarydrappeau}.
\end{proof}
\subsection{Bruggeman--Kuznetsov trace formula.}\label{subsec:kuznetsovtraceformula}
Let $\varphi \in C^2(0, \infty)$ satisfy, for some $\delta > 0$, the bounds
\begin{equation}\label{eq:boundsnear0}
    \varphi(x) \ll x^{1/2 +\delta}, \quad \varphi'(x) \ll x^{-1 + \delta}, \quad \varphi''(x) \ll x^{-2 + \delta},\quad \text{ as }x \to 0^+,
\end{equation}
as well as the bounds
\begin{equation}\label{eq:boundsnearinfty}
    \varphi(x) \ll x^{-1/2 - \delta}, \quad \varphi'(x)\ll x^{-1/2 - \delta}, \quad \varphi''(x) \ll x^{-3/2 - \delta}, \quad \text{ as }x \to \infty.
\end{equation}
Let us write
\begin{equation*}
\varphi_*(x) := \varphi(x) - \int_0^1\xi x J_k(\xi x)\left(\int_0^\infty J_k(\xi y) \varphi(y)\, dy \right)\, d\xi.
\end{equation*}
For $\text{Im}(r) < 1/2$ define
\begin{equation}\label{eq:defihintermsofvarphi}
    h(t) := \begin{dcases}
        \frac{2 \pi i }{\sinh(\pi t)}\int_0^\infty (J_{2it}(y) - J_{-2it}(y))\varphi(y) \frac{dy}{y}, & \text{ if }k = 0,\\
        \frac{2 \pi i}{\cosh(\pi t)}\int_0^\infty (J_{2it}(y) + J_{-2it}(y)) \varphi(y)\frac{dy}{y}, & \text{ if }k = 1.
    \end{dcases}
\end{equation}

We are now ready to state the Bruggeman--Kuznetsov trace formula.

\begin{prop}\label{prop:forwardkuznetsov}
Let the notation be as in Lemma \ref{lem:preliminarykuznetsov}. Suppose that $\varphi \in C^2(0, \infty)$ satisfies the bounds \eqref{eq:boundsnear0} and \eqref{eq:boundsnearinfty} for some $\delta > 0$. Then, for positive $m,n \in \frac{1}{\ell}\Z - \frac{\alpha}{\ell}$ we have the equality
\begin{equation}\label{eq:kuznetsovtraceformula}
\begin{aligned}
    & \frac{\ell}{\vol(\Gamma\backslash \Hyp)} \sum_j \overline{\rho_j(m)} \rho_j(n) \omega(t_j)h(t_j) + \sum_{i = 1}^\kappa\frac{\ell}{4\pi}\int_\R \overline{\rho_{\a_i, t}(m)}\rho_{\a_i, t}(n) \omega(t)h(t)\, dt\\
    & = \delta_{m = n}\frac{i^k}{2\pi}\int_0^\infty J_k(x)\varphi(x)\, dx + \sum_{c = 1}^\infty \frac{S_{(\Gamma, \chi)}(m, n; c)}{c\ell}\varphi_*\left(\frac{4\pi\sqrt{mn}}{c}\right),
\end{aligned}
\end{equation} 
where all sums and integrals converge absolutely.
\end{prop}

After some manipulations, \cite[Equation (84)]{proskurinarbitraryweightkuznetsov} shows that \eqref{eq:kuznetsovtraceformula} holds for a class of test functions which includes $C_c^\infty(0, \infty)$. In the discussion below we will use an approximation argument to extend the identity in \eqref{eq:kuznetsovtraceformula} to all $\varphi \in C^2(0, \infty)$ satisfying the bounds \eqref{eq:boundsnear0} and \eqref{eq:boundsnearinfty} for some $\delta > 0$. We start by collecting a couple of lemmas. 
\begin{lem}
For $t \in \R$ we have
\begin{equation}\label{eq:uniformboundonJBesselfunction}
    J_{2it}(x) \ll (1 + |t|)^{-1/2}e^{\pi |t|}
\end{equation}
uniformly in $x > 0$. 
\end{lem}
\begin{proof}
    Recall Poisson's formula for the $J$-Bessel function \cite[\href{https://dlmf.nist.gov/10.9.E4}{(10.9.4)}]{dlmf}
\begin{equation*}
    J_{2it}(x) = \frac{2(x/2)^{2it}}{\pi^{\frac{1}{2}}\Gamma\left(2it + \frac{1}{2}\right)} \int_0^1 (1 - r^2)^{2it - \frac{1}{2}} \cos(xr) \, dr.
\end{equation*}
By Stirling's formula it is enough to show that the integral is bounded by $(1 + |t|)^{-1/2}$ uniformly on $x > 0$.
It is convenient to change variables so that 
\begin{equation*}
    \int_0^1 (1 - r^2)^{2it - 1/2}\cos(xr) \, dr = \int_0^1 r^{-1/2}(2-r)^{-1/2} e^{i(2t \log r + 2t \log (2-r))} \cos(x(1-r))\, dr.
\end{equation*}
Consider the phase function $\phi_{\pm}(r) = 2t \log r + 2t \log (2-r) \pm xr$. Note that 
$|\phi_{\pm}''(r)| \asymp  |t|/r^2$ for $0 < r \leq 1$. Let $0 < A \leq 1$. Thus, by Van der Corput's lemma \cite[Proposition 2.6.7]{grafakosclassical} we have 
\begin{equation*}
    \int_A^{A + y}  e^{i\phi_{\pm}(r)}\, dr \ll A(1 + |t|)^{-1/2}
\end{equation*}
whenever $0 < y \leq A$, uniformly in $A$. By integration by parts, it follows that  
\begin{equation*}
    \int_A^{2A} r^{-1/2}(2-r)^{-1/2} e^{i(2t \log r + 2t \log (2-r))} \cos(x(1-r))\, dr \ll A^{1/2}(1 + |t|)^{-1/2},
\end{equation*}
for $A \in (0, 1/2]$. The lemma follows after summing over $A = 2^{-i}$ with  $i \geq 1$.
\end{proof}

Define the following seminorms on $C^2(0, \infty)$:
\begin{equation*}
\begin{aligned}
    & \|\varphi\|_{1, \delta} := \sup_{x \in (0, 1]} |\varphi(x)|x^{-1/2 - \delta},\quad  \|\varphi\|_{2, \delta} := \sup_{x \in (0, 1]} |\varphi'(x)|x^{1 - \delta}, \quad \|\varphi\|_{3, \delta} := \sup_{x \in (0, 1]} |\varphi''(x)|x^{2 - \delta},\\
    & \|\varphi\|_{4, \delta} := \sup_{x \in [1, \infty)} |\varphi(x)| x^{1/2 + \delta}, \quad \|\varphi\|_{5, \delta} := \sup_{x \in [1, \infty)} |\varphi'(x)|x^{1/2 + \delta} \quad \|\varphi\|_{6, \delta} := \sup_{x \in [1, \infty)}|\varphi''(x)|x^{3/2 + \delta}.
\end{aligned}
\end{equation*} 
Note that these seminorms are monotonically increasing in $\delta$. We will use these seminorms to control the geometric and spectral side of the Bruggeman--Kuznetsov. We have the following approximation result.

\begin{lem}\label{lem:approximationc2function}
Let $\varphi \in C^2(0, \infty)$ be such that $\sum_i \|\varphi\|_{i, \delta_0} < \infty$ for some $\delta_0 > 0$. Then for any $0 < \delta < \delta_0$ there exists a sequence $\varphi_n \in C_c^\infty(0, \infty)$ such that 
\begin{equation*}
    \sum_{i = 1}^6 \|\varphi- \varphi_n\|_{i, \delta} \to 0.    
\end{equation*}
\end{lem}

\begin{proof}
Let $\psi_0, \psi_\infty \in C^\infty(0, \infty)$ be functions such that $0 \leq \psi_0, \psi_\infty \leq 1$ and also
\begin{equation*}
    \psi_0(x) = \begin{dcases}
        0, & \text{ if } x \leq 1/2,\\
        1, & \text{ if } x \geq 1,
    \end{dcases}\quad \text{ as well as } \quad \psi_\infty(x) = \begin{dcases}
        1, & \text{ if }x \leq 1,\\
        0, & \text{ if }x \geq 2.
    \end{dcases}
\end{equation*}
For $N \geq 10$ define 
\begin{equation*}
    \varphi_N(x) := \varphi(x) \psi_0(xN)\psi_\infty(xN^{-1}),
\end{equation*}
so that 
\begin{equation*}
    \varphi_N(x) = \begin{dcases}
        0, & \text{ if } x \leq (2N)^{-1},\\
        \varphi(x), & \text{ if } N^{-1} \leq x \leq N,\\
        0, & \text{ if }x \geq 2N.
    \end{dcases}
\end{equation*}
One can prove that 
\begin{equation}\label{eq:approxoftestfunctioninV}
\begin{aligned}
\|\varphi - \varphi_N\|_{1, \delta} & \ll N^{\delta - \delta_0} \|\varphi\|_{1, \delta_0},\\
\|\varphi - \varphi_N\|_{2, \delta} & \ll N^{\delta - \delta_0}(N^{-1/2}\|\varphi\|_{1, \delta_0} + \|\varphi\|_{2, \delta_0}),\\
\|\varphi - \varphi_N\|_{3, \delta} & \ll N^{\delta - \delta_0}(N^{-1/2}\|\varphi\|_{1, \delta_0} + \|\varphi\|_{2, \delta_0} +  \|\varphi\|_{3, \delta_0}),\\
\|\varphi - \varphi_N\|_{4, \delta} & \ll N^{\delta - \delta_0}\|\varphi\|_{4, \delta_0},\\
\|\varphi - \varphi_N\|_{5, \delta} & \ll N^{\delta - \delta_0}(N^{-1}\|\varphi\|_{4, \delta_0} + \|\varphi\|_{5, \delta_0}),\\
\|\varphi - \varphi_N\|_{6, \delta} & \ll N^{\delta - \delta_0}(N^{-1}\|\varphi\|_{4, \delta_0} + \|\varphi\|_{5, \delta_0} +  \|\varphi\|_{6, \delta_0}).\\
\end{aligned}
\end{equation}
We will only explain the bound for $\|\cdot\|_{3, \delta}$, the remaining estimates being very similar. When $0 < x \leq 1$ the derivatives of $\psi_\infty(xN^{-1})$ vanish, so we have 
\begin{equation*}
    \varphi''(x) - \varphi_N''(x) = \varphi''(x)(1 - \psi_0(xN)\psi_\infty(xN^{-1})) - 2N\varphi'(x)\psi'_0(xN)\psi_\infty(xN^{-1}) - N^2 \varphi(x)\psi''_0(xN)\psi_\infty(xN^{-1}).
\end{equation*} 
When $0 < x \leq 1$, the quantity $(1 - \psi_0(xN)\psi_\infty(xN^{-1}))$ vanishes unless $x \leq N^{-1}$, thus we have
\begin{equation*}
    \sup_{x \in (0, 1]} |\varphi''(x)(1 - \psi_0(xN)\psi_\infty(xN^{-1}))| x^{2-\delta} \ll N^{\delta - \delta_0}\|\varphi\|_{3, \delta_0}.
\end{equation*}
Similarly, the derivatives of $\psi_0(xN)$ vanish unless $(2N)^{-1} \leq x \leq N^{-1}$, so that we have
\begin{equation*}
    \sup_{x \in (0, 1]} N|\varphi'(x)\psi_0'(xN)\psi_\infty(xN^{-1})| x^{2-\delta} \ll N^{\delta - \delta_0}\|\varphi\|_{2, \delta_0}, 
\end{equation*}
as well as 
\begin{equation*}
    \sup_{x \in (0, 1]} N^2 |\varphi(x)\psi''_0(xN)\psi_\infty(xN^{-1})| x^{2-\delta} \ll N^{-1/2 + (\delta - \delta_0)} \|\varphi\|_{1, \delta_0}.
\end{equation*}
This proves the inequality for $\|\cdot\|_{3, \delta}$ in  \eqref{eq:approxoftestfunctioninV}. Note that $\varphi_N$ is supported on $[(2N)^{-1}, 2N]$, but is only of class $C^2$. Fix $N \geq 10$. Choose $v \in C_c^\infty(0, \infty)$ with support contained in $[0, 1]$ and satisfying $\int_\R v(x) = 1$. For small $\varepsilon > 0$ define $v_\varepsilon(x):= \varepsilon^{-1}v(\varepsilon^{-1}x)$. Consider the convolution $\varphi_N * v_\varepsilon$, which is $C^\infty$ with support contained in $[(2N)^{-1} - \varepsilon, 2N + \varepsilon]$. If $0 < \varepsilon < (3N)^{-1}$, then we have 
\begin{equation*}
    \sum_{i = 1}^6 \|\varphi_N - \varphi_N * v_\varepsilon\|_{i, \delta} \ll N^{3/2 + \delta}\sum_{j = 0}^2\sup_{x \in [1/6N, 3N]}|(\varphi_N - \varphi_N * v_\varepsilon)^{(j)}(x)|.
\end{equation*}
We know that $(\varphi_N * v_\varepsilon)^{(j)}$ tends uniformly to $\varphi^{(j)}$, as $\varepsilon \to 0^+$, for $0 \leq j \leq 2$. Thus, if $\varepsilon$ is sufficiently small, the bounds of \eqref{eq:approxoftestfunctioninV} hold when $\varphi_N$ is replaced by $\varphi_N * v_{\varepsilon}$. This finishes the proof of the lemma.
\end{proof}

\begin{proof}[Proof of Proposition \ref{prop:forwardkuznetsov}]
We start by estimating the spectral and geometric sides of \eqref{eq:kuznetsovtraceformula} in terms of the seminorms $\|\cdot\|_{i, \delta}$, $1 \leq i\leq 6$. To estimate the geometric side first note that $J_k(x) \ll \min(x^k, x^{-1/2}) \ll \min(1, x^{-1/2})$. If $0 < \xi \leq 1$, it follows that 
\begin{equation*}
\begin{aligned}
    \int_0^\infty J_k(\xi y)\varphi(y) \, dy &  \leq \int_0^1 |\varphi(y)| \,dy + \int_1^{\xi^{-1}} |\varphi(y)|\, dy + \xi^{-1/2}\int_{\xi^{-1}}^\infty y^{-1/2}|\varphi(y)| \, dy\\
    & \ll_\delta \|\varphi\|_{1, \delta} + \xi^{-1/2 + \delta}\|\varphi\|_{4, \delta}
\end{aligned}
\end{equation*}
for any $0< \delta < 1/2$. This implies that
\begin{equation*}
    \int_0^1 \xi x J_k(\xi x)\left(\int_0^\infty J_k(\xi y)\varphi(y) \, dy\right)\, d\xi \ll_\delta x \left(\|\varphi\|_{1, \delta} + \|\varphi\|_{4, \delta}\right),
\end{equation*}
so that 
\begin{equation*}
    \varphi_*(x) = \varphi(x) + O_\delta\left(x(\|\varphi\|_{1, \delta} + \|\varphi\|_{4, \delta})\right)
\end{equation*}
for any $\delta > 0$ (recall that the seminorms $\|\cdot\|_{i, \delta}$ are monotonically increasing in $\delta$). Similarly, 
\begin{equation*}
\int_0^\infty J_k(x)\varphi(x) \, dx \ll_{\delta}  \|\varphi\|_{1, \delta} +  \|\varphi\|_{4, \delta}
\end{equation*}
for any $\delta > 0$. Using Lemma \ref{lem:weilbound} we obtain 
\begin{equation}\label{eq:apriorigeometricbound}
     \left|\int_0^\infty J_k(x)\varphi(x) \, dx \right| + \sum_{c \geq 1}\left|\frac{S_{(\Gamma, \chi)}(m, n; c)}{c\ell}\varphi_*\left(\frac{4\pi\sqrt{m n}}{c }\right)\right| \ll_{n, m, q, \delta}  \|\varphi\|_{1, \delta} + \|\varphi\|_{4, \delta}
\end{equation}
for any $\delta > 0$. To bound the spectral side we follow \cite{proskurinsummationformulas}. We want to prove the estimate
\begin{equation}\label{eq:estimatehfunction}
    h(t) \ll_\delta (1 + |t|)^{-2-\delta} \left(\sum_{i = 1}^6 \|\varphi\|_{i, \delta}\right)
\end{equation}
for any $\delta > 0$. This is clear for $|t| \leq 1$, so we can assume that $|t| \geq 1$. By Bessel's equation \cite[\href{https://dlmf.nist.gov/10.2.E1}{(10.2.1)}]{dlmf} we have
\begin{equation*}
    x^2 J_{2it}''(x) + xJ_{2it}'(x) + (x^2 + 4t^2)J_{2it}(x) = 0,
\end{equation*}
and after integration by parts we see that 
\begin{equation*}
   \int_0^\infty J_{2it}(y) \varphi(y)\frac{dy}{y} = -\int_0^\infty J_{2it}(y) \frac{d^2}{d^2 y}\left(\frac{y\varphi(y)}{y^2 + 4t^2}\right) \, dy + \int_0^\infty J_{2it}(y)\frac{d}{dy}\left(\frac{\varphi(y)}{y^2 + 4t^2}\right)\, dy. 
\end{equation*}
Using \eqref{eq:uniformboundonJBesselfunction} and the definition of the seminorms $\|\cdot\|_{i, \delta}$ we can estimate 
\begin{equation*}
\begin{aligned}
    e^{-\pi |t|}(1 + |t|)^{1/2}\int_0^\infty J_{2it}(y)\frac{d}{dy}\left(\frac{\varphi(y)}{y^2 + 4t^2}\right)\, dy  & \ll \int_0^1 \frac{|\varphi'(y)|}{|t|^2} + \frac{y|\varphi(y)|}{|t|^4}\, dy + \int_1^{|t|} \frac{|\varphi'(y)|}{|t|^2} + \frac{y|\varphi(y)|}{|t|^4}\, dy\\
    & + \int_{|t|}^\infty \frac{|\varphi'(y)|}{y^2} + \frac{|\varphi(y)|}{y^3}\, dy\\
    & \ll_\delta |t|^{-2}\|\varphi\|_{2, \delta} + |t|^{-4}\|\varphi\|_{1, \delta}\\
    & + |t|^{-3/2 - \delta}\|\varphi\|_{5, \delta} + |t|^{-5/2 - \delta}\|\varphi\|_{4, \delta}
\end{aligned}
\end{equation*}
for $|t| \geq 1$ and any $\delta > 0$. In the same way one can show that
\begin{equation*}
\begin{aligned}
e^{-\pi |t|}(1 + |t|)^{1/2}\int_0^\infty J_{2it}(y) \frac{d^2}{d^2 y}\left(\frac{y\varphi(y)}{y^2 + 4t^2}\right) \, dy & \ll_\delta |t|^{-4}\|\varphi\|_{1, \delta} + |t|^{-2}(\|\varphi\|_{2, \delta} + \|\varphi\|_{3, \delta})\\
& + |t|^{-5/2 - \delta}\|\varphi\|_{4, \delta} + |t|^{-3/2 - \delta}(\|\varphi\|_{5, \delta} + \|\varphi\|_{6, \delta})
\end{aligned}
\end{equation*}
for $|t| \geq 1$ and any $\delta > 0$. The bound in \eqref{eq:estimatehfunction} follows by combining these estimates with the definition of $h(t)$ from \eqref{eq:defihintermsofvarphi}.  If we combine \eqref{eq:estimatehfunction} with Lemma \ref{lem:upperboundfouriercoefficients} we arrive at the bound
\begin{equation}\label{eq:apriorispectralbound}
\begin{aligned}
& \sum_j \left|\overline{\rho_j(m)} \rho_j(n) \omega(t_j)h(t_j)\right| + \sum_{i = 1}^\kappa\frac{1}{4\pi}\int_\R \left|\overline{\rho_{\a_i, t}(m)}\rho_{\a_i, t}(n) \omega(t)h(t)\right| \, dt\\
&  \ll_{m, n, q, \delta} \sum_{i = 1}^6 \|\varphi\|_{i, \delta},
\end{aligned}
\end{equation}
valid for any $\delta > 0$. We are now ready to conclude the proof of Proposition \ref{prop:forwardkuznetsov}. After substracting the right-hand side from the left-hand side we can write the identity in \eqref{eq:kuznetsovtraceformula} as $J(\varphi) = 0$ for a linear functional $J$. Equations \eqref{eq:apriorigeometricbound} and \eqref{eq:apriorispectralbound} show that 
\begin{equation}\label{eq:aprioriestimatekuznetsov}
    |J(\varphi)| \ll_{\delta} \sum_{i = 1}^6 \|\varphi\|_{i, \delta},
\end{equation}
for any $\delta> 0$. Suppose that $\varphi_0 \in C^2(0, \infty)$ satisfies the bounds in \eqref{eq:boundsnear0} and in \eqref{eq:boundsnearinfty} for some $\delta_0 > 0$, and fix some $\delta_1$ such that $0< \delta_1 < \delta_0$. By Lemma \ref{lem:approximationc2function}, there exist functions $\varphi_n \in C_c^\infty(0, \infty)$ such that 
\begin{equation}\label{eq:testfunctionapproximated}
    \sum_{i = 1}^6 \|\varphi_0 - \varphi_n\|_{i, \delta_1} \to 0 \quad \text{ as }n \to \infty.
\end{equation}
By \cite[Equation (84)]{proskurinarbitraryweightkuznetsov} we know that $J(\varphi_n) = 0$ for each $n$. Thus, we deduce
\begin{equation*}
    J(\varphi_0) = J(\lim_n \varphi_n) = \lim_n J(\varphi_n) = 0,
\end{equation*}
where the exchange of limits is allowed by \eqref{eq:aprioriestimatekuznetsov} and \eqref{eq:testfunctionapproximated}. This finishes the proof of the proposition.
\end{proof}

\section{Fourier coefficients of Eisenstein series.}\label{sec:fouriercoefficientseisensteinseries}
In the proof of Theorem \ref{thm:luosarnakbound} we will apply Proposition \ref{prop:forwardkuznetsov} when $n = m$. In the spectral side of \eqref{eq:kuznetsovtraceformula} the contribution from the continuous spectrum can be written as
\begin{equation*}
    \frac{\ell}{4\pi}\int_{\R} h(t) \omega(t)\sum_{i = 1}^\kappa\left|\rho_{\a_i, t}(n)\right|^2 \, dt. 
\end{equation*}
Accordingly, for any $m \in \Q_{> 0}$ we let 
\begin{equation}\label{eq:defiofl2normfouriercoefficientseisenstein}
    \Sigma(m; t; \Gamma, \chi) := \vol(\Gamma\backslash \Hyp)\sum_{[\a]_\Gamma \text{ singular for }(\Gamma, \chi)}\left|\rho_{\a, t}(m)\right|^2,
\end{equation} 
where the sum is over inequivalent singular cusps of $(\Gamma, \chi)$. If we recall \eqref{eq:transformationdifferentrepresentativeeisenstein}, we see that this quantity is independent of the choice of representative for the class $[\a]_\Gamma$. 

The purpose of this section is to estimate $\Sigma(m; t; \Gamma, \chi)$, see Proposition \ref{prop:upperboundfouriercoefficentseisenstein} below. The idea is to deduce the estimate from the calculations of \cite{youngeisensteinseries}, which deal with the case $\Gamma = \Gamma_0(N)$.

\subsection{Hecke congruence subgroup.}

Consider $N \geq 1$ and let $\psi$ be a Dirichlet modulo $N$. Let $\chi_1, \chi_2$ be primitive characters modulo $q_1, q_2$ respectively, such that $q_1q_2 \mid N$ and $\chi_1\overline{\chi_2} = \psi$ when considered as Dirichlet characters modulo $N$. Define
\begin{equation*}
    E_{\chi_1, \chi_2}(z; s) := \frac{1}{2}\sum_{(c, d) = 1}\frac{(q_2y)^s\chi_1(c)\chi_2(d)}{|cq_2 z + d|^{2s}}\left(\frac{|cq_2z + d|}{cq_2z + d}\right)^k,
\end{equation*}
as in \cite[Equation (3.3)]{youngeisensteinseries}.
This definition is valid for $\text{Re}(s) > 1$, and can be analytically continued to a meromorphic function on $\CC$ satisfying a functional equation. If we let the completed Eisenstein series be
\begin{equation}\label{eq:completedeisenstein}
    E^*_{\chi_1, \chi_2}(z; s) := \frac{(q_2/\pi)^s}{i^{-k}\tau(\chi_2)}\Gamma(s + k/2)L(2s; \chi_1\chi_2)E_{\chi_1, \chi_2}(z; s),
\end{equation}
then we have a very neat Fourier expansion 
\begin{equation}\label{eq:fourierexpansioncompletedeisenstein}
    E^*_{\chi_1, \chi_2}(z; s) = e^*_{\chi_1, \chi_2}(y; s) + \sum_{n \neq 0} \frac{\lambda_{\chi_1, \chi_2}(n; s)}{|n|^{1/2}} \frac{\Gamma(s + \frac{k}{2})}{\Gamma(s + \frac{k}{2}\sgn(n))}W_{\frac{k}{2}\sgn(n), s - \frac{1}{2}}(4\pi |n|y)e(nx),
\end{equation}
see \cite[Proposition 4.1]{youngeisensteinseries}. In this equation $e^*_{\chi_1, \chi_2}(y; s)$ is the constant term and we have 
\begin{equation*}
    \lambda_{\chi_1, \chi_2}(n; s) :=\chi_2(\sgn(n))\sum_{ab = |n|} \chi_1(a)\overline{\chi_2}(b)\left(\frac{b}{a}\right)^{s-\frac{1}{2}}.
\end{equation*}
In particular, on the critical line these coefficients satisfy $|\lambda_{\chi_1, \chi_2}(n; 1/2 + it)| \leq \tau(n)$. 

Let $\psi$ a Dirichlet character modulo $N$, not necessarily primitive, which we identify with the character of $\Gamma_0(N)$ given by $\gamma \mapsto \psi(d(\gamma))$. The singular cusps of $(\Gamma_0(N), \psi)$ are described in \cite[section 5]{youngeisensteinseries}. We can always write $\a = \frac{1}{uf}$ where $(u, N) = 1$, $f\mid N$ and the conductor of $\psi$ divides $N/(f, N/f)$. The Eisenstein series associated to such a cusp can be related to the Eisenstein series associated to Dirichlet characters as follows:
\begin{equation*}
\begin{aligned}
E_{\frac{1}{uf}}(z; s; \Gamma_0(N), \psi) & = \frac{1}{(fN'')^s}\frac{1}{\varphi((f, N/f))}\sum_{q_1 \mid \frac{N}{f}}\sum_{q_2 \mid f}\, \, \,  \psum_{\substack{\chi_1 \modulo q_1\\
\chi_2 \modulo q_2\\
\chi_1 \overline{\chi_2} \simeq \psi}} \overline{\chi_1}(-u)\frac{L(2s, \chi_1 \chi_2)}{L(2s, \chi_1\chi_2 \chi_{0, N})}\\
& \quad \sum_{\substack{a \mid f\\
(a, q_2) = 1}}\sum_{\substack{b \mid \frac{N}{f}\\
(b, q_1) = 1}}\frac{\mu(a)\mu(b)\chi_1(b)\chi_2(a)}{(ab)^s}E_{\chi_1, \chi_2}\left(\frac{bf}{aq_2}z; s\right),
\end{aligned}
\end{equation*}
see \cite[Theorem 6.1]{youngeisensteinseries}. Here $N'' = N'/(f, N')$ and $\chi_{0, N}$ is the trivial character modulo $N$. Also, the notation $\chi_1 \overline{\chi_2} \simeq \psi$ means equality of the induced characters modulo $N$. Comparing this equation with \eqref{eq:completedeisenstein} and \eqref{eq:fourierexpansioncompletedeisenstein} we see that for $m \geq 1$ we have
\begin{equation*}
    \rho\left(E_{\frac{1}{uf}}(\cdot; 1/2 + it; \Gamma_0(N), \psi); m\right) \ll_{\varepsilon, N} \frac{m^\varepsilon}{L(1 + 2it; \chi_1 \chi_2) \Gamma(1/2 + it + k/2)} \ll_{\varepsilon, N} e^{\frac{\pi}{2}|t|}(1 + |t|)^{-k/2 + \varepsilon} m^\varepsilon.
\end{equation*}
Squaring and summing over the different cusps we deduce that 
\begin{equation}\label{eq:boundfouriereisensteinhecke}
    \Sigma(m; t; \Gamma_0(N), \psi) \ll_{N, \varepsilon} e^{\pi |t|}(1 + |t|)^{-k + \varepsilon} m^\varepsilon
\end{equation}
for $m \geq 1$.

\subsection{General case.}
Let $\a_1, \ldots, \a_\kappa$ be representatives of the inequivalent singular cusps of $(\Gamma, \chi)$. For $t \in \R$, we let 
\begin{equation*}
    \E_{it}(\Gamma, \chi) := \text{span}\{E_{\a_i}(\cdot; 1/2 + it; \Gamma, \chi) : 1 \leq i \leq \kappa\}.
\end{equation*}
When $t \neq 0$, the spanning set $\{E_{\a_i}(\cdot; 1/2 + it; \Gamma, \chi): 1 \leq i \leq \kappa\}$ is actually a basis of $\E_{it}(\Gamma, \chi)$. This can be deduced by computing the constant term  of the Fourier expansions of the different Eisenstein series at the different cusps and using that the scattering matrix $\Phi(s)$ is invertible on the critical line. Thus, if $t \neq 0$ we can introduce an inner product $\langle \cdot, \cdot \rangle$ on $\E_{it}(\Gamma, \chi)$ by letting
\begin{equation}\label{eq:innerproductEisensteinseries}
    \langle E_{\a_i}(\cdot; 1/2 + it; \Gamma, \chi), E_{\a_j}(\cdot; 1/2 + it; \Gamma, \chi)\rangle := \delta_{i =  j}\frac{1}{\vol(\Gamma\backslash \Hyp)},
\end{equation}
which agrees with \cite[Equation (8.1)]{youngeisensteinseries} up to normalization. Looking at \eqref{eq:defiofl2normfouriercoefficientseisenstein} we recognize that
\begin{equation}\label{eq:continuousfouriercoefficientsasl2norm}
    \Sigma(m; t; \Gamma, \chi) = \sup_{0 \neq f \in \E_{it}(\Gamma, \chi)}\frac{|\rho(f; m)|^2}{\langle f, f \rangle }.
\end{equation}
That is, the quantity $\Sigma(m; t; \Gamma, \chi)$ is the squared $L^2$-norm of the linear functional $f \mapsto \rho(f; m)$ on the finite-dimensional Hilbert space $\E_{it}(\Gamma, \chi)$.

Until the end of this section we use the notation $\Gamma_i(\a) := \{\gamma \in \Gamma_i : \gamma \a = \a\}$ and $[\a]_{\Gamma_i} := \{\gamma \a : \gamma \in \Gamma_i\}$, where $i \in \{1, 2\}$ and $\a \in \P^1(\Q)$. We also let $\sigma_i(\a)$ denote a scaling matrix for $\a \in \P^1(\Q)$ with respect to the group $\Gamma_i$. 
\begin{lem}
Suppose that we have $(\Gamma_1, \chi_1)\subset (\Gamma_2, \chi_2)$ for two congruence characters. Then, for $t \neq 0$, the natural inclusion $\E_{it}(\Gamma_2, \chi_2) \hookrightarrow \E_{it}(\Gamma_1, \chi_1)$ is an isometry.
\end{lem}
\begin{proof}
Let $\a$ be a singular cusp for $(\Gamma_2, \chi_2)$.  The map $\gamma \mapsto \gamma^{-1}\a$ establishes a bijection
\begin{equation*}
    \Gamma_2(\a)\backslash \Gamma_2 \leftrightarrow [\a]_{\Gamma_2}.
\end{equation*}
Passing to $\Gamma_1$-orbits we have 
\begin{equation*}
    [\a]_{\Gamma_2} = \bigsqcup_{\delta \in \Gamma_2(\a)\backslash \Gamma_2 /\Gamma_1} [\delta^{-1}\a]_{\Gamma_1}, 
\end{equation*}
where we note that
\begin{equation}\label{eq:relationstabilizers}
    \Gamma_2(\delta^{-1}\a) = \delta^{-1}\Gamma_2(\a)\delta, \quad 
    \Gamma_1(\delta^{-1} \a) = \delta^{-1}\Gamma_2(\a)\delta \cap \Gamma_1,
\end{equation}
for $\delta \in \Gamma_2$. We also have the identity
\begin{equation}\label{eq:sumofrelativewidthofcusps}
    \sum_{\delta \in \Gamma_2(\a) \backslash \Gamma_2/\Gamma_1}[\Gamma_2(\delta^{-1}\a): \Gamma_1(\delta^{-1}\a)] =  \sum_{\delta \in \Gamma_2(\a) \backslash \Gamma_2/\Gamma_1}[\Gamma_2(\a): \Gamma_2(\a) \cap \delta\Gamma_1 \delta^{-1}] = [\Gamma_2: \Gamma_1].
\end{equation}
We let $\sigma_i(\a)$ denote a scaling matrix for $\a \in \P^1(\Q)$ with respect to the group $\Gamma_i$. Note that we can choose $\sigma_1(\a)$, $\sigma_2(\a)$ so that 
\begin{equation}\label{eq:relationscalingmatrix}
    \sigma_2(\a)^{-1} = \begin{psmallmatrix}
        [\Gamma_2(\a): \Gamma_1(\a)]^{1/2}  & 0\\
        0 & [\Gamma_2(\a): \Gamma_1(\a)]^{-1/2} 
    \end{psmallmatrix} \sigma_1(\a)^{-1}.
\end{equation}
From the definition \eqref{eq:defofeisensteinseries} of the Eisenstein series we have  
\begin{equation}\label{eq:relationeisensteincongruence}
\begin{aligned}
    E_\a(z; s; \Gamma_2, \chi_2) & = \sum_{\delta \in \Gamma_2(\a) \backslash \Gamma_2/\Gamma_1} \overline{\chi_2}(\delta)\sum_{\gamma \in \Gamma_1(\a)\backslash \Gamma_1} \overline{\chi_2}(\gamma) [f_s |_k (\sigma_2(\a)^{-1}\delta \gamma)](z)\\
    & = \sum_{\delta \in \Gamma_2(\a) \backslash \Gamma_2/\Gamma_1} \overline{\chi_2}(\delta)[\Gamma_2(\delta^{-1}\a): \Gamma_1(\delta^{-1}\a)]^s E_{\delta^{-1}\a}(z; s; \Gamma_1, \chi_1),
\end{aligned}
\end{equation}
where we have used \eqref{eq:relationscalingmatrix} and the fact that $\sigma_2(\a)^{-1}\delta = \sigma_2(\delta^{-1}\a)^{-1}$ for $\delta \in \Gamma_2$. Identity \eqref{eq:relationeisensteincongruence} is proved first for $\text{Re}(s) > 1$ and extended to $s \in \CC$ away from poles by analytic continuation. The proof of the lemma is finished by combining \eqref{eq:relationeisensteincongruence} for $\text{Re}(s) = 1/2$ with the definition of the formal inner product \eqref{eq:innerproductEisensteinseries} and using \eqref{eq:sumofrelativewidthofcusps}.
\end{proof}
If we combine the previous lemma with \eqref{eq:continuousfouriercoefficientsasl2norm} we obtain the following fact.
\begin{cor}\label{cor:monotonicityoffouriercoefficients}
Suppose that $(\Gamma_1, \chi_1)\subset (\Gamma_2, \chi_2)$ for two congruence characters. Then for $m \in \Q_{> 0}$ we have 
\begin{equation*}
\Sigma(m; t; \Gamma_2, \chi_2) \leq \Sigma(m; t; \Gamma_1, \chi_1).
\end{equation*}
\end{cor}

Suppose now that $\Gamma_1$ is normal in $\Gamma_2$. In this case, if $\delta \in \Gamma_2$ then we have $\delta^{-1}\Gamma_1(\a)\delta = \Gamma_1(\delta^{-1}\a)$. It follows that 
\begin{equation}\label{eq:smallcomputationindexstabilizer}
[\Gamma_2(\delta^{-1}\a): \Gamma_1(\delta^{-1}\a)] = [\delta^{-1}\Gamma_2(\a)\delta: \delta^{-1}\Gamma_1(\a)\delta] = [\Gamma_2(\a): \Gamma_1(\a)].
\end{equation}
Also, from the definition \eqref{eq:defofeisensteinseries} we see that 
\begin{equation*}
\begin{aligned}
    E_\a(\cdot; s; \Gamma_1, \chi_1)|_k \delta & = \sum_{\gamma \in \Gamma_1(\a) \backslash \Gamma_1} \overline{\chi(\gamma)}[f_s |_k \sigma_1(\a)^{-1}\gamma \delta] = \sum_{\gamma \in \Gamma_1(\delta^{-1}\a)\backslash \Gamma_1}\overline{\chi(\gamma)}[f_s|_k \sigma_1(\a)^{-1}\delta \gamma]\\
    & = E_{\delta^{-1}\a}(\cdot; s; \Gamma_1, \chi_1),
\end{aligned}
\end{equation*}
so that $\Gamma_2$ acts on $\E_{it}(\Gamma_1, \chi_1)$. Since we are assuming that $\Gamma_1$ is normal in $\Gamma_2$, it is clear that the action of $\Gamma_2$ preserves $\Gamma_1$-equivalence classes of cusps. Thus, the action of $\Gamma_2$ on $\E_{it}(\Gamma_1, \chi_1)$ is unitary with respect to the inner product \eqref{eq:innerproductEisensteinseries}. Using \eqref{eq:relationeisensteincongruence} and \eqref{eq:smallcomputationindexstabilizer} it is easy to recognize that $\E_{it}(\Gamma_2, \chi_2)$ is the $\chi_2$-isotypic component of $\E_{it}(\Gamma_1, \chi_1)$, that is,
\begin{equation*}
    \E_{it}(\Gamma_2, \chi_2) = \{f \in \E_{it}(\Gamma_1, \chi_1) : f|_k \gamma = \chi_2(\gamma)f \text{ for all }\gamma \in \Gamma_2\}.
\end{equation*}

It follows that if $\chi_2$ and $\chi_2'$ are different characters extending $\chi_1$ to $\Gamma_2$, then $\E_{it}(\Gamma_2, \chi_2)$ and $\E_{it}(\Gamma_2, \chi'_2)$ are orthogonal with respect to the inner product on $\E_{it}(\Gamma_1, \chi_1)$.

We now consider a particular case of this situation. Suppose that $\Gamma_1 = \{\pm I\}\Gamma(q)$ and $\chi_1 = \sgn_q^k$ is the character which is trivial on $\Gamma(q)$ and has parity $k$. Also, suppose that $\Gamma_2 = \Gamma_d(q)$. Since 
\begin{equation*}
    \Gamma_1 \backslash \Gamma_2 = \{\pm I\}\Gamma(q) \backslash \Gamma_d(q) \simeq \{\pm 1\} \backslash (\Z/q\Z)^\times,
\end{equation*}
we deduce that there is an orthogonal decomposition
\begin{equation*}
    \E_{it}(\{\pm I\}\Gamma(q), \sgn_q^k) = \bigoplus_{\substack{\psi \modulo q\\
    \psi(-1) = (-1)^k}}\E_{it}(\Gamma_d(q), \psi).
\end{equation*}
It follows that for any $m \in \Q - \{0\}$ we have
\begin{equation*}
    \Sigma(m; t; \{\pm I\}\Gamma(q), \sgn_q^k) = \sum_{\substack{\psi \modulo q\\
    \psi(-1) = (-1)^k}} \Sigma(m; t; \Gamma_d(q), \psi).
\end{equation*} 
Combined with Corollary \ref{cor:monotonicityoffouriercoefficients} we have arrived at the following conclusion.
\begin{lem}\label{lem:controlfouriercoefficients}
Let $(\Gamma, \chi)$ be a congruence character of level $q$ and weight $k$. Then, for any $m \in \Q - \{0\}$ we have
\begin{equation*}
    \Sigma(m; t; \Gamma, \chi) \leq \sum_{\substack{\psi \modulo q\\
    \psi(-1) = (-1)^k}} \Sigma(m; t; \Gamma_d(q), \psi).
\end{equation*}
\end{lem}
Finally, note that
\begin{equation*}
    \begin{psmallmatrix}
        q & 0\\
        0 & 1
    \end{psmallmatrix}^{-1}\Gamma_d(q)\begin{psmallmatrix}
        q & 0\\
        0 & 1
    \end{psmallmatrix} = \Gamma_0(q^2),
\end{equation*}
from where it follows easily that 
\begin{equation}\label{eq:comparisonfourierGammadandGamma0}
    \Sigma(m; t; \Gamma_d(q), \psi) = q^{-1}\Sigma(qm; t; \Gamma_0(q^2), \psi).
\end{equation}
Combining this relation with \eqref{eq:boundfouriereisensteinhecke} and the asymptotic $\omega(t) \asymp e^{-\pi|t|}(1 + |t|)^k$ we deduce the following result.

\begin{prop}\label{prop:upperboundfouriercoefficentseisenstein}
Let $(\Gamma, \chi)$ be a congruence character of level $q$. Then, for $m \in q^{-1}\Z_{\geq 1}$ and $\varepsilon > 0$ we have 
\begin{equation*}
    \omega(t)\Sigma(m; t; \Gamma, \chi) \ll_{q, \varepsilon} (m(1 + |t|))^\varepsilon.
\end{equation*}
\end{prop}

\section{Explicit formula for $\Psi_\Gamma(X; \chi)$.}\label{sec:iwaniecslemma}
In this paragraph we extend the explicit formula for $\Psi_{\SL_2(\Z)}(X)$ introduced in \cite[Lemma 1]{Iwaniec1984} to a general congruence character. We just need to make a few observations to show that the same proof can be generalized.

We start by generalizing \cite[Lemma 4]{Iwaniec1984}. Recall that $(\Gamma, \chi)$ denotes a congruence character of level $q$. In particular, we assume that $\{\pm I\} \subset\Gamma$. For $x \geq 1$ we define $\pi_\Gamma(x) := \#\{\text{primitive }\{P_0\}_\Gamma:  N(P_0) \leq x, \Tr(P) > 2\}$. 
\begin{lem}
If $x^{1/2}(\log x)^2 < y < x$ we have
\begin{equation*}
    \pi_\Gamma(x + y) - \pi_\Gamma(x) \ll_q y.
\end{equation*}
\end{lem}
\begin{proof}
Recall that for a hyperbolic element $P \in \Gamma$ we let $\nu_\Gamma(P)$ be the positive integer such that $P = P_0^{\nu_\Gamma(P)}$ for a primitive $P_0 \in \Gamma$. By the basic estimate $\Psi_\Gamma(x) = O_q(x)$ we deduce that
\begin{equation}\label{eq:chebyshevsumisessentiallyoverprimes}
    \pi_\Gamma(x) = \sum_{\substack{\{P\}_\Gamma\\
    N(P) \leq x\\
    \Tr(P) > 2}} \frac{1}{\nu_\Gamma(P)} + O_q(x^{1/2}),
\end{equation}
the error term coming from the contribution of nonprimitive conjugacy classes. By \eqref{eq:sec2:auxeq2} within the proof of Lemma \ref{lem:compatibilitywithinduction} we know that
\begin{equation*}
\sum_{\substack{\{P\}_\Gamma\\
    N(P) \leq x\\
    \Tr(P) > 2}} \frac{1}{\nu_\Gamma(P)} = \sum_{\substack{\{P\}_{\SL_2(\Z)}\\
    N(P) \leq x\\
    \Tr(P) > 2}} \frac{f(P)}{\nu_{\SL_2(\Z)}(P)},
\end{equation*} 
where $f = \Ind_{\Gamma}^{\SL_2(\Z)}1$. Note that $|f(P)| \leq [\SL_2(\Z): \Gamma]$ for all $P$. Thus, using \eqref{eq:chebyshevsumisessentiallyoverprimes} twice, once for $\Gamma$ and once for $\SL_2(\Z)$, we deduce that  
\begin{equation*}
\pi_\Gamma(x + y) - \pi_\Gamma(x) = \sum_{\substack{\{P\}_{\SL_2(\Z)}\\
    x < N(P) \leq x + y\\
    \Tr(P) > 2}} \frac{f(P)}{\nu_{\SL_2(\Z)}(P)} + O_q(x^{1/2}) \ll_q \pi_{\SL_2(\Z)}(x + y) - \pi_{\SL_2(\Z)}(x) + x^{1/2} \ll y,
\end{equation*}
where the last inequality follows from \cite[Lemma 4]{Iwaniec1984}. 
\end{proof}

Now we generalize equations (23), (24) and (25) from \cite{Iwaniec1984}. 

For a hyperbolic element $P$ of $\Gamma$ define
\begin{equation*}
    \Lambda(P) = \frac{\log N(P_0)}{1 - N(P)^{-1}}
\end{equation*}
and consider the Dirichlet series 
\begin{equation*}
    D(s) := \sum_{\substack{\{P\}_\Gamma\\
    \Tr(P) > 2}} \frac{\Lambda(P) \chi(P)}{N(P)^s}
\end{equation*}
which converges absolutely for $\text{Re}(s) > 1$.
Since $\Tr(P) = \Tr(P^{-1})$ we have
\begin{equation*}
D(s) = \frac{1}{2}\sum_{\substack{\{P\}_\Gamma\\
    \Tr(P) > 2}} \frac{\Lambda(P) (\chi(P) + \overline{\chi(P)})}{N(P)^s},
\end{equation*}
so that $D(\overline{s}) = \overline{D(s)}$. In \cite[Chapter 10]{hejhal2} it is proven that $D(s)$ has meromorphic continuation to $\CC$ and its singularities are simple poles. Let us recall the precise description of these poles from \cite[Chapter 10, Theorem 2.16]{hejhal2}. Let $(\rho_n)_n$ enumerate the zeroes of the scattering determinant $\phi(s)$ of $(\Gamma, \chi)$. We know that $1/2 < \text{Re}(\rho_n) < 1$ for each $n$. Recall that $\{u_j\}_j$ is an orthonormal basis of cuspidal Maass forms so that $Pu_j = (1/4 + t_j^2)u_j$, with either $t_j \geq 0$ or else $\text{Im}(t_j) > 0$. In the region $\text{Re}(s) \geq -1/2$ the singularities of $D(s)$ are described by the formal expansion
\begin{equation}\label{eq:formalexpansionoflogarithmicselbergzeta}
    D(s) = \frac{\delta_{\chi = 1}}{s-1} + \sum_{j} \left(\frac{1}{s - 1/2 - it_j} + \frac{1}{s - 1/2 + it_j}\right) + \sum_{n}\frac{1}{s + \overline{\rho_n} - 1} + \frac{A}{s - 1/2} + \frac{B}{s},
\end{equation}
where the precise value of $A$ and $B$ is not important for our purposes. An effective version of this formal expansion is given in \cite[Chapter 10, Theorem 2.24(iv)]{hejhal2}. Precisely, for $s = \sigma + it$ with bounded $\sigma$ and $t \geq 10$ we have
\begin{equation}\label{eq:effecticexpansionoflogarithmicselbergzeta}
    D(s) = \sum_{|t_j - t| \leq 1}\frac{1}{s - 1/2 - it_j} + \sum_{|\text{Im}(\rho_n) - t| \leq 1}\frac{1}{s - (1 - \overline{\rho_n})} + O_q(t),
\end{equation}
generalizing \cite[Equation (20)]{Iwaniec1984}. Recall also that by Weyl's law we have $\#\{j : |t_j - T| \leq 1\}\ll_q T$. On the other hand, the poles $(1-\overline{\rho_n})$ are quite sparse when compared to the cuspidal eigenvalues. Indeed, from Lemma \ref{lem:basicformscattering} we can deduce that 
\begin{equation*}
    \#\{n : |\text{Im}(\rho_n)-T| \leq 1\} \ll_q \log T.
\end{equation*}

Then, exactly as in \cite[section 3]{Iwaniec1984} it follows that 
\begin{equation*}
    |D(\sigma + it)| \ll_{q, \varepsilon} |t|^{2\max(0, 1-\sigma)} \log|t| \quad \text{ for }|t| \geq 2, \sigma > \frac{1}{2} + \frac{1}{\log |t|},
\end{equation*}
that
\begin{equation*}
    D(-\varepsilon + it) \ll |t| + 1, \text{ for } 0 < \varepsilon < 1/4,
\end{equation*}
and that for any $T \geq 2$ there exists $\tau \in [T, T + 1]$ such that 
\begin{equation*}
    \int_0^1 |D(\sigma + i\tau)| \, d\sigma \ll T\log T.
\end{equation*}
These bounds generalize, respectively, (23), (24) and (25) from \cite{Iwaniec1984}. If we define 
\begin{equation*}
    \L(s) := D(s) - D(s + 1) = \sum_{\substack{\{P\}_\Gamma\\
    \Tr(P) > 2}} \frac{\log N(P_0)\chi(P)}{N(P)^s},
\end{equation*}
then the same bounds are true also for $\L(s)$. Collecting these observations we obtain the following result.

\begin{lem}[Iwaniec]\label{lem:iwaniecslemma}
Let $(\Gamma, \chi)$ be a congruence character of level $q$. Then, for any $X \geq 10$ and $1 \leq T \leq \sqrt{X}(\log X)^{-2}$, we have
\begin{equation*}
    \Psi_\Gamma(X; \chi) = \delta_{\chi = 1} X + \sum_{|t_j|\leq T}\left(\frac{X^{1/2 + it_j}}{1/2 + it_j} + \frac{X^{1/2 -it_j}}{1/2 - it_j}\right) + O_q\left(\frac{X}{T}\log^2 X\right).
\end{equation*}
\end{lem}
\begin{proof}
Follows the proof of \cite[Lemma 1]{Iwaniec1984} with the changes described above.
\end{proof}

\section{Newforms and orthonormal basis of $L_0^2(\Gamma, \chi)$.}\label{sec:orthonormalbasis}
When $(\Gamma, \chi)$ is an arbitrary congruence character of level $q$, it is not obvious how to find a basis $\{u_j\}_j$ related to Hecke newforms. In this section we show that there exists a basis of $L_0^2(\Gamma, \chi)$ with the following property: for each $j$ there is a newform $f_j$ such that $u_j$ belongs to the sum of oldspaces of the twists $f_j \otimes \psi$, where $\psi$ ranges over Dirichlet characters modulo $q$. For a more precise statement see Corollary \ref{cor:basisaslinearcombinationoftwists} below. 

Let $t \in \R_{> 0} \cup i[0, 7/64]$ and consider the finite-dimensional space 
\begin{equation}\label{eq:cuspidalspacefixedspectraldata}
    \C_{it}(\Gamma, \chi) := \{u \in L_0^2(\Gamma, \chi) : Pu = (1/4 + t^2)u\}.
\end{equation}
Recall that we have fixed the weight $k \in \{0, 1\}$ such that $\chi(-I) = (-1)^k$, and we have dropped it from the notation. Let $\sgn_q^k$ be the character on $\{\pm I\}\Gamma(q)$ which is trivial on $\Gamma(q)$ and has weight $k$. For conciseness write $\C_{it}(\Gamma(q)) := \C_{it}(\{\pm I\}\Gamma(q), \sgn_q^k)$. Then we have a natural inclusion $\C_{it}(\Gamma, \chi) \subset \C_{it}(\Gamma(q))$ which is an isometry when the inner products are normalized as in \eqref{eq:innerproductmaasforms}.

\subsection{Newforms, oldspaces and twists.}
We recall some familiar facts about newforms and oldspaces. A detailed exposition for these results can be found in sections 4.5 and 4.6 of \cite{miyake}. Even though Miyake treats holomorphic modular forms, the theory for the nonholomorphic case is completely analogous.

Consider the congruence pair $(\Gamma_0(M), \psi)$, where $\psi$ is a Dirichlet character of conductor dividing $M$. For any $n \geq 1$ the Hecke operators $T(n; M, \psi)$ are defined on $\C_{it}(\Gamma_0(M), \psi)$ by the formula
\begin{equation}\label{eq:Heckeoperatorsexplicitformula}
    (T(n; M, \psi)u)(z) := \frac{1}{\sqrt{n}}\sum_{\substack{ad = n\\
    (a, M) = 1}}\psi(a) \sum_{b \modulo d} u\left(\frac{az + b}{d}\right).
\end{equation}
We use the abbreviated notation $T(n)$ when $M$ and $\psi$ are clear from the context. These operators preserve $\C_{it}(\Gamma_0(M), \psi)$, and the effect on the Fourier expansion is given by the formula
\begin{equation}\label{eq:heckeandfourierexpansion}
    \rho(T(n)u; m) =\sum_{\substack{a \mid (n, m)\\ (a, M) = 1}} \psi(a) \rho(u; m n/a^2),
\end{equation}
where $m \in \Z_{\geq 1}$. In particular, for any $n \geq 1$ we have
\begin{equation*}
    \rho(T(n)u; 1) = \rho(u; n).
\end{equation*}
We say that $f \in \C_{it}(\Gamma_0(M), \psi)$ is a \emph{newform} if it is an eigenfunction of $T(n; M, \psi)$ for all $n \geq 1$. In that case we say that $M$ is the \emph{conductor} of $f$, denoted $\q(f) = M$, and we also say that $\psi$ is the \emph{nebentypus} of $f$. By the strong multiplicity one property (see \cite{piatetski_shapiro_corvallis}) $f$ is also an eigenfunction of the involution $Q_{s, k}$ introduced in \cite[Proposition 4.5]{DFIsubconvexityartin}, with eigenvalue $\epsilon_f \in \{\pm 1\}$. It follows that the Fourier expansion of $f$ is 
\begin{equation}\label{eq:fourierexpansonnewform}
    f(z) = \rho(f; 1)\sum_{n = 1}^\infty\frac{\lambda_f(n)}{n^{1/2}}\left(W_{\frac{k}{2}, s-\frac{1}{2}}(4\pi n y)e(nx) +  \epsilon_f \frac{\Gamma\left(s + \frac{k}{2}\right)}{\Gamma\left(s - \frac{k}{2}\right)}W_{-\frac{k}{2}, s-\frac{1}{2}}(4\pi n y)e(-nx)\right),
\end{equation}
see \cite[equations (4.36), (4.70)]{DFIsubconvexityartin}.
Denote the set of normalized newforms of nebentypus $\psi$ and conductor $M$ by
\begin{equation}\label{eq:setofnormalizednewforms}
        \HH_{it}(M, \psi) := \{\text{newforms }f \in \C_{it}(\Gamma_0(N), \psi) \text{ such that }\rho_f(1) = 1\}.
\end{equation}
For a newform $f \in \HH_{it}(M_1, \psi)$, the oldspace of $f$ in $\Gamma_0(M_2)$ is defined by
\begin{equation}\label{eq:oldspaceGamma0}
    V_f(\Gamma_0(M_2)) := \spn\{z \mapsto f(d z) : d \mid (M_2/M_1)\} \subset \C_{it}(\Gamma_0(M), \psi).
\end{equation}
In particular, we set $V_f(\Gamma_0(M_2)) = \{0\}$ when $M_1 \nmid M_2$. The spanning set $\{z \mapsto f(dz): d\mid (M_2/M_1)\}$ is actually a basis of $V_f(\Gamma_0(M_2))$. The space $\C_{it}(\Gamma_0(M), \psi)$ has an orthogonal decomposition given by
\begin{equation*}
    \C_{it}(\Gamma_0(M), \psi) = \bigoplus_{d\mid M}\bigoplus_{f \in \HH_{it}(d, \psi)} V_f(\Gamma_0(M)).
\end{equation*}
The subspaces $V_f(\Gamma_0(M))$ can be characterized in terms of Hecke operators as follows: given $u \in \C_{it}(\Gamma_0(M), \psi)$, we have $u \in V_f(\Gamma_0(M))$ if and only if 
\begin{equation}\label{eq:characterizationoldspace}
    T(n; M, \psi)u = \lambda_f(n)u 
\end{equation}
holds whenever $(n, M) = 1$. This characterization follows from the strong multiplicity one property. 

Finally, let us recall the notion of \emph{twist} of a newform. Suppose given a Dirichlet character $\eta$, as well as a newform $f \in \HH_{it}(M, \psi)$. Then there exists a newform, denoted $f\otimes \eta$, such that 
\begin{equation}\label{eq:heckeeigenvaluetwist}
    \lambda_{f \otimes \eta}(n) = \lambda_f(n)\eta(n) 
\end{equation}
whenever $(n, \q(f)\q(\eta)) = 1$. Furthermore, we know that the nebentypus of $f\otimes \eta$ is $\eta^2 \psi$, and that 
\begin{equation}\label{eq:conductortwist}
    \q(f\otimes \eta) \, \text{ divides } \lcm(\q(f), \q(\eta)^2, \q(\eta)\q(\psi)).
\end{equation}

\subsection{Translation to $(\Gamma_d(q), \psi)$.} 
Since $\Gamma(q)\backslash \Gamma_d(q) \simeq (\Z/q\Z)^\times$ we have 
\begin{equation}\label{eq:orthogonaldecompositionintocharacters}
    \C_{it}(\Gamma(q)) = \bigoplus_{\substack{\psi \modulo q\\
    \psi(-1) = (-1)^k}} \C_{it}(\Gamma_d(q), \psi).
\end{equation}
Since $\Gamma_d(q) = \begin{psmallmatrix}
    q & 0\\
    0 & 1
\end{psmallmatrix}\Gamma_0(q^2)\begin{psmallmatrix}q & 0\\
    0 & 1
\end{psmallmatrix}^{-1}$, the remarks of the previous paragraph can be translated to the congruence character $(\Gamma_d(q), \psi)$. For convenience we let 
\begin{equation}\label{eq:specificsetofnewforms}
    \HH'_{it}(q) := \bigsqcup_{\substack{\psi \modulo q\\
\psi(-1) = (-1)^k}} \bigsqcup_{N \mid q^2} \HH_{it}(N, \psi)
\end{equation}
Given $f \in \HH'_{it}(q)$ we define
\begin{equation}\label{eq:oldspaceGammad}
    V_f(\Gamma_d(q)) := V_f(\Gamma_0(q^2))|_k \begin{psmallmatrix}
        q^{-1} & 0\\
        0 & 1
    \end{psmallmatrix} = \spn\{z \mapsto f(d z/q) : d \mid (q^2/\q(f))\}.
\end{equation}
We have an orthogonal decomposition
\begin{equation*}
    \C_{it}(\Gamma_d(q), \psi) = \bigoplus_{N \mid q^2}\bigoplus_{f \in \HH_{it}(N, \psi)} V_f(\Gamma_d(q)),
\end{equation*}
which combined with \eqref{eq:orthogonaldecompositionintocharacters} leads to 
\begin{equation}\label{eq:orthogonaldecompositionGamma}
    \C_{it}(\Gamma(q)) = \bigoplus_{f \in \HH'_{it}(q)}V_f(\Gamma_d(q)).
\end{equation}
Using \eqref{eq:conductortwist} one can check that if $f \in \HH'_{it}(q)$ and $\eta$ is a Dirichlet character modulo $q$, then $f\otimes \eta \in \HH'_{it}(q)$ as well. We let 
\begin{equation*}
    \T_f(q):= \{f\otimes \eta : \eta \text{ of level } q\}.
\end{equation*}
In the set $\HH'_{it}(q)$ we introduce the equivalence relation 
\begin{equation}\label{eq:equivalencerelationtwist}
    f \sim g \iff f = g \otimes \eta \text{ for some Dirichlet character }\eta \modulo q,
\end{equation} 
and denote by $\HH'_{it}(q)/\sim$ the set of equivalence classes under this relation. Thus, we have a decomposition
\begin{equation*}
    \HH'_{it}(q) = \bigsqcup_{[f] \in \HH'_{it}(q)/\sim } \T_f(q).
\end{equation*}
For $[f] \in \HH'_{it}(q)/\sim$ we define
\begin{equation}\label{eq:oldspacetwists}
    V_{[f]}(\Gamma_d(q)) := \bigoplus_{g \in \T_f(q)} V_g(\Gamma_d(q)).
\end{equation}
Then we have the orthogonal decomposition
\begin{equation}\label{eq:orthogonaldecompositionGammatwists}
    \C_{it}(\Gamma(q)) := \bigoplus_{[f] \in \HH'_{it}(q)/\sim} V_{[f]}(\Gamma_d(q)),
\end{equation}
the relevance of which will become clear in Proposition \ref{prop:decompositiontintoisotypictwists} below.

\subsection{Hecke operators as double cosets.}\label{subsec:heckeasdoublecoset}
For each semigroup $\Delta \subset \GL_2^+(\Q)$ which contains $\Gamma(q)$ we can construct the Hecke ring $R(\Gamma(q), \Delta)$, as in section 3.1 of \cite{shimurabook}. It has a $\Z$-basis consisting of the double cosets $\Gamma(q)\alpha\Gamma(q)$ as $\alpha$ ranges in $\Delta$, and the multiplication is given explicitly in equation (3.1.1) of \cite{shimurabook}. Very importantly, the ring $R(\Gamma(q), \Delta)$ acts linearly on $\C_{it}(\Gamma(q))$ on the right by the following recipe: given a decomposition $\Gamma(q)\alpha\Gamma(q) = \bigsqcup_{i} \Gamma(q)\alpha_i$ and given $u \in \C_{it}(\Gamma(q))$ we have  
\begin{equation}\label{eq:slashoperatordoublecosets}
    u |_k \Gamma(q)\alpha\Gamma(q) := \det(\alpha)^{-1/2}\sum_i u|_k \alpha_i.
\end{equation}
Since $\Gamma(q)$ is normal in $\Gamma$, we can simplify $\Gamma(q)\gamma \Gamma(q) = \Gamma(q)\gamma$ whenever $\gamma \in \Gamma$. It follows that we can characterize $\C_{it}(\Gamma, \chi)$ as
\begin{equation}\label{eq:formsasisotopycspace}
    \C_{it}(\Gamma, \chi) = \{u \in \C_{it}(\Gamma(q)): u|_k \Gamma(q) \gamma\Gamma(q) = \chi(\gamma)u \text{ for all }\gamma \in \Gamma\}.
\end{equation}
In words, $\C_{it}(\Gamma, \chi)$ is the $R(\Gamma(q), \Gamma)$-isotypic component of $\C_{it}(\Gamma(q))$ corresponding to $\chi$. In addition to $\Gamma$, the following two semigroups of $\GL_2^+(\Q)$ are relevant for us:
\begin{equation}\label{eq:defiofsemigroups}
\begin{aligned}
    \Delta'(q) & := \{\begin{psmallmatrix}
        a & b\\
        c & d
    \end{psmallmatrix} \in M_2(\Z) \cap \GL_2^+(\Q) : a = 1 \modulo q, b =  c = 0 \modulo q, (d, q) = 1\},\\
    \Delta(q) & := \{\alpha \in \Delta'(q) : \alpha = I_2 \modulo q\}.
\end{aligned}
\end{equation}
Note that $\Delta(q) \subset \Delta'(q)$. The structure of $R(\Gamma(q), \Delta'(q))$ is well-known. We collect a few observations, whose proofs can be found in \cite[Sections 3.3 and 3.5]{shimurabook}.

First of all, there is a bijection
\begin{equation}\label{eq:bijectioncoprimeheckeoperators}
    \{(a, d) \in \Z^2_{\geq 1}: a \mid d \text{ and }(d, q) = 1\} \longleftrightarrow \Gamma(q)\backslash \Delta'(q)/\Gamma(q)
\end{equation}
given by
\begin{equation}\label{eq:definitiongeneralheckeoperators}
    (a, d) \mapsto T'(a, d) := \Gamma(q) \sigma_a \begin{psmallmatrix}
        a & 0\\
        0 & d
    \end{psmallmatrix}\Gamma(q),
\end{equation}
where $\sigma_a$ is any matrix in $\SL_2(\Z)$ such that $\sigma_a = \begin{psmallmatrix}
    a^{-1} & 0\\
    0 & a
\end{psmallmatrix} \modulo q$.

The bijection from \eqref{eq:bijectioncoprimeheckeoperators} restricts to a bijection 
\begin{equation*}
    \{(a, d) \in \Z^2_{\geq 1}: a \mid d \text{ and }ad = 1 \modulo q\} \longleftrightarrow \Gamma(q)\backslash \Delta(q)/\Gamma(q),
\end{equation*}
which describes the $\Z$-basis of double cosets for $R(\Gamma(q), \Delta(q))$. 

For each prime $p \nmid q$, and $u \in \C_{it}(\Gamma_d(q), \psi)$ we have 
\begin{equation}\label{eq:actionbyscalarhecke}
    T'(p, p) u = p^{-1}\psi(p) u. 
\end{equation}
For $(n, q) = 1$ we let 
\begin{equation*}
    T'(n) := \sum_{\substack{ad = n\\
    a \mid d}}T'(a, d).
\end{equation*}

Then from Proposition 3.36 and Equation (3.5.6) in \cite{shimurabook}, it follows that 
\begin{equation}\label{eq:conjugationheckeoperators}
    T(n; q^2, \psi) \circ \left(\cdot  |_k \begin{psmallmatrix}
        q & 0\\
        0 & 1
    \end{psmallmatrix}\right) = \left(\cdot |_k \begin{psmallmatrix}
        q & 0\\
        0 & 1
    \end{psmallmatrix}\right) \circ T'(n) : \C_{it}(\Gamma_d(q), \psi)\rightarrow \C_{it}(\Gamma_0(q^2), \psi),
\end{equation}
so after a conjugation the operators $T'(n)$ are given by the familiar formula \eqref{eq:Heckeoperatorsexplicitformula}.\newline

\begin{lem}\label{lem:generatorsheckering}
We have the following facts:
\begin{enumerate}
    \item [i)] The ring $R(\Gamma(q), \Delta'(q))$ is a polynomial ring over $\Z$ on the variables $\{T'(p), T'(p, p)\}_{p \nmid q}$. 
    \item [ii)] The $\Q$-algebra $R(\Gamma(q), \Delta'(q)) \otimes_\Z \Q$ is a polynomial ring over $\Q$ on the variables $\{T'(p), T'(p^2)\}_{p \nmid q}$.   
    \item [iii)] The $\Q$-algebra $R(\Gamma(q), \Delta'(q)) \otimes_\Z \Q$ is a polynomial ring over $\Q$ on the variables $\{T'(1, p), T'(1, p^2)\}_{p \nmid q}$.   
\end{enumerate}
\end{lem}

\begin{proof}
Part i) follows from  \cite[Theorem 3.34]{shimurabook}. Note that we have the familiar multiplicativity relation 
\begin{equation}\label{eq:basicmultiplicativityrelation}
    T'(p)^2 = T'(p^2) + pT'(p, p),
\end{equation}
see \cite[Proposition 3.31]{shimurabook} and relation (2) in \cite[Theorem 3.24]{shimurabook}. Thus, part ii) follows from part i) and \eqref{eq:basicmultiplicativityrelation}. Since by definition, $T'(p) = T'(1, p)$ and $T'(p^2) = T'(1, p^2) + T'(p, p)$, it follows from \eqref{eq:basicmultiplicativityrelation} that 
\begin{equation}\label{eq:nebentypusintermsofeigenvalues}
    (p + 1)T'(p, p) = T'(1, p)^2 - T'(1, p^2).
\end{equation}
Part iii) follows from part i) and \eqref{eq:nebentypusintermsofeigenvalues}.
\end{proof}

\subsection{A decomposition into isotypic components.}

The characterization of the spaces $V_f(\Gamma_0(q^2))$ given in \eqref{eq:characterizationoldspace}, together with equation \eqref{eq:conjugationheckeoperators}, shows the following:
\begin{center} 
Equation \eqref{eq:orthogonaldecompositionGamma} is the decomposition of $\C_{it}(\Gamma(q))$ into $R(\Gamma(q), \Delta'(q))$-isotypic components. 
\end{center}

In particular, for each $f \in \HH'_{it}(q)$ and for each pair $(a,d)$ of positive integers such that $a\mid d$ and $(d, q) = 1$ there is a scalar $\lambda_f(a, d)$ such that 
\begin{equation}\label{eq:eigenvaluegeneralheckeoperator}
    T'(a, d)u = \lambda_f(a, d)u \quad \text{ for all }u \in V_f(\Gamma_d(q)).
\end{equation} 
What about the decomposition of $\C_{it}(\Gamma(q))$  into $R(\Gamma(q), \Delta(q))$-isotypic components? 
\begin{prop}\label{prop:decompositiontintoisotypictwists}
Equation \eqref{eq:orthogonaldecompositionGammatwists} is the decomposition of $\C_{it}(\Gamma(q))$ into $R(\Gamma(q), \Delta(q))$-isotypic components.
\end{prop}
For use in the proof of this proposition we take the following definition from \cite{linearindependenceLfunctions}.
\begin{defi}
We say that two multiplicative functions $A_1, A_2: \Z_{\geq 1} \rightarrow \CC$ are \emph{equivalent} if there exists $M \geq 1$ such that $A_1(n) = A_2(n)$ whenever $(n, M) = 1$. Otherwise, we say that they are \emph{nonequivalent}.
\end{defi}
We know that pairwise nonequivalent multiplicative functions $A_1, \ldots, A_n$ are linearly independent, see \cite[Theorem 2]{linearindependenceLfunctions}.
\begin{proof}[Proof of Proposition \ref{prop:decompositiontintoisotypictwists}.]
Since the double cosets $T'(a, d)$ with $ad = 1 \modulo q$ form a $\Z$-basis of $R(\Gamma(q), \Delta(q))$ the result can be restated explicitly as follows: given $f, g \in \HH'_{it}(q)$, we have 
\begin{equation}\label{eq:tobeshowndecompositionisotypictwists}
     f = g \otimes \eta \iff \lambda_f(a, d) = \lambda_g(a, d) \text{ whenever }ad = 1 \modulo q.
\end{equation} 
By part ii) from Lemma \ref{lem:generatorsheckering}, the elements $\{T'(p), T'(p^2)\}_{p \nmid q}$ generate $R(\Gamma(q), \Delta'(q)) \otimes_\Z \Q$. Thus, it follows from \eqref{eq:heckeeigenvaluetwist} and \eqref{eq:conjugationheckeoperators} that 
\begin{equation}\label{eq:twistinggeneralheckeeigenvalues}
\lambda_{f \otimes \eta}(a, d) = \eta(ad)\lambda_{f}(a, d)
\end{equation}
for any $f \in \HH_{it}'(q)$, any Dirichlet character $\eta$ modulo $q$ and any $(a, d) \in \Z^2_{\geq 1}$ such that $a\mid d$ and $(d, q) = 1$. This proves the forward implication in \eqref{eq:tobeshowndecompositionisotypictwists}.  

To treat the other direction, given $h \in \HH'_{it}(q)$ we introduce the multiplicative function $A_h$ given by 
\begin{equation*}
    A_h(n):= \delta_{\{(n, q) = 1\}}\lambda_h(1, n).
\end{equation*}
We first claim that the equivalence class of $A_h$ determines $h \in \HH'_{it}(q)$ uniquely. Indeed, we know from part iii) of Lemma \ref{lem:generatorsheckering} that the elements $\{T'(1, p), T'(1, p^2)\}_{p \nmid q}$ generate $R(\Gamma(q), \Delta'(q)) \otimes_\Z \Q$. Thus, if $A_{h_1}$ and $A_{h_2}$ are equivalent, we deduce that there exists $M\geq 1$ such that $\lambda_{h_1}(n) = \lambda_{h_2}(n)$ whenever $(n, qM) = 1$. By the strong multiplicity one property we conclude that $h_1 = h_2$, as desired.

We are now ready to finish the proof of the proposition. Let $f, g \in \HH'_{it}(q)$.  By \eqref{eq:twistinggeneralheckeeigenvalues} and orthogonality of characters we have  
\begin{equation*}
    \frac{1}{\varphi(q)}\sum_{\eta \modulo q} A_{f \otimes \eta}(n) = \delta_{\{n = 1 \modulo q\}} \lambda_f(1, n),
\end{equation*}
and similarly for $g$. Thus, if $f, g \in \HH'_{it}(q)$ are such that $\lambda_f(a, d) = \lambda_g(a, d)$ whenever $ad = 1 \modulo q$, then we have the linear relation
\begin{equation*}
    \sum_{\eta \modulo q} A_{f\otimes \eta} = \sum_{\eta \modulo q} A_{g\otimes \eta}.
\end{equation*}
By \cite[Theorem 2]{linearindependenceLfunctions}, it must be that $A_{f\otimes \eta_1}$ and $A_{g \otimes \eta_2}$ are equivalent for some Dirichlet characters $\eta_1, \eta_2$. By the previous paragraph this implies that $f = g\otimes \eta$ for $\eta = \eta_2 \overline{\eta_1}$, thus proving the backward implication in \eqref{eq:tobeshowndecompositionisotypictwists} and finishing the proof of Proposition \ref{prop:decompositiontintoisotypictwists}.
\end{proof}

We also have a very relevant commutativity property.
\begin{lem}\label{lem:DeltaqcentralizesSL2}
The subrings $R(\Gamma(q), \Delta(q))$ and $R(\Gamma(q), \SL_2(\Z))$ centralize each other inside $R(\Gamma(q), \GL_2^+(\Q))$. 
\end{lem}

\begin{proof}
Let $\alpha \in \Delta_q$ and $\beta \in \SL_2(\Z)$. Since $\Gamma(q)$ is normal in $\SL_2(\Z)$ we have 
\begin{equation}\label{auxeq:multdoublecosets}
    \Gamma(q) \alpha \Gamma(q) \cdot \Gamma(q)\beta \Gamma(q) = \Gamma(q) \alpha \beta \Gamma(q), \quad \Gamma(q) \beta \Gamma(q) \cdot \Gamma(q)\alpha \Gamma(q) = \Gamma(q) \beta \alpha \Gamma(q).
\end{equation}
By the definition of $\Delta_q$ we have 
\begin{equation*}
    \alpha \beta = \beta = \beta \alpha \modulo q.
\end{equation*}
Since $\beta \in \SL_2(\Z)$ we have
\begin{equation*}
    \SL_2(\Z)\alpha \beta \SL_2(\Z) = \SL_2(\Z)\alpha\SL_2(\Z) = \SL_2(\Z)\beta \alpha \SL_2(\Z).
\end{equation*}
Therefore, by part (2) of \cite[Lemma 3.29]{shimurabook} we deduce that 
\begin{equation*}
    \Gamma(q)\alpha \beta\Gamma(q) = \Gamma(q)\beta \alpha\Gamma(q).
\end{equation*}
Combined with \eqref{auxeq:multdoublecosets}, this concludes the proof of the lemma.
\end{proof}

Thus, we have a commutative ring $R(\Gamma(q), \Delta(q))$ whose action on $\C_{it}(\Gamma(q))$ commutes with the action of $R(\Gamma(q), \Gamma)$. If we combine \eqref{eq:formsasisotopycspace} with Proposition \ref{prop:decompositiontintoisotypictwists} we obtain the following corollary, which is the main result of this section. 
\begin{cor}\label{cor:basisaslinearcombinationoftwists}
Let $(\Gamma, \chi)$ be a congruence pair of level $q$. Then, we have
\begin{equation*}
    \C_{it}(\Gamma, \chi) = \bigoplus_{[f] \in \HH'_{it}(q)/\sim}\left(V_{[f]}(\Gamma_d(q)) \cap \C_{it}(\Gamma, \chi)\right).
\end{equation*}   
\end{cor}

\section{Lindelöf bound for the spectral average of Rankin-Selberg convolutions.}\label{sec:rankinselbergandadjointlfunction}
Corollary \ref{cor:basisaslinearcombinationoftwists} guarantees the existence of an orthonormal basis $\{u_j\}_{j \geq 1}$ of $L_0^2(\Gamma, \chi)$ with the following properties:
\begin{enumerate}
    \item [1)] Each $u_j$ is an eigenfunction of the Laplacian. That is, we have $P u_j = \lambda_j u_j$ for some $\lambda_j > 0$. 
    \item [2)] The eigenvalues $\lambda_j$ are in increasing order. We write $\lambda_j = 1/4 + t_j^2$ with the convention that $t_j > 0$ or $\text{Im}(t_j) \geq 0$.    
    \item [3)] For each $j \geq 1$ there exists a newform $f_j$ such that $u_j \in V_{[f_j]}(\Gamma_d(q))$.
    \item [4)] The map $j \mapsto [f_j] \in \HH'_{it_j}(q)/\sim$ is $O_q(1)$-to-one.
\end{enumerate}
Part $4)$ follows from the fact that $\dim_\CC V_{[f]}(\Gamma_d(q)) \leq \varphi(q)\tau(q^2)$ for any $f \in \HH'_{it}(q)$, recall \eqref{eq:oldspaceGammad} and \eqref{eq:oldspacetwists}.

From the orthogonal decomposition in \eqref{eq:oldspacetwists} we see that, for each $j \geq 1$, we can write uniquely
\begin{equation}\label{eq:decompositionmaassformintocomponents}
    u_j = \sum_{g \in \T_{f_j}(q)} u_j^g, \quad  \text{ where }\quad u_j^g \in V_g(\Gamma_d(q)).
\end{equation} 
To estimate $\S_0(T;X)$ and derive Theorem \ref{thm:luosarnakbound} we will apply the Bruggeman-Kuznetsov trace formula from Proposition \ref{prop:forwardkuznetsov} for different values of $n = m$ and then remove the weights $|\rho_j(n)|^2$. To remove these weights we need to consider the \emph{naive} Rankin-Selberg $L$-series
\begin{equation}\label{eq:naiverankinselberg}
    \L(s; |u_j|^2) := \sum_{n \in \frac{1}{q}\Z_{\geq 1}} \frac{|\rho_j(n)|^2}{n^s}.
\end{equation}

The purpose of this section is to prove the following result. 
\begin{thm}\label{thm:informationaboutrankinselberg}
Let $\{u_j\}_{j \geq 1}$ be a basis of $L_0^2(\Gamma, \chi)$ with the properties above. For each $j\geq 1$ the series $\L(s; |u_j|^2)$ has meromorphic continuation to $\CC$. On the region $\text{Re}(s) \geq 1/2$ the unique singularity is a simple pole at $s = 1$ of residue $\omega(t)^{-1}$. Furthermore, there is $A > 0$ such that for all $T \geq 1$ we have 
\begin{equation}\label{eq:lindelofonaverage}
   \sum_{|t_j| \leq T} \omega(t_j)^2|\L(s; |u_j|^2)|^2 \ll_{q, \varepsilon} T^{2 + \varepsilon}|s|^{A},  
\end{equation}
on the critical line $\text{Re}(s) = 1/2$. 
\end{thm}

More generally, for $v_1, v_2 \in L_0^2(\Gamma_d(q))$ we let 
\begin{equation*}
    \L(s; v_1, \overline{v_2}) := \sum_{n \in \frac{1}{q}\Z_{\geq 1}} \frac{\rho(v_1; n) \overline{\rho(v_2; n)}}{n^s}.
\end{equation*}
By \eqref{eq:decompositionmaassformintocomponents} we can write 
\begin{equation}\label{eq:naiverankinselbergdecomposition}
\L(s; |u_j|^2) = \sum_{g_1, g_2 \in \T_{f_j}(q)} \L(s; u_j^{g_1}, \overline{u_j^{g_2}}).
\end{equation}
Also, by \eqref{eq:heckeandfourierexpansion} and \eqref{eq:characterizationoldspace} we see that, for $g_i \in \T_{f_j}(q)$, we have
\begin{equation*}
    \rho(u_j^{g_i}; n_1 n_2/q) = \lambda_{g_i}(n_1)\rho(u_j^{g_i}; n_2/q),
\end{equation*}
where $n_1, n_2 \in \Z_{\geq 1}$ are such that $(n_1, q) = 1$ and $n_2 \mid q^\infty$. Thus, we can write 
\begin{equation}\label{eq:structuralformularankinselberg}
    \L(s; u_j^{g_1}, \overline{u_j^{g_2}}) = L^{(q)}(s; g_1 \otimes \overline{g_2})\L_q(s; u_j^{g_1},  \overline{u_j^{g_2}}),
\end{equation}
where 
\begin{equation}\label{eq:partialproductrankinselberg}
    L^{(q)}(s;  g_1 \otimes \overline{g_2}) := \sum_{(m, q) = 1}\frac{\lambda_{g_1}(m)\overline{\lambda_{g_2}(m)}}{m^s},
\end{equation}
and 
\begin{equation}\label{eq:definitionofbadfactorsrankinselberg}
    \L_q(s; u_j^{g_1},  \overline{u_j^{g_2}}) := q^s \sum_{m \mid q^\infty} \frac{\rho(u_j^{g_1}; m/q)\overline{\rho(u_j^{g_2}; m/q)}}{m^s}.
\end{equation}
In the following paragraphs we will collect information about $\L(s; u_j^{g_1}, \overline{u_j^{g_2}})$. We start by computing its residue at $s = 1$. 

\subsection{Analytic continuation and residue at $s = 1$.}
Let $g_1, g_2 \in \HH'_{it}(q)$ and $v_i \in V_{g_i}(\Gamma_d(q))$ for $i \in \{1, 2\}$. 
One can check from the expansion \eqref{eq:fourierexpansonnewform} and the construction of the space $V_{g_i}(\Gamma_d(q))$ that 
\begin{equation}\label{auxeq:negativefouriercoefficient}
    \rho(v_i; -n/q) = \epsilon_{g_i} \frac{\Gamma(it_j + (1 + k)/2)}{\Gamma(it_j + (1 - k)/2)} \rho(v_i; n/q)
\end{equation}
for any $n \geq 1$. Let $E_\infty(z; s) = E_\infty(z; s; \Gamma_d(q), \mathbbm{1}_{\Gamma_d(q)})$ denote the Eisenstein series of weight $0$ attached to the cusp $\infty$ of $\Gamma_d(q)$. Note that this cusp has width $q$, so $\sigma_\infty = \begin{psmallmatrix}
q^{1/2}& 0\\
0 & q^{-1/2}
\end{psmallmatrix}$. Using \eqref{auxeq:negativefouriercoefficient}, the Fourier expansion of $v_i$ from \eqref{eq:generalfourierexpansion} and the unfolding method one can check that  
\begin{equation}\label{eq:unfoldingeisenstein}
    \int_{\Gamma_d(q)\backslash \Hyp} v_1(z)\overline{v_2(z)} E_\infty(z; s) \, d\mu(z) = q^{1-s}\L(s; v_1, \overline{v_2}) \times I(s, it_j, \varepsilon_{g_1}\varepsilon_{g_2}), 
\end{equation}
where we define
\begin{equation*}
    I(s, \beta, \delta) := \int_0^\infty \left(W_{\frac{k}{2}, \beta}^2 (4\pi y) + \delta\left|\frac{\Gamma\left(\beta + \frac{1 + k}{2}\right)}{\Gamma\left(\beta + \frac{1- k}{2}\right)}\right|^2 W^2_{-\frac{k}{2}, \beta}(4\pi y)\right)y^{s-1}\frac{dy}{y},
\end{equation*}
with $\delta \in \{1, -1\}$. Equation \eqref{eq:unfoldingeisenstein} shows the claim about meromorphic continuation in Theorem \ref{thm:informationaboutrankinselberg}. Recall that $E_\infty(z; s)$ has a simple pole at $s = 1$ of residue $\vol(\Gamma_d(q)\backslash \Hyp)^{-1}$. Since the spaces $V_{g_1}(\Gamma_d(q))$ and $V_{g_2}(\Gamma_d(q))$ are orthogonal unless $g_1 = g_2$, it follows that $\L(s; v_1, \overline{v_2})$ is holomorphic at $s = 1$ unless $g_1 = g_2$. In the last case, $\L(s; v_1, \overline{v_2})$ may have a simple pole at $s = 1$ whose residue we are going to compute. 

Suppose that $g_1 = g_2$. Using the computation from \cite[Lemma 8.2]{DFIsubconvexityartin} (including the corrections from \cite[Section 12]{youngeisensteinseries}), together with the first Barnes lemma \cite[\href{https://dlmf.nist.gov/5.13.E3}{(5.13.3)}]{dlmf}, one can deduce that 
\begin{equation}\label{eq:archimedeanrankinselbergexplicitformula}
    I(s, \beta, 1) = \pi^{-s}\frac{\Gamma^2\left(\frac{s}{2}\right)\Gamma\left(\frac{s + k}{2} + \beta\right)\Gamma\left(\frac{s + k}{2} - \beta\right)}{\Gamma(s)}.
\end{equation}
In particular, we have $I(1, it, 1) = \omega(t)$, where $\omega(t)$ was introduced in \eqref{eq:weightskuznetsov}. It follows that
\begin{equation}\label{eq:residuenaiverankinselbergproductlseries}
    \text{Res}_{s = 1}\L(s; v_1, \overline{v_2}) = \frac{\langle v_1, v_2 \rangle}{\omega(t_j)}
\end{equation}
when $g_1 = g_2$. We let $g \in \T_{f_j}(q)$ and use this equality when $v_1 = v_2 = u_j^g$. By \eqref{eq:decompositionmaassformintocomponents} and orthogonality we have 
\begin{equation*}
    \langle u_j, u_j \rangle  = \sum_{g \in \T_{f_j}(q)} \langle u_j^{g}, u_j^g \rangle,
\end{equation*}
so we can deduce that
\begin{equation}\label{eq:computationresiduerankinselberg}
    \text{Res}_{s = 1}\L(s; |u_j|^2) = \frac{1}{\omega(t_j)},
\end{equation}
finishing the first part of the proof of Theorem \ref{thm:informationaboutrankinselberg}. We now start working towards the proof of the inequality in \eqref{eq:lindelofonaverage}.

\subsection{Good primes and adjoint $L$-function.}
Let $g_1, g_2 \in \HH'_{it}(q)$ be newforms, of nebentypus $\psi_1$ and $\psi_2$ respectively. Recall that these normalized newforms correspond to automorphic cuspidal representations $\pi_1, \pi_2$ of $\GL_2/\Q$, of conductor dividing $q^2$. For a prime $p \nmid q$ write
\begin{equation*}
1 - \lambda_{g_i}(p)X + \psi_i(p)X^2 = (1 - \alpha_{i, p}X)(1 - \beta_{i, p}X).
\end{equation*}
We know that there is an $L$-function 
\begin{equation*}
    L(s; g_1 \otimes \overline{g_2}) = \prod_p L_p(s; g_1 \otimes \overline{g_2}),
\end{equation*}
called the Rankin-Selberg convolution of $g_1$ and $g_2$, with the following properties:
\begin{itemize}
    \item  For primes $p$ such that $(p, q) = 1$ we have the explicit formula
\begin{equation*}
    L_p(s; g_1 \otimes \overline{g_2}) = \frac{1}{(1 - \alpha_1 \overline{\beta_1}p^{-s})(1 - \alpha_1 \overline{\beta_2}p^{-s})(1 - \alpha_2 \overline{\beta_1}p^{-s})(1 - \alpha_2 \overline{\beta_2}p^{-s})}.
\end{equation*}
    \item $L(s; g_1 \otimes \overline{g_2})$ has meromorphic continuation to the complex plane, with at most a simple pole at $s = 1$.
    \item In general, for any prime $p$ we have an expression 
\begin{equation*}
    L_p(s; g_1 \otimes \overline{g_2}) = \frac{1}{(1-\alpha_{1, p} p^{-s})(1-\alpha_{2, p} p^{-s})(1 - \alpha_{3, p} p^{-s})(1 - \alpha_{4, p} p^{-s})},
\end{equation*}
where, by the second appendix to \cite{kimramakrishnansarnak2003} we know that $\max(|\alpha_{i, p}|, |\alpha_{i, p}|^{-1}) \leq p^{7/32}$. 
\end{itemize} 
The precise computation of $L(s; g_1\otimes\overline{g_2})$ at primes $p \mid q$ depends on the local $p$-adic component of the automorphic representations $\pi_1, \pi_2$, see \cite{jacquetrankinselberg}, as well as Sections 1 and 2 in \cite{gelbartjacquetlift}. We warn the reader that the convention in those sources is that $L(s; g_1\otimes\overline{g_2})$ is the completed $L$-function, whereas for us it is just the product over finite primes. One can check that the Dirichlet series $L^{(q)}(s; g_1 \otimes \overline{g_2})$ in \eqref{eq:partialproductrankinselberg} is given by the product 
\begin{equation*}
    L^{(q)}(s; g_1 \otimes \overline{g_2}) = \prod_{p \nmid q} L_p(s; g_1 \otimes \overline{g_2}),
\end{equation*}
so that in particular we have
\begin{equation}\label{eq:rankinselbergisalmosteulerproduct}
    q^{-\varepsilon_1}|L(s; g_1 \otimes \overline{g_2})| \ll_{\varepsilon_1, \varepsilon_2} |L^{(q)}(s; g_1 \otimes \overline{g_2})| \ll_{\varepsilon_1, \varepsilon_2} q^{\varepsilon_1} |L(s; g_1 \otimes \overline{g_2})|
\end{equation}
on the half-plane $\text{Re}(s) \geq 7/32 + \varepsilon_2$, for any $\varepsilon_1, \varepsilon_2 > 0$.

Suppose that $g_1 = g_2 \otimes \psi$, where $\psi$ is a Dirichlet character modulo $q$. By work of Gelbart and Jacquet \cite{gelbartjacquetlift}, which generalizes previous work of Shimura \cite{shimuraholomorphy}, we know that 
\begin{equation}\label{eq:gelbartjacquetlift}
    L(s; g_1 \otimes \overline{g_2}) = L(s; \psi)L(s; \Ad(f_j)\otimes \psi),
\end{equation}
where $\Ad(f_j)$ is the adjoint lift of $f_j$ from $\GL(2)$ to $\GL(3)$.
Let $\delta \in \{0, 1\}$ be such that $\psi(-1) = (-1)^\delta$ and let $\epsilon \in \{0, 1\}$ be such that $(-1)^\epsilon = (-1)^{\delta + k}$. Let us write
\begin{equation}\label{eq:gammafactorinfinity}
    \gamma(s; \Ad(f_j)\otimes \psi):= \pi^{-\frac{3s}{2} - \epsilon - \frac{\delta}{2}} \Gamma\left(\frac{s + \delta}{2}\right)\Gamma\left(\frac{s + \epsilon}{2} + it_j\right)\Gamma\left(\frac{s + \epsilon}{2} - it_j\right).
\end{equation}
If we define 
\begin{equation}\label{eq:completedadjointlfunction}
    \Lambda(s; \Ad(f_j)\otimes \psi) := \gamma(s; \Ad(f_j)\otimes \psi)L(s; \Ad(f_j)\otimes \psi),
\end{equation}
then we have the functional equation
\begin{equation}\label{eq:functionalequationadjoint}
    \Lambda(s; \Ad(f_j)\otimes \psi) = W(\Ad(f_j)\otimes \psi)\q(\Ad(f_j)\otimes \psi)^{1/2 - s} \Lambda(1-s; \Ad(f_j)\otimes \overline{\psi}),
\end{equation}
where $|W(\Ad(f_j)\otimes \psi)| = 1$ and the conductor $\q(\Ad(f_j)\otimes \psi)$ divides $q^4$. Furthermore, the completed $L$-function $\Lambda(s; \Ad(f_j)\otimes \psi)$ is entire unless the following condition holds: there exists a Dirichlet character $\eta \neq 1$ such that $f_j \otimes \eta = f_j$ \cite[Theorem 9.3]{gelbartjacquetlift}. If this condition holds we say that the newform $f_j$ is \emph{exceptional}. Otherwise, we say that $f_j$ is \emph{regular}. Accordingly, we write 
\begin{equation*}
    \HH'_{it_j}(q) = \HH^{\text{reg}}_{it_j}(q) \bigsqcup \HH^{\text{excep}}_{it_j}(q)
\end{equation*}
where
\begin{equation*}
    \HH^{\text{excep}}_{it_j}(q) = \{f \in \HH'_{it_j}(q) : f  = f\otimes \eta\text{ for some } \eta \neq 1\}.
\end{equation*}
Denote the collection of newforms of spectral parameter $|t_j|\leq T$ by
\begin{equation*}
    \HH'(q; T) = \bigcup_{|t_j| \leq T} \HH'_{it_j}(q),
\end{equation*}
and define similarly $\HH^{\text{reg}}(q; T)$ and $\HH^{\text{excep}}(q; T)$. We will treat separately the contributions to the inequality in \eqref{eq:lindelofonaverage} of these two types of newforms.

\subsubsection{Norm of arithmetically normalized newforms.}
Recall that the newform $f_j \in \HH'_{it_j}(q)$ is normalized so that $\rho_{f_j}(1) = 1$, see \eqref{eq:setofnormalizednewforms} and \eqref{eq:specificsetofnewforms}. Let $v_1 = v_2 = v$, given by $v := f_j|_k \begin{psmallmatrix}
    q^{-1} & 0\\
    0 & 1
\end{psmallmatrix}$. Note that $v \in V_{f_j}(\Gamma_d(q))$. Also, in this case we have $\rho(v; n/q) = q^{-1/2}\lambda_{f_j}(n)$ for $n \geq 1$. Thus, 
\begin{equation*}
\L(s; v, \overline{v}) = q^{s-1}\sum_{n = 1}^\infty \frac{|\lambda_{f_j}(n)|^2}{n^s}.
\end{equation*}
If we apply \eqref{eq:residuenaiverankinselbergproductlseries} to $v_1 = v_2 = v$ we deduce that 
\begin{equation}\label{eq:newauxeq1}
    \text{Res}_{s = 1}\sum_{n = 1}^\infty \frac{|\lambda_{f_j}(n)|^2}{n^s} = \frac{\langle f_j, f_j \rangle}{\omega(t_j)}.
\end{equation}
By multiplicativity of Hecke eigenvalues we have 
\begin{equation}\label{eq:newauxeq2}
    \sum_{n = 1}^\infty \frac{|\lambda_{f_j}(n)|^2}{n^s} = \left(\sum_{n \mid q^\infty} \frac{|\lambda_{f_j}(n)|^2}{n^s}\right) L^{(q)}(s; f_j \otimes \overline{f_j}).
\end{equation} 
By the bound $|\lambda_{f_j}(n)| \ll_\varepsilon n^{7/64 + \varepsilon}$ from the second appendix to \cite{kimramakrishnansarnak2003}, we know that
\begin{equation}\label{eq:upperboundbadfactornewform}
    \sum_{n \mid q^\infty} \frac{|\lambda_{f_j}(n)|^2}{n^s} \ll_{\varepsilon} q^\varepsilon
\end{equation}
on $\text{Re}(s) \geq 7/32 + \varepsilon$. Combining \eqref{eq:newauxeq1}, \eqref{eq:newauxeq2} and \eqref{eq:upperboundbadfactornewform} with \eqref{eq:rankinselbergisalmosteulerproduct} and \eqref{eq:gelbartjacquetlift} we obtain 
\begin{equation}\label{eq:newauxeq3}
    q^{-\varepsilon} \frac{\langle f_j, f_j \rangle}{\omega(t_j)} \ll_\varepsilon L(1; \Ad(f_j)) \ll_\varepsilon q^{\varepsilon}\frac{\langle f_j, f_j \rangle}{\omega(t_j)}
\end{equation}
for any $\varepsilon > 0$. Also, by \cite{hoffsteinlockhart} we know that  
\begin{equation}\label{eq:newauxeq4}
    ((1 +|t_j|)q)^{-\varepsilon} \ll_\varepsilon L(1; \Ad(f_j)) \ll_\varepsilon  ((1 +|t_j|)q)^{\varepsilon},
\end{equation}
the upper bound having been established previously by Iwaniec \cite{iwaniecsmalleigenvalues} in the case $k = 0$. Combining \eqref{eq:newauxeq3} and \eqref{eq:newauxeq4} we arrive at the following lemma.
\begin{lem}\label{lem:normnormalizednewform}
For $j \geq 1$ and any $\varepsilon > 0$ we have 
\begin{equation*}
    ((1 +|t_j|)q)^{-\varepsilon} \ll_\varepsilon \frac{\langle f_j, f_j \rangle}{\omega(t_j)} \ll_\varepsilon((1 +|t_j|)q)^{\varepsilon}.
\end{equation*}
\end{lem}

\subsubsection{Contribution from exceptional newforms.}
 By Weyl's law \eqref{eq:basicweyllaw} we have 
\begin{equation*}
    |\HH'(q; T)| \asymp_q T^2.
\end{equation*}
On the other hand, by the following lemma we know that the subfamily $\HH^{\text{excep}}(q; T)$ is quite thin.
\begin{lem}\label{lem:exceptionalnewformsaresparse}
For any $T \geq 1$ we have 
\begin{equation}\label{eq:boundexceptionalnewforms}
    |\HH^{\text{excep}}(q; T)| = O_q(T).
\end{equation}
\end{lem}
\begin{proof}
The idea is that exceptional newforms are parametrized by Hecke Größencharacters of quadratic extensions of $\Q$. Let $f \in \HH^{\text{excep}}(q; T)$ and consider $\eta \neq 1$ such that $f \otimes \eta = f$. By elementary considerations $\eta$ is a  quadratic Dirichlet character of conductor dividing a power of $q$. Fix one of the finitely many possibilities for $\eta$ and let $H$ be the quadratic extension of $\Q$ attached to it. This means that if $D_H$ is the discriminant of the quadratic extension we have 
\begin{equation*}
    |D_{H}| = \q(\eta), \quad \text{and}\quad \sgn(D_{H}) = \eta(-1).
\end{equation*}
Let $\A_{H}^\times$ denote the group of ideles over $H$. Recall that $\A_H^\times$ contains $H^\times$ embedded diagonally as a discrete subgroup. By work of Labesse and Langlands \cite[Proposition 6.5]{Labesse_Langlands_1979} we know that there exists a unitary Hecke Größencharacter $\theta$ of $\A_{H}^\times/H^\times$ such that the standard $L$-functions attached to $f$ and $\theta$ agree. We write $f = f(\theta)$ in this case. By comparing conductors of the respective $L$-functions we see that if $f = f(\theta)$ for $f \in \HH'_{it_j}(q)$, then $\q(\theta) \mid q^2$. Let us examine the archimedean components of $\theta$. First note that the gamma factor of the $L$-function attached to $f \in \HH'_{it_j}(q)$ is of the form 
\begin{equation}\label{eq:gammafactorinfinityoff}
    \Gamma_\R(s + it_j + k_1)\Gamma_\R(s - it_j + k_2),
\end{equation}
where $\Gamma_\R(s) := \pi^{-s/2}\Gamma(s/2)$ and $k_1, k_2 \in \{0, 1\}$ are such that $k_1 + k_2 = k \modulo 2$. 
We need to distinguish two cases.

Suppose first that the quadratic extension $H$ is imaginary. There is only one infinite place of $H$, which is complex. Thus, the archimedean component $\theta_\infty$ of $\theta$ is a unitary character of $\CC^\times$. There exist $r \in \R$ and $m \in \Z$ such that 
\begin{equation*}
    \theta_\infty(z) = |z|^{ir} (z/|z|)^m, \text{ for all }z \in \CC.
\end{equation*}
The gamma factor of $\theta$ is 
\begin{equation*}
    \Gamma_\CC(s + ir/2 + |m|/2),
\end{equation*}
where 
\begin{equation*}
    \Gamma_\CC(s) := 2 (2\pi)^{-s}\Gamma(s).
\end{equation*}
Comparing with \eqref{eq:gammafactorinfinityoff} we see that $f = f(\theta)$ is only possible if $t_j = r = m = 0$ and $k = 1$. By the finiteness of the class number, the number of characters $\theta$ of $H$ with fixed archimedean component and fixed conductor is bounded independently of the archimedean component. It follows that  
    \begin{equation*}
        |\{f \in \HH^{\text{excep}}(q; T): f = f(\theta) \text{ for some }\theta \text{ on } \A_{H}^\times/H^\times\}| = O_q(1)
    \end{equation*} 
in this case.

Suppose otherwise that $H$ is a totally real quadratic extension. Then there are two real places $v_1, v_2$, and we let $\theta_1, \theta_2$ be the components of $\theta$ at these places. Thus, $\theta_1, \theta_2$ are unitary characters on $\R^\times$. Then, for $i \in \{1, 2\}$ there exist $r_i \in \R$ and $m_i \in \{0, 1\}$ such that 
    \begin{equation*}
        \theta_i(x) = |x|^{ir_i} (x/|x|)^{m_i} \quad \text{ for all }x \in \R^\times. 
    \end{equation*}
    The gamma factor of the $L$-function attached to $\theta$ is given by 
    \begin{equation*}
        \Gamma_\R(s + ir_1 + m_1)\Gamma_\R(s + ir_2 + m_2). 
    \end{equation*}
Comparing with \eqref{eq:gammafactorinfinityoff} we deduce that $r_1 = -r_2 = \pm t_j$. Let $\O_H$ denote the ring of integers of $H$, let $\O^\times_{H}$ denote the group of units of $\O_H$ and define 
    \begin{equation*}
        \O_H^\times(q^2) := \{u \in \O_H: u = 1 \modulo q^2\}.
    \end{equation*}
    Let $\iota_i$ denote the embedding $\iota_i : H \hookrightarrow \R$ corresponding to the place $v_i$.  
    Since the conductor of $\theta$ divides $q^2$, we have 
    \begin{equation}\label{eq:auxeqarchimedeanfactor}
        \theta_1(\iota_1(u))\theta_2(\iota_2(u)) = 1 \text{ for all }u \in \O_H^\times(q^2). 
    \end{equation}
    By Dirichlet's unit theorem we know that $\O_H^\times(q^2)$ is finitely generated of rank one. Let $u \in \O_H^\times(q^2)$ be an element which is not a root of unity, and define $a_i := \log |\iota_i(u)|$, so that $a_i \neq 0$. Note that $a_1 = -a_2$ since $u$ is a unit. Recall that we know that $r_1 = -r_2$. From \eqref{eq:auxeqarchimedeanfactor} we deduce that 
    \begin{equation*}
        r_1 \in \frac{\pi}{2a_1}\Z.
    \end{equation*}  
Thus, if $f = f(\theta) \in \HH'_{it_j}(q)$ for some Größencharacter $\theta$ on $H^\times \backslash \A^\times_H$, then $t_j$ lies in an arithmetic progression. Also, by the finiteness of the class number, the number of characters $\theta$ of $H$ with fixed archimedean component and fixed conductor is bounded independently of the archimedean component. It follows that  
    \begin{equation*}
        |\{f \in \HH^{\text{excep}}(q; T): f = f(\theta) \text{ for some }\theta \text{ on } \A_{H}^\times/H^\times\}| = O_q(T)
    \end{equation*} 
in this case. Since there are only finitely many possibilities for $H$, the proof of the lemma is finished. 
\end{proof}

Let $f \in \HH_{it_j}^{\text{excep}}(q)$. Note that the archimedean analytic conductor of $L(1/2 + ir; \Ad(f)\otimes \psi)$ is $\asymp (1 + |r|)(1 + |r - 2t_j|)(1 + |r + 2t_j|)$. By work of Li \cite{liupperboundedge} we know good upper bounds for $L$-functions at the edge of the critical strip, in particular we have
\begin{equation*}
    L(1 + \varepsilon_2 + ir; \Ad(f)\otimes \psi) \ll_{q, \varepsilon_1, \varepsilon_2} ((1 + |r|)(1 + |r - t|)(1 + |r + t|))^{\varepsilon_1}
\end{equation*}
for any $\varepsilon_1, \varepsilon_2 > 0$. It follows by the Phrägmen--Lindelöf principle that
\begin{equation*}
    L(s; \Ad(f)\otimes \psi) \ll_q (1 + |t|)^{1/2 + \varepsilon}(1 + |s|)^A,
\end{equation*}
for some $A > 0$ and $\text{Re}(s) = 1/2$. When we combine this estimate with Lemma \ref{lem:exceptionalnewformsaresparse} we obtain the following corollary.

\begin{cor}\label{cor:sumLfunctionsexceptional}
There is $A > 0$ such that for any $T \geq 1$ and any $\psi \modulo q$ we have
\begin{equation*}
    \sum_{f_j \in \HH^{\text{excep}}(q; T)} |L(s; \Ad(f_j)\otimes \psi)|^2 \ll_q T^{2 + \varepsilon} (1 + |s|)^A, 
\end{equation*}
for $\text{Re}(s) = 1/2$.
\end{cor}

\subsubsection{Regular newforms.}
The purpose of this section is to prove the following estimate, which complements Corollary \ref{cor:sumLfunctionsexceptional}. 
\begin{prop}\label{prop:lindelofforregularnewforms}
There is $A > 0$ such that for any $T \geq 1$ and any $\psi \modulo q$ we have
 \begin{equation*}
    \sum_{f_j \in \HH^{\text{reg}}(q; T)} |L(s; \Ad(f_j)\otimes \psi)|^2 \ll_{\varepsilon, q} T^{2+\varepsilon} (1 + |s|)^A
\end{equation*}
whenever $\text{Re}(s) = 1/2$.   
\end{prop}  
We will follow the original argument of \cite[Pages 220-221]{luosarnak}. We let $t_j$ be the spectral parameter of $f_j$, so that $f_j \in \HH^{\text{reg}}_{it_j}(q)$. We can break up the sum over $f_j$ into dyadic pieces with respect to the spectral parameter and assume that $T \geq 100$ and that $T \leq t_j \leq 2T$. Also, since the conductor $\q(\Ad(f_j) \otimes \psi)$ is a positive divisor of $q^4$, it can take only finitely many values. Thus, by a further partition of the family we can assume that $\q(\Ad(f_j) \otimes \psi)$ is a fixed divisor of $q^4$, which we denote simply by $\q$. Finally, it is enough to consider the case $|s| \leq T^{1-\varepsilon}$, since we can use the convexity bound on the complementary range $|s| \geq T^{1 - \varepsilon}$. 

For concinesses we abbreviate $F_j := \Ad(f_j)\otimes \psi$ and write $L(s; F_j) = \sum_{n = 1}^\infty \lambda_{F_j}(n)/n^s$.  By the assumption that $f_j$ is regular we know that $L(s; F_j)$ is an entire $L$-function. Also, we know that it satisfies the functional equation \eqref{eq:functionalequationadjoint}. Therefore, we can use the approximate functional equation \cite[Theorem 5.3]{iwanieckowalski} with the function $G(u) := e^{-u^2}$ and deduce that, for any $X > 1$, we have 
\begin{equation}\label{eq:approxfunctionaleq}
    L(s; F_j) = \sum_{n = 1}^\infty \frac{\lambda_{F_j}(n)}{n^s}V_s\left(\frac{n}{X\sqrt{\q}}; F_j\right) + \epsilon(F_j, s)\sum_{n = 1}^\infty \frac{\overline{\lambda_F(n)}}{n^{1-s}}V_{1-s}\left(\frac{nX}{\sqrt{\q}}; F_j\right),
\end{equation}
where $\text{Re}(s) = 1/2$, $|\epsilon(F_j, s)| = 1$ and 
\begin{equation}\label{eq:definitionweightfunctionapproxfunctionalequation}
    V_s(y; F_j):= \frac{1}{2\pi i}\int_{(3)}y^{-u}G(u)\frac{\gamma(s + u; F_j)}{\gamma(s; F_j)}\frac{du}{u}.
\end{equation}
The gamma factor $\gamma(s; F_j)$ was defined in \eqref{eq:gammafactorinfinity}. We will only need to apply \eqref{eq:approxfunctionaleq} for values $1/2 \leq X \leq 2$. Note that with our assumptions on the size of $s$ and $t_j$, the archimedean analytic conductor of $L(s; F_j)$ is comparable to $|s|T^2$. By \cite[Proposition 5.4]{iwanieckowalski} we have 
\begin{equation*}
     V_s(y; F_j) \ll_{B} \left(1 + \frac{y}{|s|^{1/2}T}\right)^{-B}
\end{equation*}
for any $B > 0$. Recall also that $\lambda_{F_j}(n) \ll_\varepsilon n^{7/32 + \varepsilon}$ (any polynomial bound would work here). It follows that 
\begin{equation}\label{eq:truncatedapproximatefunctionalequation}
    L(s; F_j) = \sum_{n \leq (\sqrt{\q|s|}T)^{1 + \varepsilon}} \frac{\lambda_{F_j}(n)}{n^s}V_s\left(\frac{n}{X\sqrt{\q}}; F_j\right) + \epsilon(F_j, s)\sum_{n \leq (\sqrt{\q|s|}T)^{1 + \varepsilon}} \frac{\overline{\lambda_F(n)}}{n^{1-s}}V_{1-s}\left(\frac{nX}{\sqrt{\q}}; F_j\right) + O_{\varepsilon, B}(T^{-B}),
\end{equation}
for any $B > 0$. If we use the very rapid decay of $G(u) = e^{-u^2}$ in Equation \eqref{eq:definitionweightfunctionapproxfunctionalequation} we can see that 
\begin{equation*}
    V_s(y; F_j) = \frac{1}{2\pi i}\int_{\substack{\text{Re}(u) = \varepsilon\\
    |u| \leq \log T}} y^{-u}G(u) \frac{\gamma(s + u; F_j)}{\gamma(s; F_j)}\frac{du}{u} + O_{\varepsilon, B}(y^{-\varepsilon}T^{-B}).
\end{equation*}
Note that $V_s(y; F_j)$ depends on $F_j$ via the spectral parameter $t_j$. We want to make this dependence as explicit as possible. By Stirling's formula we have 
\begin{equation*}
    \log \Gamma(s) = (s - 1/2)\log s - s + \frac{\log 2\pi}{2} + O(|s|^{-1}).
\end{equation*}
Using the power series expansion of $s \mapsto \log(1 + s)$ as well as the bounds $|s|\leq T^{1-\varepsilon}$, $|u| \leq \log T$ and $|t_j| \asymp T$ we deduce that 
\begin{equation*}
\begin{aligned}
    \log \Gamma((s + u + \epsilon)/2 \pm it_j) - \log \Gamma((s + \epsilon)/2 \pm it_j) & = \frac{u}{2} \log((s + \epsilon)/2 \pm it_j) + O((\log T)/T)\\
    & = \frac{u}{2}(\log t_j \pm i\frac{\pi}{2}) + O(\log T/T^\varepsilon).
\end{aligned}
\end{equation*}
Recalling \eqref{eq:gammafactorinfinity} it follows that
\begin{equation}\label{eq:approximateweightfunction}
    V_s(y; F_j) = W_s(y/t_j) + O_{\varepsilon, B}(y^{-\varepsilon} T^{-B}),
\end{equation}
where
\begin{equation}\label{eq:defiapproximateweightfunction}
    W_s(y) := \frac{1}{2\pi i}\int_{\substack{\text{Re}(u) = \varepsilon\\
    |u| \leq \log T}} (y \pi^{3/2})^{-u} G(u) \frac{\Gamma((s + u + \delta)/2)}{\Gamma((s + \delta)/2)}\frac{du}{u}.
\end{equation}
Note that $W_s(y)$ is independent of the spectral parameter $t_j$. Inserting \eqref{eq:approximateweightfunction} into \eqref{eq:truncatedapproximatefunctionalequation} it follows that 
\begin{equation}\label{eq:truncatedapproximatefunctionalequationv2}
    L(s; F_j) = \sum_{n \leq (\sqrt{\q|s|}T)^{1 + \varepsilon}} \frac{\lambda_{F_j}(n)}{n^s}W_s\left(\frac{n}{Xt_j \sqrt{\q}}\right) + \epsilon(F_j, s)\sum_{n \leq (\sqrt{\q|s|}T)^{1 + \varepsilon}} \frac{\overline{\lambda_F(n)}}{n^{1-s}}W_{1-s}\left(\frac{nX}{t_j \sqrt{\q}}\right) + O_{q, \varepsilon, B}(T^{-B}).
\end{equation}
We have proved this formula for $T \leq |t_j| \leq 2T$ and $X \in (1/2, 2)$. We can use the variable $X$ to relax the dependence on $t_j$. Indeed, if we take absolute values, square both sides of \eqref{eq:truncatedapproximatefunctionalequationv2} and integrate $X$ over $(1/2, 2)$ with respect to Lebesgue measure we obtain
\begin{equation*}
    |L(s; F_j)|^2 \ll \int_{1/4}^{4}\left|\sum_{n \leq (\sqrt{\q|s|}T)^{1 + \varepsilon}} \frac{\lambda_{F_j}(n)}{n^s}W_s\left(\frac{nr}{T\sqrt{\q}}\right)\right|^2\, dr + O_{q, \varepsilon, B}(T^{-B}).
\end{equation*}
After summing over $T \leq t_j \leq 2T$ we have
\begin{equation}\label{eq:truncatedapproximatefunctionalinequality}
    \sum_{T \leq |t_j| \leq 2T}  |L(s; F_j)|^2 \ll  \int_{1/4}^{4}  \sum_{T \leq t_j \leq 2T} \left|\sum_{n \leq (\sqrt{\q|s|}T)^{1 + \varepsilon}} \frac{\lambda_{F_j}(n)}{n^s}W_s\left(\frac{nr}{T\sqrt{\q}}\right)\right|^2 \, dr + O_{q, \varepsilon, B}(T^{-B}).
\end{equation}

\begin{lem}\label{lem:largesieveadjoint}
For any sequence $(\alpha_n)_n$ and $N \geq 1$ we have 
\begin{equation}\label{eq:largesieveadjoint}
    \sum_{T \leq t_j \leq 2T} \left|\sum_{n \leq N} \lambda_{F_j}(n)\alpha_n\right|^2\ll_{q, \varepsilon} (NT)^\varepsilon(T^2 + N^2)\sum_{r \leq N}\left(\sum_{u \leq \sqrt{N/r}} |\alpha_{ru^2}|\right)^2. 
\end{equation} 
\end{lem}
\begin{proof}
Recall that $F_j = \Ad(f_j) \otimes \psi$. From \eqref{eq:gelbartjacquetlift} it is clear that $\Ad(f_j)$ and $F_j$ depend on $f_j$ only through its equivalence class of twists. Thus, it makes no difference to assume that $f_j$ is \emph{twist minimal}, meaning that $\q(f_j) \mid \q(f_j \otimes \psi')$ for any Dirichlet character $\psi'$. Also, we can break the $t_j$-sum into $O_q(1)$ different sums, so that the nebentypus and the conductor of $f_j$ are fixed in each sum. We denote the constant nebentypus by $\eta$, and the common conductor by $M$, respectively. In this proof we see $\eta$ and $\psi$ as primitive Dirichlet characters modulo $\q(\eta)$ and $\q(\psi)$ respectively.

For any prime $p$ we have  
\begin{equation*}
    \sum_{i = 0}^\infty \lambda_{f_j}(p^i) X^i  = \frac{1}{(1 - \alpha_1 X)(1 - \alpha_2 X)}
\end{equation*}
for some local roots $\alpha_1, \alpha_2$ depending on $p$. If $p \mid M$, then $\alpha \beta = 0$, otherwise $\alpha \beta = \eta(p)$. When $(p, M) = 1$, then $(p, \q(\Ad(f_j))) = 1$ as well and we have 
\begin{equation*}
    \sum_{i = 0}^\infty \lambda_{F_j}(p^i) X^i = \frac{1}{(1 - \psi(p)\alpha_1 \alpha_2^{-1} X)(1 - \psi(p)X)(1 - \psi(p)\alpha_1^{-1}\alpha_2 X)},
\end{equation*} 
see \cite[Proposition 1.4]{gelbartjacquetlift}. One can compute that
\begin{equation*}
    \lambda_{F_j}(p^i) = \psi(p^i)\sum_{r = 0}^{\lfloor i/2\rfloor} \overline{\eta(p^{i - 2r})}\lambda_{f_j}(p^{2(i-2r)}),
\end{equation*}
and by multiplicativity of $n \mapsto \lambda_{F_j}(n)$ we deduce that
\begin{equation}\label{eq:relationadjointLfunction}
     \lambda_{F_j}(n) = \psi(n)\sum_{\substack{d\mid n\\n = du^2}} \overline{\eta(d)}\lambda_{f_j}(d^2)
\end{equation}
whenever $(n, M) = 1$. On the other hand, when $p \mid M$ one can check that 
\begin{equation}
    \sum_{i = 0}^\infty \lambda_{F_j}(p^i) X^i = \frac{1}{(1 - \beta_1 X)(1 - \beta_2 X)(1 - \beta_3 X)}
\end{equation}
where $\max(|\beta_1|, |\beta_2|, |\beta_3|) \leq 1$, see \cite[Corollary 1.3 and Proposition 1.4]{gelbartjacquetlift} and see also the criterion for unitarizability of principal series \cite[Proposition 4.6.12]{bump}. Therefore, when $n \mid M^\infty$ we have 
\begin{equation}\label{eq:boundbadprimesadjoint}
    |\lambda_{F_j}(n)| \ll_\varepsilon n^\varepsilon. 
\end{equation}
Since $M \mid q^2$ we have
\begin{equation*}
    |\{n' \leq N: n'\mid M^\infty\}| \ll_q N^\varepsilon.
\end{equation*}
Using this estimate, multiplicativity of $n\mapsto \lambda_{F_j}(n)$ and \eqref{eq:boundbadprimesadjoint} it follows that
\begin{equation}\label{eq:auxeq:breakupoflargesieve}
    \left|\sum_{n \leq N} \lambda_{F_j}(n)\alpha_n\right|^2 \ll_q N^\varepsilon \sum_{\substack{n' \leq N\\
    n' \mid M^\infty}} \left|\sum_{\substack{n_0 \leq N/n'\\
    (n_0, M) = 1}} \lambda_{F_j}(n_0)\alpha_{n_0n'}\right|^2.
\end{equation}
Using relation \eqref{eq:relationadjointLfunction} it follows that 
\begin{equation*}
    \sum_{\substack{n_0 \leq N/n'\\
    (n_0, M) = 1}} \lambda_{F_j}(n_0)\alpha_{n_0n'} = \sum_{\substack{r \leq N/n'\\
    (r, M) = 1}} \lambda_j(r^2) \beta_{r, n'}
\end{equation*}    
where the sequence $(\beta_{r, n'})_r$ is given by 
\begin{equation*}
\beta_{r, n'} := \overline{\eta(r)} \psi(r)\sum_{\substack{u \leq \sqrt{N/n'r}\\
    (u, M) = 1}}\psi(u^2)\alpha_{rn'u^2}.
\end{equation*}
With the assumptions above, the collection $\{f_j: T \leq t_j \leq 2T\}$ is an orthogonal set of Maass forms in $L_0^2(\Gamma_0(M), \eta)$. Taking into account Lemma \ref{lem:normnormalizednewform}, we can apply the large sieve inequality of Lemma \ref{lem:largesieveinequality} and obtain 
\begin{equation*}
\begin{aligned}
    \sum_{T \leq t_j \leq 2T}\left|\sum_{\substack{r \leq N/n'\\
    (r, M) = 1}} \lambda_j(r^2) \beta_{r, n'}\right|^2 & \ll_{q, \varepsilon} (TN)^\varepsilon (T^2 + (N/n')^2)\sum_{r \leq N/n'} \left(\sum_{u \leq \sqrt{N/n'r}} |\alpha_{rn'u^2}|\right)^2,\\
    & \leq (TN)^\varepsilon (T^2 + N^2)\sum_{r \leq N} \left(\sum_{u \leq \sqrt{N/r}} |\alpha_{ru^2}|\right)^2.
\end{aligned}
\end{equation*}
The proof of the lemma is concluded by recalling \eqref{eq:auxeq:breakupoflargesieve} and the bound $|\{n' \leq N: n'\mid M^\infty\}| \ll_q N^\varepsilon$.
\end{proof}

Note from \eqref{eq:defiapproximateweightfunction} that $W_s(y) \ll_\varepsilon T^\varepsilon y^{-\varepsilon}$. For $y \in (1/4, 4)$ and $\text{Re}(s) = 1/2$ consider the sequence
\begin{equation*}
    \alpha_n := \frac{W_s\left(\frac{ny}{T \sqrt{\q}}\right)}{n^s}.
\end{equation*}
Letting $N = (\sqrt{\q|s|}T)^{1 + \varepsilon}$ it is clear that  
\begin{equation*}
    \sum_{r \leq N}\left(\sum_{u \leq \sqrt{N/r}}|\alpha_{ru^2}|\right)^2 \ll_{q, \varepsilon} (TN)^\varepsilon. 
\end{equation*}
Thus, if we apply Lemma \ref{lem:largesieveadjoint} in Equation \eqref{eq:truncatedapproximatefunctionalinequality} we obtain 
\begin{equation*}
    \sum_{T \leq |t_j| \leq 2T}  |L(s; F_j)|^2 \ll_{q, \varepsilon} |s|^{1 + \varepsilon} T^{2 + \varepsilon},
\end{equation*}
which completes the proof of Proposition \ref{prop:lindelofforregularnewforms}.

\subsection{Bad primes.}\label{subsec:badprimes}
Recall from \eqref{eq:definitionofbadfactorsrankinselberg} the series $\L_q(s; u_j^{g_1},  \overline{u_j^{g_2}})$. We are going to control it by bounding the Dirichlet coefficients directly, as in the following lemma. Recall that the newform $f \in \HH'_{it}(q)$ is normalized so that $\rho(f; 1) = 1$.
\begin{lem}\label{lem:boundfouriercoefficientoldspace}
Let $f \in \HH'_{it}(q)$, with conductor $\q(f) \mid q^2$, and let $v \in V_f(\Gamma_d(q))$, defined in \eqref{eq:oldspaceGammad}. Let $\theta > 0$ be a bound towards the Ramanujan conjecture. Then for any $n \geq 1$ and any $\varepsilon > 0$ we have
\begin{equation*}
    \rho(v; n/q) \ll_{\varepsilon} q^{1/2 + \varepsilon} n^{\theta + \varepsilon} \times  (q^2/\q(f), n)^{1/2 - \theta} \times \frac{\langle v, v\rangle^{1/2}}{\langle f, f \rangle^{1/2}}.
\end{equation*}
\end{lem}
 If $f \in \HH'_{it}(q)$ and $(n, \q(f)) = 1$, then we know that 
\begin{equation}\label{eq:heckeeigenvalueboundgoodprimes}
    |\lambda_f(n)| \ll_{\varepsilon} n^{7/64 + \varepsilon}
\end{equation}
by the second appendix to \cite{kimramakrishnansarnak2003}. Thus, we can take $\theta = 7/64 + \varepsilon$ in Lemma \ref{lem:boundfouriercoefficientoldspace}. If $n \mid \q(f)^\infty$ then we have the stronger bound 
\begin{equation}\label{eq:heckeeigenvalueboundbadprimes}
    |\lambda_f(n)| \leq 1,
\end{equation} 
which follows from the criterion for unitarizability of principal series, see \cite[Proposition 4.6.12]{bump}.

Lemma \ref{lem:boundfouriercoefficientoldspace} is an adaption of \cite[Theorem 5.1]{Schulze_Pillot_Petersson} to the nonholomorphic case. Since \cite{Schulze_Pillot_Petersson} uses the material from \cite[Sections 4.5 and 4.6]{miyake}, for the convenience of the reader we recall the setup in the latter. 

Let $N$ be a positive integer, and consider the semigroups $\Delta_0(N), \Delta^*_0(N) \subset \GL_2^+(\Q)$ defined by 
\begin{equation*}
\begin{aligned}
\Delta_0(N) & := \{\begin{psmallmatrix}
    a & b\\
    c & d
\end{psmallmatrix}\in M_2(\Z)\cap \GL_2^+(\Q) : c \equiv 0 \modulo N, (a, N) = 1\},\\
\Delta^*_0(N) & := \{\begin{psmallmatrix}
    a & b\\
    c & d
\end{psmallmatrix}\in M_2(\Z)\cap \GL_2^+(\Q) : c \equiv 0 \modulo N, (d, N) = 1\}.
\end{aligned}
\end{equation*}
Define the Hecke algebras 
\begin{equation*}
    R(N) := R(\Gamma_0(N), \Delta_0(N)), \quad R^*(N) := R(\Gamma_0(N), \Delta^*_0(N)).
\end{equation*}
A $\Z$-basis for these Hecke algebras is described by the bijections
\begin{equation*}
    \{(a, d) \in \Z^2_{\geq 1}: a \mid d \text{ and }(a, N) = 1\} \longleftrightarrow \Gamma_0(N)\backslash \Delta_0(N)/\Gamma_0(N)
\end{equation*}
and 
\begin{equation*}
    \{(d, a) \in \Z^2_{\geq 1}: a \mid d \text{ and }(a, N) = 1\} \longleftrightarrow \Gamma_0(N)\backslash \Delta^*_0(N)/\Gamma_0(N)
\end{equation*} 
given by
\begin{equation*}
    (a, d)  \mapsto T(a, d) := \Gamma_0(N) \begin{psmallmatrix}
        a & 0\\
        0 & d
    \end{psmallmatrix}\Gamma_0(N), \text{ and }
    (d, a) \mapsto T^*(d, a)  := \Gamma_0(N) \begin{psmallmatrix}
        d & 0\\
        0 & a
    \end{psmallmatrix}\Gamma_0(N),
\end{equation*}
respectively, see \cite[Lemma 4.5.2]{miyake}. For $n \geq 1$ we define 
\begin{equation*}
    T(n) := \sum_{\substack{ad = n\\
    a \mid d}}T(a, d), \text{ and } T^*(n) := \sum_{\substack{ad = n\\
    a \mid d}}T^*(d, a).
\end{equation*}

Let $\psi$ be a Dirichlet character modulo $N$. We extend $\psi$ to a function on $\Delta_0(N)$, still denoted by $\psi$. The extension is defined as follows: given $\alpha \in \Delta_0(N)$ we let  
\begin{equation*}
    \psi(\alpha) := \overline{\psi(a(\alpha))}.
\end{equation*}
We can also extend $\psi$ to a function on $\Delta^*_0(N)$, denoted by $\psi^*$. The extension is defined as follows: given $\beta \in \Delta_0^*(N)$ we let  
\begin{equation*}
    \psi^*(\beta) := \psi(d(\beta)).
\end{equation*}
Note that, in general, $\psi$ and $\psi^*$ do not agree on $\Delta_0(N) \cap \Delta_0^*(N)$. 

The Hecke ring $R(N)$ acts on $\C_{it}(\Gamma_0(N), \psi)$ by the following recipe: given a decomposition $\Gamma_0(N)\alpha\Gamma_0(N) = \bigsqcup_{i = 1}^r \Gamma_0(N)\alpha_i$, and given $u \in \C_{it}(\Gamma_0(N), \psi)$ we let 
\begin{equation*}
    u|_k \Gamma_0(N)\alpha\Gamma_0(N) := \det(\alpha)^{-1/2}\sum_{i = 1}^r \overline{\psi(\alpha_i)}f|_k \alpha_i.
\end{equation*} 
Similarly, the Hecke ring $R^*(N)$ acts on $\C_{it}(\Gamma_0(N), \psi)$ by the formula 
\begin{equation*}
    u|_k \Gamma_0(N)\alpha\Gamma_0(N) := \det(\alpha)^{-1/2}\sum_{i = 1}^r \overline{\psi^*(\alpha_i)}f|_k \alpha_i.
\end{equation*} 
In particular, for $(n, N) = 1$ we have 
\begin{equation*}
    u|_k T(n, n) = n^{-1}\psi(n) u, \quad u|_k T^*(n, n) = n^{-1}\overline{\psi(n)} u.
\end{equation*}
By \cite[Equation (4.5.26)]{miyake}, for $n \geq 1$ the action of $T(n)$ on $\C_{it}(\Gamma_0(N), \psi)$ is given by the familiar formula \eqref{eq:Heckeoperatorsexplicitformula}.

Suppose now that $f \in \HH_{it}(N, \psi)$ is a newform. By definition, $f$ is an eigenfunction of all the Hecke operators $T(n)$. By \cite[Theorem 4.5.9]{miyake}, the ring $R(N)$ is a polynomial ring over $\Z$ on the variables $\{T(p), T(p, p)\}_{p \nmid N} \cup \{T(p)\}_{p \mid N}$. It follows that $f$ is an eigenfunction of the operator $T(a, d)$ for each pair $(a, d)\in \Z_{\geq 1}^2$ with $a \mid d$ and $(a, N) = 1$. We denote the corresponding eigenvalues by $\lambda_f(a, d)$, so that 
\begin{equation*}
    f|_k T(a, d) = \lambda_f(a, d)f.
\end{equation*} 
One can check that when $(d, N) = 1$ these eigenvalues agree with those defined in \eqref{eq:eigenvaluegeneralheckeoperator}, but this fact is not necessary for our purposes. By \cite[Theorem 4.5.4(2)]{miyake}, we know that $T(a, d)$ and $T^*(d, a)$ are adjoint operators with respect to the inner product \eqref{eq:innerproductmaasforms} on $\C_{it}(\Gamma_0(N), \psi)$. It follows that we have 
\begin{equation*}
    f|_k T^*(d, a) = \overline{\lambda_f(a, d)}f
\end{equation*}
for $(a, d) \in \Z_{\geq 1}^2$ such that $a \mid d$ and $(a, N) = 1$. Furthermore, when $(d, N) = 1$ we know that 
\begin{equation*}
    \overline{\lambda_f(a, d)} = \overline{\psi(ad)}\lambda_f(a, d),
\end{equation*}
see \cite[Theorem 4.5.4(1)]{miyake}.

We end up this discussion by recalling that $n \mapsto \lambda_f(1, n)$ is multiplicative, and that we have the following relations:
\begin{enumerate}
    \item [i)] If $p \mid N$ and $k \geq 1$, we have \begin{equation*}
        \lambda_f(1, p^k) = \lambda_f(1, p)^k.
    \end{equation*} 
    \item [ii)] If $(p, N) = 1$ we have 
    \begin{equation*}
        \lambda_f(1,p)^2 = \lambda_f(1, p^2) + (1 + p^{-1})\psi(p).
    \end{equation*}
    \item [iii)] If $(p, N) = 1$ and $k \geq 2$ we have 
    \begin{equation*}
        \lambda_f(1, p)\lambda_f(1, p^k) = \lambda_f(1, p^{k + 1}) + \psi(p)\lambda_f(1, p^{k-1}).
    \end{equation*}
\end{enumerate}

Since $\lambda_f(1, p) = \lambda_f(p)$, it follows that if $\theta$ is an exponent towards the Ramanujan conjecture, then
\begin{equation}\label{eq:ramanujanhecke1eigenvalues}
    |\lambda_f(1, n)| \ll_\varepsilon n^{\theta + \varepsilon}
\end{equation}
for each $n \geq 1$.

\begin{defi}
For $r \in \R_{> 0}$ we let $V_r := \begin{psmallmatrix}
    r & 0\\
    0 & 1
\end{psmallmatrix}$.
\end{defi}
The following result is the nonholomorphic analogue of \cite[Theorem 2.7]{Schulze_Pillot_Petersson}.

\begin{lem}[Schulze-Pillot and Yenirce]\label{lem:innerproducdilations}
Let $f \in \HH_{it}(N, \psi)$. Then, for $m, n \geq 1$ we have 
\begin{equation}\label{eq:innerproductdilations}
    \langle f|_k V_m, f|_k V_n\rangle = \frac{\lambda_f\left(1, \frac{n}{d}\right)\overline{\lambda_f\left(1, \frac{m}{d}\right)}}{\left(\frac{mn}{d^2}\right)^{1/2}  \prod_{\substack{p \mid (mn/d^2)\\
    p \nmid N}}\left(1 + \frac{1}{p}\right)   } \langle f, f\rangle,
\end{equation}
where $d = (m, n)$. 
\end{lem}
\begin{proof}
Taking into account the previous discussion, the proof of this lemma is completely analogous to the proof of \cite[Theorem 2.7]{Schulze_Pillot_Petersson}. 
\end{proof}

If $N \mid q^2$, the set $\{f|_k V_d : d \mid (q^2/N)\}$ is a basis for the oldspace $V_f(\Gamma_0(q^2))$. Thanks to Lemma \ref{lem:innerproducdilations}, we can compute an orthogonal basis of $V_f(\Gamma_0(q^2))$. For this purpose, we introduce maps $\{A(n; f): C(\Hyp)\rightarrow C(\Hyp)\}_{n \geq 1}$, determined uniquely by the properties:
\begin{enumerate}
    \item [i)] $A(1; f)=  \text{Id}$.
    \item [ii)] $A(nm; f) = A(n; f)A(m; f)$ whenever $(n, m) = 1$.
    \item [iii)] For $u \in C(\Hyp)$, $p \mid N$ and $j \geq 1$ we have
    \begin{equation*}
        A(p^j; f)u := u|_k V_{p^j} - \frac{\overline{\lambda_f(1, p)}}{p^{1/2}}u|_k V_{p^{j-1}}.
    \end{equation*}
    \item [iv)] For  $u \in C(\Hyp)$, $p \nmid N$ and $j = 1$ we have
    \begin{equation*}
        A(p; f)u := u|_k V_p - \frac{\overline{\lambda_f(1, p)}}{p^{1/2} + p^{-1/2}}u.
    \end{equation*}
    \item [v)] For  $u \in C(\Hyp)$, $p \nmid N$ and $j \geq 2$ we have
    \begin{equation*}
        A(p^j; f)u := u|_k V_{p^j} - \frac{\overline{\lambda_f(1, p)}}{p^{1/2}}u|_k V_{p^{j-1}} + \frac{\overline{\psi(p)}}{p}u|_k V_{p^{j-2}}.
    \end{equation*}
\end{enumerate}

The following result is the nonholomorphic analogue of \cite[Theorem 3.2]{Schulze_Pillot_Petersson}.
\begin{lem}[Schulze-Pillot and Yenirce]\label{lem:orthogonalbasisschulzepillot}
Let $f \in \HH_{it}(N, \psi)$ such that $N \mid q^2$. Then $\{A(d; f)f : d \mid (q^2/N)\}$ is an orthogonal basis of $V_f(\Gamma_0(q^2))$ with $L^2$-norms given by 
\begin{equation}\label{eq:normorthogonalbasisoldspace}
\frac{\langle A(d; f)f, A(d; f)f \rangle }{ \langle f, f \rangle} = \prod_{p \mid (d, N)}\left(1 - \frac{|\lambda_f(1, p)|^2}{p}\right) \times \prod_{\substack{p \mid d\\p \nmid N}}\left(1 - \frac{|\lambda_f(1, p)|^2}{p + 1}\right)\times \prod_{\substack{p^2 \mid d\\ p \nmid N}} \left(1 - \frac{1}{p^2}\right).
\end{equation}
Furthermore, we have
\begin{equation}\label{eq:estimatenormorthogonalbasisoldspace}
    d^{-\varepsilon} \ll_{\varepsilon} \frac{\langle A(d; f)f, A(d; f)f \rangle}{ \langle f, f \rangle} \ll_{\varepsilon} d^\varepsilon,
\end{equation}
for $d \mid (q^2/N)$ and any $\varepsilon > 0$.
\end{lem}
\begin{proof}
The proof of the first claim is completely analogous to the proof of \cite[Theorem 3.2]{Schulze_Pillot_Petersson}. The second claim follows by combining  \eqref{eq:ramanujanhecke1eigenvalues} with \eqref{eq:normorthogonalbasisoldspace}.
\end{proof}

We are now ready to give the proof of Lemma \ref{lem:boundfouriercoefficientoldspace}.
\begin{proof}[Proof of Lemma \ref{lem:boundfouriercoefficientoldspace}]
Let $v \in V_f(\Gamma_d(q))$. Note that $w := v|_k V_q$ has the following properties: $\rho(w; n) = q^{1/2}\rho(v; n/q)$, $w \in V_f(\Gamma_0(q^2))$ and $\langle v, v\rangle = \langle w, w \rangle$. Thus, it will be enough to bound the Fourier coefficients of $w \in V_f(\Gamma_0(q^2))$. Furthermore, we can and will assume that $\langle w, w\rangle = 1$. 

From the definition of $A(d; f)$ we see that 
\begin{equation*}
    (A(d; f)f)(z) = \sum_{r \mid d} c_{r, d} f|_k V_r, 
\end{equation*}
where $c_{r, d}$ are certain coefficients which satisfy $c_{r, d} \ll_\varepsilon (d/r)^{\theta -1/2 + \varepsilon}$. Note that for positive integers $r, n \geq 1$ we have
\begin{equation*}
\rho(f|_k V_r; n) = \begin{dcases}
    r^{1/2}\lambda_f(n/r), & \text{ if }r \mid n,\\
    0, & \text{ otherwise}.
\end{dcases}
\end{equation*}
Therefore,
\begin{equation}\label{eq:auxiliaryestimatefouriercoefficent}
\begin{aligned}
    |\rho(A(d; f)f; n)| & \ll_\varepsilon \sum_{r \mid (d, n)}  r^{1/2}|\lambda_f(n/r)|(d/r)^{\theta -1/2 + \varepsilon}\\
    & \ll_\varepsilon n^{\theta + \varepsilon}d^{\theta - 1/2 + \varepsilon}\sum_{r \mid (d, n)}r^{1 -2\theta + \varepsilon}\\
    & \ll_\varepsilon  n^{\theta + \varepsilon}d^{\theta - 1/2 + \varepsilon} (d, n)^{1 - 2\theta + \varepsilon}.
\end{aligned}
\end{equation}

Since $\{A(d; f)f\}_{d \mid (q^2/N)}$ is an orthogonal basis of $V_f(\Gamma_0(q^2))$, given $w \in V_f(\Gamma_0(q^2))$ such that $\langle  w, w \rangle = 1$ we can write
\begin{equation*}
    w = \sum_{ d \mid (q^2/N)} a_d \frac{A(d; f)f}{\langle A(d; f)f, A(d; f)f\rangle^{1/2}},
\end{equation*}
for some $a_d$ with $\sum_{d}|a_d|^2 = 1$. By Cauchy's inequality we have
\begin{equation}\label{eq:secondauxiliaryboundfourier}
\begin{aligned}
    |\rho(w; n)|^2 & \leq \sum_{d \mid (q^2/N)}\frac{|\rho(A(d; f)f; n)|^2}{\langle A(d; f)f, A(d; f)f\rangle }\\
    & \ll_\varepsilon \langle f, f \rangle^{-1} \sum_{d \mid (q^2/N)} n^{2\theta + \varepsilon}d^{2\theta - 1 + \varepsilon} (d, n)^{2 - 4\theta + \varepsilon},
\end{aligned}
\end{equation}
where we have used estimates \eqref{eq:estimatenormorthogonalbasisoldspace} and \eqref{eq:auxiliaryestimatefouriercoefficent}. Using the trivial inequality $(d, n)^{2 - 4\theta + \varepsilon} \leq \sum_{r \mid (d, n)} r^{2 - 4\theta + \varepsilon}$ in \eqref{eq:secondauxiliaryboundfourier} and exchanging the order of summation we obtain
\begin{equation*}
\begin{aligned}
    |\rho(w; n)|^2 & \ll_\varepsilon \langle f, f \rangle^{-1}n^{2\theta + \varepsilon} \sum_{r \mid (q^2/N, n)} r^{1 - 2\theta + \varepsilon}\sum_{s \mid q^2/(Nr)} s^{2\theta - 1 + \varepsilon}\\
    & \ll_\varepsilon \langle f, f\rangle^{-1}n^{2\theta + \varepsilon} q^\varepsilon (q^2/N, n)^{1 - 2\theta + \varepsilon}.
\end{aligned}
\end{equation*}
The desired bound follows by taking square roots.
\end{proof}

Lemma \ref{lem:upperboundfouriercoefficients} allows us to control the contribution of the bad primes to $\L(s; |u_j|^2)$ on the critical line.
\begin{cor}\label{cor:boundrankinselbergbadprimes}
The Dirichlet series $\L_q(s; u_j^{g_1},  \overline{u_j^{g_2}})$ introduced in \eqref{eq:definitionofbadfactorsrankinselberg} converges absolutely on $\text{Re}(s) > 7/32$. On the critical line $\text{Re}(s) = 1/2$ it satisfies
\begin{equation*}
    |\L_q(s; u_j^{g_1},  u_j^{g_2})| \ll_{q} \frac{1}{\langle g_1, g_1 \rangle^{1/2}\langle g_2, g_2 \rangle^{1/2}}.
\end{equation*}
\end{cor}

\subsection{Proof of Theorem \ref{thm:informationaboutrankinselberg}.}
The only assertion of Theorem \ref{thm:informationaboutrankinselberg} yet to be proven is the inequality in \eqref{eq:lindelofonaverage}. Recall that 
\begin{equation*}
    \L(s; |u_j|^2) = \sum_{g_1, g_2 \in \T_{f_j}(q)} \L_q(s; u_j^{g_1}, \overline{u_j^{g_2}})L^{(q)}(s; g_1 \otimes \overline{g_2}).
\end{equation*}
By Lemma \ref{lem:normnormalizednewform}, Corollary \ref{cor:boundrankinselbergbadprimes} and formulas \eqref{eq:rankinselbergisalmosteulerproduct} and \eqref{eq:gelbartjacquetlift} and we have 
\begin{equation*}
    \omega(t_j)^2 |\L(s; |u_j|^2)|^2 \ll_{q, \varepsilon} T^\varepsilon \sum_{\psi \modulo q} |L(s; \psi)|^2|L(s; \Ad(f_j)\otimes \psi)|^2.
\end{equation*}
By the convexity bound $|L(s; \psi)|^2 \ll_{q, \varepsilon} (1 + |s|)^{1/2 + \varepsilon}$ on $\text{Re}(s) = 1/2$. Thus, summing over $t_j$ we have 
\begin{equation*}
    \sum_{|t_j| \leq T} \omega(t_j)^2 |\L(s; |u_j|^2)|^2 \ll_{q, \varepsilon} T^\varepsilon (1 + |s|)^{1/2 + \varepsilon}  \sum_{\psi \modulo q}  \sum_{|t_j| \leq T}|L(s; \Ad(f_j)\otimes \psi)|^2.
\end{equation*}
At this point the inequality of \eqref{eq:lindelofonaverage} follows from Corollary \ref{cor:sumLfunctionsexceptional} and Proposition \ref{prop:lindelofforregularnewforms}.

\section{Proof of Theorem \ref{thm:luosarnakbound}.}\label{sec:proofofluosarnakbound}
Let $X, T$ be sufficiently large with $T \leq X^{1/2}$. We assume that $T \geq X^{1/6}$, since otherwise Theorem \ref{thm:luosarnakbound} follows from the trivial bound \eqref{eq:trivialboundexponentialsum}. Let $2\beta = \log X + iT^{-1}$ and define
\begin{equation}\label{eq:deftestfunction}
    \varphi(x) = \frac{(-1)^k\sinh(\beta)}{4\pi} x\exp(ix\cosh \beta).
\end{equation}
In the proof of Theorem \ref{thm:luosarnakbound} we are going to apply Proposition \ref{prop:forwardkuznetsov} to this test function $\varphi$. Note that 
\begin{equation}
    \text{Im}(\cosh \beta) = \frac{X^{1/2}}{4T}\left(1 + O(T^{-2})\right),
\end{equation}
so that   
\begin{equation}\label{eq:testfunctionestimate}
    |\varphi(x)| \ll X^{1/2} x \exp(-x\frac{\sqrt{X}}{T})
\end{equation}
for all $x > 0$, provided that $T$ is sufficiently large. In particular, we see that $\varphi$ satisfies the bounds \eqref{eq:boundsnear0} and \eqref{eq:boundsnearinfty} above, so that we can apply Proposition \ref{prop:forwardkuznetsov} to $\varphi$. Combining \cite[\href{https://dlmf.nist.gov/10.22.E49}{(10.22.49)}]{dlmf} and \cite[\href{https://dlmf.nist.gov/15.4.E18}{(15.4.18)}]{dlmf} we have 
\begin{equation*}
    \int_0^\infty J_{2it}(y)y\exp(iy\cosh \beta)\frac{dy}{y} = \frac{i}{\sinh \beta} \exp(-(\pi + 2i\beta)t).
\end{equation*}
Therefore, 
\begin{equation}\label{eq:htestfunction}
    h(t) = \begin{dcases}
        \frac{\sinh(\pi + 2i\beta)t}{\sinh(\pi t)}, & \text{ if }k = 0,\\
        \frac{\cosh(\pi + 2i\beta)t}{\cosh(\pi t)}, & \text{ if }k = 1.
    \end{dcases}
\end{equation}
Note that in both cases we have  
\begin{equation}\label{eq:preciseevaluationh}
    h(t) = X^{i |t|}\exp(-|t|/T) + O(\exp(-2|t|(\pi - 1/T))), \quad t \in \R, |t| \geq 1/4.
\end{equation}
When $k = 0$ we have the bound
\begin{equation}\label{eq:boundforhwhenkis0}
    h(t) \ll X^{|\text{Im}(t)|} \log X, \quad t \in \CC, |t| \leq 1/4.
\end{equation}
Note that we cannot discard the existence of small eigenvalues for $(\Gamma, \chi)$, so we need to control $h(t)$ when $t$ is purely imaginary and $|\text{Im}(t)| \leq 7/64$. The bound \eqref{eq:boundforhwhenkis0} is used to handle this situation. When $k = 1$ there are no small eigenvalues and we only need to evaluate $h(t)$ for real $t$. In this case \eqref{eq:preciseevaluationh} is valid for all real $t$, and for small $t$ it simply asserts that $h(t) \ll 1$.

Moving on to the other terms in Proposition \ref{prop:forwardkuznetsov}, for any $\xi \in (0, 1]$ we can use \cite[\href{https://dlmf.nist.gov/10.22.E49}{(10.22.49)}]{dlmf} again to obtain
\begin{equation*}
\begin{aligned}
    \int_0^\infty J_k(\xi y)\varphi(y)\, dy  &=  \frac{(-1)^k \sinh \beta}{4\pi }\frac{(\xi/2)^k(k + 1)}{(-i\cosh \beta)^{2 + k}}F\left(1 + \frac{k}{2}, 1 +\frac{1 + k}{2}; 1 + k; \frac{\xi^2}{\cosh^2 \beta}\right)\\
    & \ll \xi^k X^{-(1 + k)/2}.
\end{aligned}
\end{equation*}
In particular, the contribution from the identity term in the Bruggeman-Kuznetsov trace formula is 
\begin{equation}\label{eq:identitytermestimate}
    \int_0^\infty J_k(x)\varphi(x) \, dx \ll X^{-(1 + k)/2},
\end{equation}
and the weight function of the Kloosterman sums is 
\begin{equation}\label{eq:weightkloostermansums}
\begin{aligned}
    \varphi_*(x) & = \varphi(x) - \int_0^1\xi x J_k(\xi x)\left(\int_0^\infty J_k(\xi y) \varphi(y)\, dy \right)\, d\xi\\
    & = \varphi(x) + O\left(X^{-(1 + k)/2} \min(x^{1 + k}, x^{1/2})\right),
\end{aligned}
\end{equation}
where we have used that $J_k(x) \ll \min(x^k, x^{-1/2})$. Combined with \eqref{eq:testfunctionestimate} and the estimate for Kloosterman sums from Lemma \ref{lem:weilbound} we deduce that 
\begin{equation*}
\sum_{c = 1}^\infty \frac{S_{(\Gamma, \chi)}(n, n; c)}{c \ell}\varphi_*\left(\frac{4\pi n}{c}\right) \ll_{q,\varepsilon} n^{1/2}T^{1/2}X^{1/4}(XTn)^\varepsilon.
\end{equation*}
Also, using Proposition \ref{prop:upperboundfouriercoefficentseisenstein} as well as formulas \eqref{eq:preciseevaluationh} and \eqref{eq:boundforhwhenkis0} we can estimate the contribution of the continuous spectrum by
\begin{equation*}
\sum_{i = 1}^\kappa\frac{\ell}{4\pi}\int_\R |\rho_{\a_i, t}(n)|^2 \omega(t)h(t)\, dt \ll_{q, \varepsilon} n^\varepsilon T^{1 + \varepsilon}\log X.
\end{equation*}
Thus, applying Proposition \ref{prop:forwardkuznetsov} to $\varphi$ we obtain 
\begin{equation}\label{eq:estimatecoefficientsproofofluosarnak}
    \sum_{t_j} |\rho_j(n)|^2 \omega(t_j)h(t_j)\ll_{q, \varepsilon} n^{1/2}T^{1/2}X^{1/4}(XTn)^\varepsilon.
\end{equation}

Let $v \in C_c^\infty(0, \infty)$ be a smooth bump function with $\int_\R v(x)\, dx = 1$. Let $N \geq 1$  be a parameter to be chosen in terms of $T, X$. By Mellin inversion
\begin{equation*}
\frac{1}{N}\sum_{n \in \frac{1}{q}\Z_{\geq 1}} v(n/N)\sum_{j \geq 1} |\rho_j(n)|^2 \omega(t_j)h(t_j) = \frac{1}{2\pi i}\int_{(\sigma)} \sum_{j \geq 1} \L(s; |u_j|^2) \omega(t_j) h(t_j) \widehat{v}(s) N^{s-1}\, ds,
\end{equation*}
where $\sigma > 1$ and $\L(s; |u_j|^2)$ was introduced in \eqref{eq:naiverankinselberg}. Moving the contour to $\text{Re}(s) = 1/2$ and picking up the pole at $s = 1$  
\begin{equation}\label{eq:keyidentityluosarnak}
\begin{aligned}
\frac{1}{N}\sum_{n \in \frac{1}{q}\Z_{\geq 1}} v(n/N)\sum_{t_j \geq 0} |\rho_j(n)|^2 \omega(t_j)h(t_j)  & = \sum_{j \geq 1} h(t_j)\\
&  + \frac{1}{2\pi i}\int_{(1/2)}\sum_{j \geq 1} \L(s; |u_j|^2) \omega(t_j) h(t_j) \widehat{v}(s) N^{s-1}\, ds,
\end{aligned}
\end{equation}
where we have used the first part of Theorem \ref{thm:informationaboutrankinselberg}. By \eqref{eq:preciseevaluationh} and the following discussion we see that  
\begin{equation}\label{eq:approximationtosmoothsum}
    \sum_{j \geq 1} h(t_j) = \S_0(T, X) + O_{q}(1 + \delta_{k = 0} X^{7/64} \log X),
\end{equation}
where the smoothed sum $\S_0(T, X)$ was introduced in \eqref{eq:smoothversionofS}. Note that $\widehat{v}(s)$ decreases faster than any polynomial on vertical strips. Thus, if we apply Cauchy-Schwarz and inequality \eqref{eq:lindelofonaverage} of Theorem \ref{thm:informationaboutrankinselberg} we see that 
\begin{equation}\label{eq:applicationlindelofonaverage}
\int_{(1/2)} \sum_{j \geq 1} \L(s; |u_j|^2) \omega(t_j) h(t_j) \widehat{v}(s) N^{s-1}\, ds \ll_{q, \varepsilon} N^{-1/2}T^{2 + \varepsilon}.
\end{equation}
Thus, inserting \eqref{eq:estimatecoefficientsproofofluosarnak}, \eqref{eq:approximationtosmoothsum} and \eqref{eq:applicationlindelofonaverage} into \eqref{eq:keyidentityluosarnak} we obtain 
\begin{equation*}
    \S_0(T, X) \ll_{q, \varepsilon} N^{-1/2}T^{2 + \varepsilon} + N^{1/2}T^{1/2}X^{1/4}(XTN)^\varepsilon.
\end{equation*}
Finally, optimizing the value of $N$ by letting $N = T^{3/2}X^{-1/4}$ we deduce 
\begin{equation*}
    \S_0(T, X) \ll_{q, \varepsilon} T^{5/4} X^{1/8 + \varepsilon}.
\end{equation*}
By Proposition \ref{prop:conclusionsmoothingsum}, this bound implies Theorem \ref{thm:luosarnakbound}.

\section{Bykovskii--Zagier series.}\label{sec:bykovskiizagierseries}
In this section we collect some preliminaries about the Bykovskii--Zagier series which we need for the proofs of Theorems \ref{thm:balkanovafrolenkovbound} and \ref{thm:soundyoungshortintervals}. 

\subsection{The sum $\Psi_\Gamma(X; \chi)$ as a sum of residues.}
Let $\{P\}_\Gamma$ be a hyperbolic $\Gamma$-conjugacy class and choose a representative $P$. For $\text{Re}(s) > 1$ define the Dirichlet series 
\begin{equation*}
    Z_P(s) := \ell^{-s}\sum_{\gamma \in C_\Gamma(P)\backslash \Gamma/\Gamma_\infty} \frac{1}{|c(\gamma^{-1}P\gamma)|^s},
\end{equation*}
which only depends on the $\Gamma$-conjugacy class of $P$. The following lemma will allow us to express the sum $\Psi_\Gamma(X; \chi)$ as a sum of residues of these Dirichlet series. In our arguments it plays a role which is analogous to the role played by the class number formula in \cite{Soundararajan_2013}. 
\begin{lem}\label{lem:hyperboliczetafunction}
The series $Z_P(s)$ converges absolutely for $\text{Re}(s) > 1$ and has meromorphic continuation to $\CC$. On the region $\text{Re}(s) \geq 1/2$ its only singularity is a simple pole at $s = 1$ of residue
\begin{equation}\label{eq:residuezetafunction}
    \text{Res}_{s = 1} Z_P(s) = \frac{2}{\pi}\frac{\log N(P_0)}{\vol(\Gamma\backslash \Hyp)(N(P)^{1/2}- N(P)^{-1/2})}.
\end{equation}
\end{lem}
\begin{proof}
Let $t := \Tr(P)$ so that $|t| > 2$. Recall that $N := N(P) > 1$ is defined implicitly by the identity $|t| = N^{1/2} + N^{-1/2}$ and let $\Delta := t^2 - 4 = (N^{1/2} - N^{-1/2})^2$. Let $\Phi \in C_c^\infty([0, \infty))$ and consider the pair-point invariant 
\begin{equation*}
    k(z, w):= \Phi\left(\frac{|z-w|^2}{\text{Im}(z)\text{Im}(w)}\right).
\end{equation*}
We can consider the automorphic kernel attached to the conjugacy class $\{P\}_\Gamma$, defined by
\begin{equation*}
    K_P(z; \Phi) :=  \sum_{\gamma \in C_\Gamma(P)\backslash \Gamma} k(\gamma z, P\gamma z).
\end{equation*}
Recall the Eisenstein series $E(z; s) := E_\infty(z, s; \Gamma, \mathbbm{1}_\Gamma)$ introduced in \eqref{eq:defofeisensteinseries}, given explicitly as 
\begin{equation*}
    E(z; s) = \sum_{\delta \in \Gamma_\infty \backslash \Gamma} \text{Im}(\ell^{-1} \delta z)^s = \ell^{-s} \sum_{\delta \in \Gamma_\infty \backslash \Gamma} \text{Im}(\delta z)^s.
\end{equation*}
We consider the inner product of $K_P(z; \Phi)$ and $E(z;s)$. If we unfold with respect to the Eisenstein series we have
\begin{equation}\label{auxeq:unfoldingeisensteinautomorphic}
\begin{aligned}
\int_{\Gamma\backslash \Hyp} K_P(z; \Phi)E(z; s)\, d\mu(z) & = \ell^{-s}\int_{\Gamma_\infty\backslash \Hyp} \sum_{\gamma \in C_\Gamma(P)\backslash \Gamma} k(\gamma z, P\gamma z) \, y^s d\mu(z)\\
 & = \ell^{-s}\sum_{\gamma \in C_\Gamma(P)\backslash \Gamma/\Gamma_\infty} \int_\Hyp k(z, \gamma^{-1}P\gamma z) y^s \, d\mu(z).\\
\end{aligned}
\end{equation}
Let us write
\begin{equation*}
    \gamma^{-1}P\gamma = \begin{pmatrix}
        a & b\\
        c & d
    \end{pmatrix}.
\end{equation*}
Note that $c \neq 0$ since $|t| > 2$ and $a, b, c, d$ are integers. If we let 
\begin{equation*}
    \sigma_1 := \begin{pmatrix}
|c|^{-1/2} & a|c|^{-1/2}\sgn(c)\\
 0 & |c|^{1/2}
    \end{pmatrix},
\end{equation*}
then a direct computation shows that 
\begin{equation*}
    \sigma_1^{-1}\gamma^{-1}P\gamma \sigma_1 = \begin{pmatrix}
        0 & -\sgn(c)\\
        \sgn(c) & t
    \end{pmatrix}.
\end{equation*}
Thus, performing the change of variables $z \mapsto \sigma_1 z$ we see that 
\begin{equation*}
\int_\Hyp k(z, \gamma^{-1}P\gamma z) y^s \, d\mu(z)  = |c|^{-s} \int_\Hyp k\left(z, \begin{psmallmatrix}
    0 & -\sgn(c)\\
    \sgn(c) & t
\end{psmallmatrix}z\right) y^s \, d\mu(z).
\end{equation*}
Furthermore, by the change of variables $z \mapsto -\overline{z}$ we see that 
\begin{equation*}
    \int_\Hyp k\left(z, \begin{psmallmatrix}
    0 & 1\\
    -1 & t
\end{psmallmatrix}z\right) y^s \, d\mu(z) =\int_\Hyp k\left(z, \begin{psmallmatrix}
    0 & -1\\
    1 & t
\end{psmallmatrix}z\right) y^s \, d\mu(z).
\end{equation*}
Thus, coming back to \eqref{auxeq:unfoldingeisensteinautomorphic} we deduce that 
\begin{equation}\label{eq:expressioninnerproductautomorphiceisenstein}
\int_{\Gamma\backslash \Hyp} K_P(z; \Phi)E(z; s)\, d\mu(z) = Z_P(s)V(s; t; \Phi),
\end{equation}
where 
\begin{equation*}
    V(s;t; \Phi) := \int_\Hyp k\left(z, \begin{psmallmatrix}
    0 & -1\\
    1 & t
\end{psmallmatrix}z\right) y^s \, d\mu(z).
\end{equation*}

Note that if we let 
\begin{equation*}
\sigma_2 := \begin{pmatrix}
    \frac{-t + \sqrt{t^2 - 4}}{2} & \frac{-t - \sqrt{t^2 - 4}}{2}\\
    1 & 1
\end{pmatrix},
\end{equation*}
then $\det(\sigma_2) = \Delta^{1/2}$ and, furthermore, we have 
\begin{equation*}
    \sigma_2^{-1}\begin{pmatrix}
    0 & -1\\
    1 & t
\end{pmatrix} \sigma_2 = \begin{pmatrix}
    N^{1/2} & 0\\
    0 & N^{-1/2}
\end{pmatrix}.
\end{equation*}
Thus, after changing variables $z \mapsto \sigma_2 z$ we have
\begin{equation*}
\begin{aligned}
    V(s;t; \Phi) & = \Delta^{s/2}\int_\Hyp k(z, Nz) \frac{y^s}{|z  + 1|^{2s}}\, d\mu(z)\\
    & = \Delta^{s/2}\int_\R\int_0^\infty \Phi\left(\frac{|(N - 1)z|^2}{Ny^2}\right) \frac{y^s}{((x + 1)^2 + y^2)^s} \frac{dx\, dy}{y^2}\\
    & = \Delta^{s/2}\int_\R \Phi(\Delta(1 + x^2))G(x; s) \, dx,
\end{aligned}
\end{equation*}
where 
\begin{equation*}
G(x; s) := \int_0^\infty \frac{y^s}{((1 + xy)^2 + y^2)^s} \frac{dy}{y}.
\end{equation*}
It is clear that this integral converges absolutely on $\text{Re}(s) > 0$, so that $V(s;t; \Phi)$ is holomorphic on $\text{Re}(s) > 0$. All the previous manipulations are justified when $\Phi \in C_c^\infty([0, \infty))$ is nonnegative and $s = \sigma > 1$. Thus, from \eqref{eq:expressioninnerproductautomorphiceisenstein} we deduce that the series defining $Z_P(s)$ is absolutely convergent for $\text{Re}(s) > 1$. 

The function $V(s;t; \Phi)$ has meromorphic continuation to $\CC$, see \cite[Proposition 4]{zagierrankinselberg}. By varying the test function $\Phi \in C_c^\infty([0, \infty))$ in equation \eqref{eq:expressioninnerproductautomorphiceisenstein} we deduce the assertion of the lemma regarding the meromorphic continuation of $Z_P(s)$. To compute the residue at $s = 1$, recall that $\text{Res}_{s = 1} E(z; s) = \vol(\Gamma\backslash \Hyp)^{-1}$. Thus, 
\begin{equation*}
\begin{aligned}
    \text{Res}_{s = 1}\int_{\Gamma\backslash \Hyp} K_P(z; \Phi)E(z; s)\, d\mu(z) & = \frac{1}{\vol(\Gamma\backslash \Hyp)} \int_{\Gamma\backslash\Hyp} K_P(z; \Phi) \, d\mu(z)\\
    & =  \frac{1}{\vol(\Gamma\backslash \Hyp)}\int_{C_\Gamma(P)\backslash \Hyp} k(z, Pz)\, d\mu(z).
\end{aligned}
\end{equation*}
Recall that $C_\Gamma(P) = \{\pm P_0^n: n \in \Z\}$ for a $\Gamma$-primitive hyperbolic element $P_0$. Letting $\sigma_3$ such that $\sigma_3^{-1}P\sigma_3 = \begin{psmallmatrix}
    N^{1/2}& 0\\
    0 & N^{-1/2}
\end{psmallmatrix}$ and changing variable $z \mapsto \sigma_3 z$ we see that 
\begin{equation*}
\begin{aligned}
    \text{Res}_{s = 1}\int_{\Gamma \backslash \Hyp} K_P(z; \Phi)E(z; s)\, d\mu(z) & = \frac{1}{\vol(\Gamma\backslash \Hyp)}\int_{1}^{N(P_0)}\int_{x \in \R} \Phi\left(\frac{|(N - 1)z|^2}{Ny^2}\right) \frac{dx \,  dy}{y^2}\\
    & = \frac{\log N(P_0)}{\vol(\Gamma\backslash \Hyp)} \int_\R \Phi(\Delta(1 + x^2))\, dx.
\end{aligned}
\end{equation*}
On the other hand, it is not difficult to show that $G(x; 1) = \frac{\pi}{2} - \arctan(x)$, which by parity considerations implies that
\begin{equation*}
    V(1; t; \Phi) = \frac{\pi \Delta^{1/2}}{2}\int_\R \Phi(\Delta(1 + x^2))\, dx.
\end{equation*}
Comparing with \eqref{eq:expressioninnerproductautomorphiceisenstein} and taking $\Phi$ such that $V(1; t; \Phi) \neq 0$ we obtain 
\begin{equation*}
    \text{Res}_{s = 1} Z_P(s) = \frac{2}{\pi}\frac{\log N(P_0)}{\vol(\Gamma\backslash \Hyp)\Delta^{1/2}},
\end{equation*}
as desired.
\end{proof}

It is a natural step to combine the series $Z_P(s)$ attached to the different conjugacy classes $\{P\}_\Gamma$ for which $\Tr(P)$ takes a fixed value. When $|t| \geq 3$ we define 
\begin{equation}\label{eq:defiofzagierbykovskiizeta}
    Z(s; t) := \sum_{\substack{\{P\}\\
    \Tr(P) = t}} \overline{\chi}(P)Z_P(s) = \ell^{-s}\sum_{c \geq 1} \frac{a(c; t)}{c^s},
\end{equation}
for certain coefficients $a(c; t)$. By Lemma \ref{lem:hyperboliczetafunction} the series $Z(s; t)$ has meromorphic continuation to $\text{Re}(s) \geq 1/2$ with at most a single pole at $s = 1$. Furthermore, by formula \eqref{eq:residuezetafunction} we have 
\begin{equation}\label{eq:formulaforprimegeodesicasresidues}
    \text{Res}_{s = 1}Z(s; t) = \sum_{\substack{\{P\}_\Gamma\\
    \Tr(P) = t}} \frac{2 \log N(P_0) \chi(P)}{\pi \vol(\Gamma\backslash \Hyp)(N(P)^{1/2} - N(P)^{-1/2})} ,
\end{equation}
where we have used that $\Tr(P) = \Tr(P^{-1})$ and $\chi(P^{-1}) = \overline{\chi(P)}$. In particular, we can express the sum $\Psi_\Gamma(X; \chi)$ as
\begin{equation}\label{eq:secondformulaforprimegeodesicasresidues}
    \Psi_\Gamma(X; \chi) = \frac{\vol(\Gamma\backslash \Hyp)\pi}{2} \sum_{2 < t \leq X^{1/2} + X^{-1/2}} (t^2 - 4)^{1/2} \text{Res}_{s = 1}Z(s; t).
\end{equation}

\subsection{Average over $t$ of the coefficients of $Z(s; t)$.}

Let us analyze the coefficients $a(c; t)$ more closely. For two elements of $\Gamma$, let us write $P_1 \sim_\ell P_2$ if they are conjugate by an element of $\Gamma_\infty = \{\pm \begin{psmallmatrix}
    1 & \ell\Z\\
    0 & 1
\end{psmallmatrix}\}$, and let us denote the resulting equivalence classes of $\Gamma$ by $\Gamma/\sim_\ell$. The coefficients $a(c; t)$ are given by
\begin{equation*}
    a(c; t) = a^+(c; t) + a^{-}(c; t),
\end{equation*}
where
\begin{equation*}
    a^{\pm}(c; t) := \sum_{\substack{[P] \in \Gamma/\sim_\ell\\
    c(P) = \pm c, \, \Tr(P) = t}} \overline{\chi(P)}.
\end{equation*}
Note that the map $P \mapsto \begin{psmallmatrix}
    1 & \ell\\
    0 & 1
\end{psmallmatrix}P$ establishes a bijection between the two sets
\begin{equation*}
    \{[P] \in \Gamma/\sim_\ell : c(P) = \pm c, \Tr(P) = t\} \longleftrightarrow \{[P] \in \Gamma/\sim_\ell : c(P) = \pm c, \Tr(P) = t \pm c\ell\}.
\end{equation*}
Since $\chi(\begin{psmallmatrix}
    1 & \ell\\
    0 & 1
\end{psmallmatrix}) = e(-\alpha)$, it follows that we have $a^{\pm}(c; t + c\ell) = e(\pm \alpha)a^{\pm}(c; t)$. Thus, we see that
$t \mapsto a^\pm(c; t)e(\mp \alpha t/c\ell)$ is periodic modulo $c\ell$. 

The coefficients $a^{\pm}(c; t)$ can be expressed in terms of Kloosterman sums.
\begin{lem}\label{lem:fourierinversoncoeff}
Let $c \geq 1$. Then 
\begin{equation}\label{eq:fourierinversoncoeff}
    a^{\pm}(c; t) = \frac{\chi(\pm I)}{c\ell}\sum_{n \modulo c\ell} S_{(\Gamma, \chi)}\left(\frac{n-\alpha}{\ell}, \frac{n-\alpha}{\ell}; c\right)e\left(\mp\frac{(n-\alpha)t}{c\ell}\right).
\end{equation}
\end{lem}
\begin{proof}
Since $- I \in \Gamma$, it follows that $a^{-}(c; t) = \chi(-I)a^+(c; -t)$. Thus, it is enough to prove the lemma for $a^+(c; t)$. 
Recalling the definition of Kloosterman sums from \eqref{eq:generalkloostermansum} and grouping the summands by their trace modulo $c\ell$ we see that 
\begin{equation*}
\begin{aligned}
    S_{(\Gamma, \chi)}\left(\frac{n-\alpha}{\ell}, \frac{n-\alpha}{\ell}; c\right) &  =\sum_{\substack{[\gamma] \in \Gamma_\infty \backslash \Gamma/\Gamma_\infty\\
        c(\gamma) = c}}\overline{\chi(\gamma)} e\left(\frac{(n-\alpha)(a(\gamma) + d(\gamma))}{c\ell}\right) \\
        & = \sum_{t \modulo c\ell}\sum_{\substack{[\gamma] \in \Gamma_\infty \backslash \Gamma/\Gamma_\infty\\
        c(\gamma) = c\\
        \Tr(\gamma) = t \modulo c\ell}}\overline{\chi(\gamma)}e\left(-\frac{\alpha t}{c\ell}\right)e\left(\frac{n t}{c\ell}\right).
\end{aligned}
\end{equation*}
By a direct computation $\Tr\left(\begin{psmallmatrix}
         1 & xl\\
        0 &  1
    \end{psmallmatrix}\gamma\right) = \Tr(\gamma) + xc(\gamma)l$ and $c\left(\begin{psmallmatrix}
         1 & xl\\
        0 &  1
    \end{psmallmatrix}\gamma\right) = c(\gamma)$. It follows that the natural projection map $\Gamma/\sim_\ell \rightarrow \Gamma_\infty\backslash \Gamma/\Gamma_\infty$ induces a bijection
\begin{equation*}
    \{[\gamma] \in \Gamma/\sim_\ell : c(\gamma) = c, \Tr(\gamma) = t\} \longleftrightarrow  \{[\gamma] \in \Gamma_\infty \backslash \Gamma /\Gamma_\infty : c(\gamma) = c, \Tr(\gamma) = t \modulo c\ell\}. 
\end{equation*} 
Thus, we obtain 
\begin{equation*}
    S_{(\Gamma, \chi)}\left(\frac{n-\alpha}{\ell}, \frac{n-\alpha}{\ell}; c\right) = \sum_{t \modulo c\ell } \left(a^+(c;t)e\left(\frac{-\alpha t}{c\ell}\right)\right)e\left(\frac{nt}{c\ell}\right).
\end{equation*}
The lemma follows from this identity by Fourier inversion.
\end{proof}

Thanks to Lemma \ref{lem:fourierinversoncoeff} we can obtain precise asymptotics for partial sums of the coefficients $a(c; t)$ over the variable $t$. The following result is a generalization of \cite[Lemma 2.3]{Soundararajan_2013}.
\begin{lem}\label{lem:sumofcoefficients}
If $c,z \geq 1$, then
\begin{equation}\label{eq:sumofcoefficients}
    \sum_{t \leq z} a(c; t) = (1 + \chi(-I))\delta_{\alpha = 0}\frac{S_{(\Gamma, \chi)}(0, 0; c)}{c\ell}  z + O_{q, \varepsilon}(c^{1/2 + \varepsilon}).
\end{equation}
\end{lem}
\begin{proof}
It is enough to look at the partial sums of $a^+(c; t)$, since the case of $a^{-}(c; t)$ is analogous and we have $a(c; t) = a^+(c; t) + a^{-}(c; t)$. By \eqref{eq:fourierinversoncoeff} we have  
\begin{equation*}
    \sum_{t \leq z} a^+(c; t) = \frac{1}{c\ell}\sum_{n \modulo c\ell}S_{(\Gamma, \chi)}\left(\frac{n-\alpha}{\ell},\frac{n-\alpha}{\ell}; c\right)\sum_{t\leq z} e\left(-\frac{(n-\alpha)t}{c\ell}\right).
\end{equation*}
When $\alpha = 0$, the term $n = 0$ gives the contribution
\begin{equation*}
\frac{S_{(\Gamma, \chi)}(0, 0;c\ell)}{c\ell}\sum_{t \leq z}1 =  \frac{S_{(\Gamma, \chi)}(0, 0;c\ell)}{c\ell}z + O_q(1).
\end{equation*}
When $n -\alpha \neq 0$ we have the familiar bound
\begin{equation*}
    \sum_{t \leq z}e\left(-\frac{(n-\alpha)t}{c\ell}\right) \ll ||(n-\alpha)/c\ell||^{-1}
\end{equation*}
where $||x||$ is the distance from $x$ to the nearest integer. Using this estimate together with the bound for Kloosterman sums from Lemma \ref{lem:weilbound} we obtain
\begin{equation*}
\begin{aligned}
    \sum_{t \leq z} a(c; t) & =\delta_{\alpha = 0}\frac{S_{(\Gamma, \chi)}(0, 0; c\ell)}{c\ell}  z + O_{q, \varepsilon}\left(c^{1/2 + \varepsilon}\sum_{1\leq n \leq \lfloor c\ell/2\rfloor } \left(\frac{((n-\alpha)q, c)^{1/2}}{n-\alpha} + \frac{((n+\alpha)q, c)^{1/2}}{n+\alpha}\right)\right)\\
    & = \delta_{\alpha = 0}\frac{S_{(\Gamma, \chi)}(0, 0; c\ell)}{c\ell}  z  + O_{q, \varepsilon}(c^{1/2 + \varepsilon}),
\end{aligned}
\end{equation*}
as desired.
\end{proof}

\subsection{Main term of the $t$-average.}

Suppose that $\alpha = 0$ and $k = 0$. In this case the cusp $\infty$ of $(\Gamma, \chi)$ is singular. Recall the Eisenstein series $E(z; s) =E_\infty(z; s; \Gamma, \chi)$, given as
\begin{equation*}
    E(z; s) = \ell^{-s}\sum_{\gamma \in \Gamma_\infty \backslash \Gamma} \overline{\chi}(\gamma)\text{Im}(\gamma z)^s,
\end{equation*}
where $l\geq 1$ is the width of the cusp $\infty$. We know that 
\begin{equation}\label{eq:constanttermeisenstein}
\int_{0}^{1}E(\ell(x + iy); s) \, dx = y^s +  \phi_{\infty, \infty}(s) y^{1-s}, 
\end{equation}
where 
\begin{equation}\label{eq:factorizationconstantterm}
    \phi_{\infty, \infty}(s) := \sqrt{\pi}\frac{\Gamma(s - 1/2)}{\Gamma(s)}\psi_{(\Gamma, \chi)}(s),
\end{equation}
and
\begin{equation*}
    \psi_{(\Gamma, \chi)}(s) := \ell^{-2s}\sum_{c \geq 1}\frac{S_{(\Gamma, \chi)}(0, 0; c)}{c^{2s}},
\end{equation*}
see \cite[Theorem 3.4]{iwaniec} for the case of trivial $\chi$, the proof of the general case being the same. Thus, $\phi_{\infty, \infty}(s)$ is the $(\infty, \infty)$-entry in the scattering matrix for $(\Gamma, \chi)$. By \eqref{eq:constanttermeisenstein} and \eqref{eq:factorizationconstantterm}, the function $\psi_{(\Gamma, \chi)}(s)$ inherits the holomorphicity properties of $E(z; s)$. In particular $\psi_{(\Gamma, \chi)}(s)$ is holomorphic on $\text{Re}(s) \geq 1/2$ except for a possible simple pole at $s = 1$ of residue
\begin{equation*}
    \text{Res}_{s = 1}\psi_{(\Gamma, \chi)}(s) = \frac{\delta_{\chi = 1}}{\pi \vol(\Gamma\backslash \Hyp)}.
\end{equation*}
In order to handle an average over $c$ of the main term of \eqref{eq:sumofcoefficients} using Mellin inversion we need to control the growth of $\psi_{(\Gamma, \chi)}(s)$ on vertical strips.
\begin{lem}\label{lem:growthofconstantterm}
Let $s = \sigma + it$. Then, for $\sigma \geq 1/2$ and $|t| \geq 1$ it holds that 
\begin{equation*}
    \psi_{(\Gamma, \chi)}(s) \ll_{q, \varepsilon} |s|^{\max(1-\sigma, 0) + \varepsilon}.
\end{equation*}
\end{lem}
\begin{proof}
From Section 2.7 in the proof of \cite[Theorem 1]{reznikoveisensteinmatrix}, it follows that $\psi_{(\Gamma, \chi)}(s)$ is a linear combination of functions of the form 
\begin{equation*}
    \frac{L(2s - 1, \chi)}{L(2s, \chi)} \prod_{p \in S}P_p(p^s) ,
\end{equation*}
where $L(s, \chi)$ is the Dirichlet $L$-function attached to a character $\chi$, $S$ is a finite set of primes, and $P_p(X)$ is a rational function for each $p \in S$. The lemma follows by standard bounds for Dirichlet $L$-functions.
\end{proof}
\begin{remark}
The same bound is true for a general cofinite group $\Gamma$ and a character $\chi: \Gamma \rightarrow S^1$. The idea is to apply the Phrägmen-Lindelöf principle. First, one uses the bound $S_{(\Gamma, \chi)}(0, 0; c) \ll_\Gamma c$ to deduce that $\psi_{(\Gamma, \chi)}(s)$ converges absolutely on $\text{Re}(s) > 1$. Since the scattering matrix $\Phi(s)$ is unitary on $\text{Re}(s) = 1/2$, its $(\infty, \infty)$-entry $\phi_{\infty, \infty}(s)$ satisfies $|\phi_{\infty, \infty}(s)| \leq 1$ on $\text{Re}(s) = 1/2$. By \eqref{eq:factorizationconstantterm} and Stirling's formula we deduce that $|\psi_{(\Gamma, \chi)}(s)| \ll |s|^{1/2}$ on $\text{Re}(s) = 1/2$. As an outcome of a proof of the analytic continuation of $E(z; s)$ one obtains some information about the growth of $\psi_{(\Gamma, \chi)}$ on vertical strips. For example, from (6.11), (6.12) and (6.15) in \cite[Proposition 6.1]{iwaniec} it follows that $|\psi_{(\Gamma, \chi)}(s)| = O_{q, \varepsilon}(\exp(|s|^{8 + \varepsilon}))$ for $\text{Re}(s) \geq 1/2$ and $|t| \geq 1$ (Iwaniec's proof assumes that $\chi$ is trivial, but the proofs in the general case are similar). At this point one can apply the Phrägmen-Lindelöf principle to deduce the bound of Lemma \ref{lem:growthofconstantterm}.
\end{remark}

We can now evaluate the sum over $c$ of the main term from \eqref{eq:sumofcoefficients}.
\begin{lem}\label{lem:zerothsumofkloostermansumevaluation}
Let $\omega \in C_c^\infty(0, \infty)$ be such that $\int_\R \omega(v)\, dv = 1$, and let $V > 0$ be a parameter. Then 
\begin{equation}\label{eq:zerothsumofkloostermansumevaluation}
    \frac{\vol(\Gamma\backslash \Hyp)\pi}{2V}\sum_{c \geq 1} \frac{S_{(\Gamma, \chi)}(0, 0; c)}{c\ell}\omega(c\ell/V) = \delta_{\chi = 1} + O_q(V^{-1}).
\end{equation}
\end{lem}

\begin{proof}
Let $\widehat{\omega}(s) := \int_0^\infty \omega(v)v^s \, dv/v$ be the Mellin transform of $\omega$. On any fixed vertical strip we have $|\widehat{\omega}(s)| \ll_A (1 + |s|)^{-A}$ for any $A > 0$. By Mellin inversion it follows that  
\begin{equation*}
    \frac{\vol(\Gamma\backslash \Hyp)\pi}{2V}\sum_{c \geq 1} \frac{S_{(\Gamma, \chi)}(0, 0; c)}{c\ell}\omega(c\ell/V) = \frac{\vol(\Gamma\backslash \Hyp)\pi}{4\pi i}\int_{(\sigma)} \psi_{(\Gamma, \chi)}\left(\frac{1 + s}{2}\right)\widehat{\omega}(s)V^{s-1}\, ds,
\end{equation*}
where $\sigma > 1$. Moving the contour to the left is justified by Lemma \ref{lem:growthofconstantterm} and the decay of $\widehat{\omega}(s)$. If we move the contour to $\text{Re}(s) = 0$ and pick up the residue at $s = 1$ when $\chi$ is trivial we deduce that 
\begin{equation*}
\begin{aligned}
\frac{\vol(\Gamma\backslash \Hyp)\pi}{2V}\sum_{c \geq 1} \frac{S_{(\Gamma, \chi)}(0, 0; c)}{c\ell}\omega(c\ell/V) & = \delta_{\chi = 1} + \frac{\vol(\Gamma\backslash \Hyp)\pi}{4\pi i}\int_{(0)} \psi_{(\Gamma, \chi)}\left(\frac{1 + s}{2}\right)\widehat{\omega}(s)V^{s-1}\, ds\\
& = \delta_{\chi = 1} + O_q(V^{-1}),
\end{aligned}
\end{equation*} 
as desired.
\end{proof}

\subsection{Estimate for $Z(s; t)$ on the critical line.}

In this section we estimate $Z(s; t)$ on the critical line. Let 
\begin{equation*}
    Z^{\pm}(s; t) := \ell^{-s}\sum_{c = 1}^\infty \frac{a^{\pm}(c; t)}{c^s}.
\end{equation*}
It is clear that 
\begin{equation}\label{eq:pmzetafunction}
    Z^{-}(s; t) = \chi(-I)Z^{+}(s; -t),
\end{equation}
and by \eqref{eq:defiofzagierbykovskiizeta} we have 
\begin{equation*}
    Z(s; t) = Z^{+}(s; t) + Z^{-}(s; t).
\end{equation*}
Recall Definition \ref{defi:admissibleexponent} regarding admissible exponents.

\begin{prop}\label{prop:boundzagierzetafunction}
Let $\theta> 0$ be any admissible exponent. There exists $A > 0$ such that for any $\varepsilon > 0$ the bound
    \begin{equation*}
        Z^{\pm}(s; t) \ll_{q, \theta, \varepsilon} (1 + |s|)^A t^{2\theta + \varepsilon} 
    \end{equation*}
holds uniformly for $|t| \geq 3$, $\text{Re}(s) \geq 1/2$ and $|s - 1| \geq \varepsilon$. 
\end{prop}

\begin{proof}
By \eqref{eq:pmzetafunction} it is enough to consider the case of $Z^+(s; t)$. We write $\gamma_1 \sim_q \gamma_2$ if the matrices are conjugate by an element of $\{\pm \begin{psmallmatrix}
    1 & q\Z\\
    0 & 1
\end{psmallmatrix}\}$, and let $\Gamma/\sim_q$ denote the corresponding equivalence classes. Note that we can write
\begin{equation*}
    Z^+(s; t) = \frac{\ell^{1-s}}{q}  \sum_{c = 1}^\infty \frac{b(c; t)}{c^s},
\end{equation*}
where 
\begin{equation*}
    b(c; t) = \sum_{\substack{[P] \in \Gamma/\sim_q\\
    c(P) = c, \, \Tr(P) = t}} \overline{\chi(P)}.
\end{equation*}
Recall the group $\widetilde{\Gamma} := \Gamma(q)\backslash \Gamma$. Since $\chi$ is trivial on $\Gamma(q)$, we can group matrices  by their reduction modulo $q$ and write
\begin{equation*}
    Z^+(s; t) = \frac{\ell^{1-s}}{q}\sum_{\substack{\widetilde{\gamma} \in \widetilde{\Gamma}\\
    \Tr(\widetilde{\gamma}) = t \modulo q}} \overline{\chi\left(\widetilde{\gamma}\right)} \sum_{c = 1}^\infty \frac{b(c; t; \widetilde{\gamma})}{c^s},
\end{equation*}
where 
\begin{equation*}
    b(c; t; \widetilde{\gamma}) := \left|\{[\gamma] \in \Gamma/\sim_q : \gamma = \widetilde{\gamma} \modulo q, c(\gamma) = c, \Tr(\gamma) = t\}\right|.
\end{equation*}
We can evaluate the cardinality $b(c; t; \widetilde{\gamma})$ by classifying equivalence classes $[\gamma] \in \Gamma/\sim_q$ according to the residue class $a(\gamma) \modulo qc$. Recall that 
\begin{equation*}
    \begin{pmatrix}
        1 & yq\\
        0 & 1
    \end{pmatrix}\begin{pmatrix}
        a & b\\
        c & d
    \end{pmatrix}\begin{pmatrix}
        1 & -yq\\
        0 & 1
    \end{pmatrix} = \begin{pmatrix}
        a + yqc & *\\
        c & d - yqc
    \end{pmatrix}.
\end{equation*}
It follows immediately that 
\begin{equation*}
  b(c; t; \widetilde{\gamma}) = \delta_{\{c = c(\widetilde{\gamma})\modulo q\}}\times |\{a \modulo qc: a = a(\widetilde{\gamma}) \modulo q \text{ and } a^2 - at + 1 = -b(\widetilde{\gamma})c \modulo qc\}|.
\end{equation*}
For each $c \geq 1$ we write $c = c_0 c_q$ where $(c_0, q) = 1$ and $c_q \mid q^\infty$. By the Chinese remainder theorem we can write
\begin{equation*}
    b(c; t; \widetilde{\gamma}) = b_0(c_0; t) \times b_q(c_q; t; \widetilde{\gamma}; c_0 \modulo q),
\end{equation*}
where 
\begin{equation*}
    b_0(c_0; t) := \left|\{a \modulo c_0 : a^2 -ta + 1 = 0 \modulo c_0\}\right|
\end{equation*}
and, for $r \in (\Z/q\Z)^\times$,
\begin{equation*}
    b_q(c_q; t; \widetilde{\gamma}; r) := \delta_{\{c_q r = c(\widetilde{\gamma}) \modulo q\}}\times \left|\{a \modulo q c_q: a = a(\widetilde{\gamma}) \modulo q \text{ and } a^2 - at + 1 = -b(\widetilde{\gamma})rc_q \modulo qc_q\}\right|.
\end{equation*}
Let us define 
\begin{equation*}
    Z_q(s; t; \widetilde{\gamma}; r) := \sum_{c \mid q^\infty} \frac{b_q(c; t; \widetilde{\gamma}; r)}{c^s},
\end{equation*}
as well as
\begin{equation*}
    Z_0(s; t; \psi) := \sum_{(c, q) = 1} \frac{b_0(c; t)\psi(c)}{c^s},
\end{equation*}
where $\psi$ is a Dirichlet character modulo $q$. Then by orthogonality of characters we have
\begin{equation}\label{eq:expressionforgeneralzagierzeta}
    Z^+(s; t) = \frac{\ell^{1-s}}{q\varphi(q)}\sum_{\substack{\widetilde{\gamma} \in \widetilde{\Gamma}\\
    \Tr(\widetilde{\gamma}) = t}}\overline{\chi(\widetilde{\gamma})}\sum_{a\in (\Z/q\Z)^\times}Z_q(s; t; \widetilde{\gamma}; a)\sum_{\psi \modulo q} \overline{\psi(a)}Z_0(s; t; \psi).
\end{equation}
Write $t^2 - 4 = Df^2$ where $D$ is a fundamental discriminant. According to \cite[Proposition 3]{zagiermodularformsfouriercoeff} we have
\begin{equation*}
    \sum_{c = 1}^\infty \frac{b_0(c; t)}{c^s} = \frac{\zeta(s)}{\zeta(2s)}L(s; \psi_D)\sum_{d \mid f} \mu(d)\left(\frac{D}{d}\right)d^{-s}\sigma_{1 - 2s}(f/d).
\end{equation*}
Here $\psi_D(\cdot) = \left(\frac{D}{\cdot}\right)$ is the Kronecker symbol and $\sigma_\nu(m) = \sum_{\substack{d \mid m\\d>0}} d^\nu$. Since $c \mapsto \psi(c)$ is completely multiplicative, it follows by standard manipulations with Dirichlet series that
\begin{equation*}
    Z_0(s; t; \psi) = H(s; t; \psi) \frac{L(s; \psi)}{L(2s; \psi)}L(s; \psi_D \psi)
\end{equation*}
where $H(s; t;\psi)$ is a Dirichlet series depending on $t$, $q$ and $\psi$ which is holomorphic on $\text{Re}(s) \geq 1/2$ and satisfies
\begin{equation*}
    H(s; t; \psi) \ll_{q} \log(3 + |s|)
\end{equation*}
in the same region, independently of $t$. By the definition of admissible exponent $\theta > 0$ we have, for some $A> 0$, the bound
\begin{equation*}
    L(s; \psi_D \psi) \ll_{q, \varepsilon} (t^2 - 4)^{\theta + \varepsilon}(1 + |s|)^A,
\end{equation*}
where $\text{Re}(s) \geq 1/2$. We deduce that, for some $A > 0$, we have
\begin{equation}\label{eq:boundtwistedzagierbykovskii}
     Z_0(s; t; \psi) \ll_{q, \varepsilon} (t^2 - 4)^{\theta + \varepsilon} (1 + |s|)^A
\end{equation}
in the region $\text{Re}(s) \geq 1/2$. It remains to bound $Z_q(s; t; \widetilde{\gamma}; r)$. It is easy to check that if $c \mid q^\infty$, we have $b(c; t; \widetilde{\gamma}; r) \ll_q (t^2 - 4, c)^{1/2}$. If we write $t^2 - 4 = \prod_{p} p^{e_p}$, then we have
\begin{equation}\label{eq:boundbadpartzagier}
    Z_q(s; t; \widetilde{\gamma}; r) \ll_q \prod_{p\mid q} \left(e_p  + \frac{1}{1 - p^{-1/2}}\right) \ll_{q, \varepsilon} t^\varepsilon
\end{equation}
in the region $\text{Re}(s) \geq 1/2$. The proof of Proposition \ref{prop:boundzagierzetafunction} is concluded by combining \eqref{eq:expressionforgeneralzagierzeta} with \eqref{eq:boundtwistedzagierbykovskii} and \eqref{eq:boundbadpartzagier}.
\end{proof}

\section{Applications of the Bykovskii-Zagier series.}\label{sec:shortintervals}

\subsection{Proof of Theorem \ref{thm:soundyoungshortintervals}.}\label{subsec:shortintervals}

Let $2 \leq u \leq X$. For conciness, let $q(y)$ be the function $q(y) := y^{1/2} + y^{-1/2}$. By equation \eqref{eq:secondformulaforprimegeodesicasresidues}, we can express the difference $\Psi_\Gamma(X + u; \chi) - \Psi_\Gamma(X; \chi)$ as
\begin{equation*}
    \Psi_\Gamma(X + u; \chi) - \Psi_\Gamma(X; \chi) = \frac{\vol(\Gamma\backslash \Hyp)\pi}{2}\sum_{q(X) < t \leq q(X + u)} (t^2 - 4)^{1/2}\text{Res}_{s = 1}Z(s; t).
\end{equation*}
Let $V > 0$ be a parameter to be chosen later, and let $\omega \in C_c^\infty(0, \infty)$ satisfy $\int_\R \omega(x) \, dx = 1$. Consider the approximation 
\begin{equation}\label{eq:approximationtogeometricsideshortintervals}
    \S_V := \frac{\vol(\Gamma\backslash \Hyp)\pi}{2V}\sum_{c\geq 1}\sum_{q(X) < t \leq q(X + u)} (t^2 - 4)^{1/2} a(c; t)\omega(c \ell/V).
\end{equation}
If we execute first the summation over $c$ and use Mellin inversion we obtain
\begin{equation*}
    \S_V = \frac{\vol(\Gamma\backslash \Hyp)\pi}{4\pi i}\int_{(\sigma)} \left(\sum_{q(X) < t \leq q(X + u)} (t^2 - 4)^{1/2}Z(s; t)\right) \widehat{\omega}(s)V^{s-1}\, ds,
\end{equation*}
where $\sigma > 1$. Moving the contour to $\sigma = 1/2$ and picking up the residue at $s = 1$ we get
\begin{equation*}
    \S_V =  \Psi_\Gamma(X + u; \chi) - \Psi_\Gamma(X; \chi)  + \frac{\vol(\Gamma\backslash \Hyp)\pi}{4\pi i}\int_{(1/2)} \left(\sum_{q(X) < t \leq q(X + u)} (t^2 - 4)^{1/2}Z(s; t)\right) \widehat{\omega}(s)V^{s-1}\, ds.
\end{equation*}
Note that $q(X) \sim X^{1/2}$ and $q(X + u) - q(X) \sim uX^{-1/2}/2$. Bounding $Z(s; t)$ via Proposition \ref{prop:boundzagierzetafunction} we get
\begin{equation}\label{eq:shortintervalssumoverc}
    \S_V = \Psi_\Gamma(X + u; \chi) - \Psi_\Gamma(X; \chi) + O_{q, \varepsilon}(u X^\theta V^{-1/2 + \varepsilon}).
\end{equation}
We now come back to \eqref{eq:approximationtogeometricsideshortintervals} and execute the summation over $t$. Using summation by parts and Lemma \ref{lem:sumofcoefficients} it follows that 
\begin{equation*}
\begin{aligned}
    \sum_{q(X) < t \leq q(X + u)} (t^2 - 4)^{1/2}a(c; t) & = \delta_{\alpha = 0}(1 + \chi(-I))\frac{S_{(\Gamma, \chi)}(0, 0; c)}{c \ell}\int_{q(X)}^{q(X + u)}(t^2 - 4)^{1/2}\,dt +  O_{q, \varepsilon}(c^{1/2 + \varepsilon} X^{1/2})\\
    & + O_{q, \varepsilon}\left(c^{1/2 + \varepsilon} \int_{q(X)}^{q(X + u)} 1 \, dt\right)\\
    & =  \delta_{\alpha = 0}(1 + \chi(-I))\frac{S_{(\Gamma, \chi)}(0, 0; c)}{2c \ell} u + O_{q, \varepsilon}(c^{1/2 + \varepsilon} X^{1/2}).
\end{aligned}
\end{equation*}
Therefore, the sum $\S_V$ can be evaluated as
\begin{equation*}
    \S_V = \delta_{\alpha = 0} \frac{\pi \vol(\Gamma\backslash \Hyp)(1 + \chi(-I))}{2}\sum_{c \geq 1}\frac{S_{(\Gamma, \chi)}(0, 0; c)}{2c \ell V}\omega(c \ell / V) u + O_{q, \varepsilon}(V^{1/2 + \varepsilon}X^{1/2}). 
\end{equation*}
If we use \eqref{eq:zerothsumofkloostermansumevaluation} we can simplify this formula to 
\begin{equation}\label{eq:shortintervalssumovert}
    \S_V = \delta_{\chi = 1} u + O_{q, \varepsilon}(V^{1/2 + \varepsilon} X^{1/2} + uV^{-1}).
\end{equation}
Comparing \eqref{eq:shortintervalssumoverc} and \eqref{eq:shortintervalssumovert} we obtain
\begin{equation*}
    \Psi_\Gamma(X + u; \chi) - \Psi_\Gamma(X; \chi) = \delta_{\chi = 1}u + O_{q, \varepsilon}(u X^{\theta}V^{-1/2 + \varepsilon} + V^{1/2 + \varepsilon} X^{1/2}).
\end{equation*}
If we optimize the error term by letting $V = uX^{-1/2 + \theta}$ we arrive at 
\begin{equation*}
    \Psi_\Gamma(X + u; \chi) - \Psi_\Gamma(X; \chi) = \delta_{\chi = 1} u + O_{q, \varepsilon}(u^{1/2} X^{1/4 + \theta/2 + \varepsilon}),
\end{equation*}
as desired.

\subsection{Proof of Theorem \ref{thm:balkanovafrolenkovbound}.}\label{subsec:proofofboundBF}

Let $X \geq 1$ sufficiently large, let $1 \leq T \leq X^{1/2}$ and consider $\beta$ defined by $2\beta = \log X + i/T$. We are going to apply the Selberg trace formula (Lemma \ref{lem:selbergtraceformula}) with the test function
\begin{equation*}
    h(t) = \frac{\sinh(\pi + 2i\beta)t}{\sinh(\pi t)}.
\end{equation*}
Recall that there are no small eigenvalues in the case $k = 1$. On the spectral side we get
\begin{equation*}
    \Lambda_{\text{Res}}(h) + \Lambda_{\text{cusp}}(h) = \delta_{\chi = 1}h(i/2) + \S_0(T, X) + O_q(1 + \delta_{k = 0} X^{7/64}\log X),
\end{equation*}
where we used \eqref{eq:approximationtosmoothsum}.
The Fourier transform of $h$ is easily computed to be 
\begin{equation*}
    g(x) = \frac{-i\sinh(2\beta)}{4}\frac{1}{\cosh^2(x/2) - \cosh^2(\beta)},
\end{equation*}
recall \eqref{eq:pairtestfunctionselbergtraceformula}. Note that $\text{Re}(4\cosh^2(\beta)) \sim X$ and $\text{Im}(4\cosh^2(\beta)) \sim X/T$. It follows easily that 
\begin{equation*}
\int_0^\infty |g(x)|\, dx \ll X \int_1^\infty \frac{1}{u + u^{-1} + 2 - 4\cosh^2(\beta)} \frac{du}{u} \ll \log X.
\end{equation*}
Also, we know that   
\begin{equation*}
    \int_{-R}^{R}\left|\frac{\phi'}{\phi}(1/2 + it)\right|\, dt + \int_{-R}^{R}\left|\frac{\Gamma'}{\Gamma}(1 + it)\right|= O(R\log R),
\end{equation*}
recall \eqref{eq:meanvalueestimatescatteringdeterminant} and Stirling's formula. At this point it is straightforward to check that
\begin{equation*}
    \Lambda_{\text{Id}}(h) = O(\log X), \quad \Lambda_{\text{Ell}}(h) = O\left( \int_\R |g(u)| \, du\right) = O(\log X), \quad \Lambda_{\text{Par} - \text{Cont}}(h) = O(\log X + T\log T).
\end{equation*}
Thus, an application of the Selberg trace formula to the congruence character $(\Gamma, \chi)$ and the test functions $h$ produces the identity
\begin{equation}\label{eq:evaluationselbergtraceformulav1}
    \delta_{\chi = 1}h(i/2) + \S_0(T, X) = \sum_{\substack{\{P\}_\Gamma\\
    \Tr(P) > 2}}\frac{\chi(P)\log N(P_0)}{N(P)^{1/2} - N(P)^{-1/2}}g(\log N(P)) + O(T\log T + \log X + \delta_{k = 0}X^{7/64})
\end{equation}
It remains to evaluate the hyperbolic contribution. Note that $\Tr(P)$ and $N(P)$ are related by $\Tr(P) = N(P)^{1/2} + N(P)^{-1/2}$. It follows that 
\begin{equation*}
    g(\log N(P)) = -\frac{i\sinh(2\beta)}{\Tr(P)^2 - 4\cosh^2 \beta}.
\end{equation*}
Accordingly, let us write 
\begin{equation}\label{eq:explicitformulaf}
    f(t) := -\frac{i\sinh(2\beta)}{t^2 - 4\cosh^2 \beta}.
\end{equation}
We have $f(2\cosh(x/2)) = g(x)$. It follows that 
\begin{equation}\label{eq:preciseevaluationfintegral}
    \int_{2}^\infty f(t)\, dt = \int_0^\infty g(x) \sinh(x/2)\, dx = \frac{h(i/2)}{2} + O(1). 
\end{equation}
Using \eqref{eq:formulaforprimegeodesicasresidues} we deduce that
\begin{equation*}
    \Lambda_{\text{Hyp}}(h) = \sum_{\substack{\{P\}_\Gamma\\
    \Tr(P) > 2}} \frac{\chi(P)\log N(P_0)}{N(P)^{1/2} - N(P)^{-1/2}}g(\log N(P)) = \frac{\vol(\Gamma\backslash \Hyp) \pi}{2} \sum_{t > 2} f(t)\text{Res}_{s = 1}Z(s; t).
\end{equation*}
Introduce a parameter $1 \leq V \leq X^{1/2}$, and let $\omega \in C_c^\infty((0, \infty))$ be such that $\int_\R \omega(x) \, dx = 1$. Consider the double sum
\begin{equation}\label{eq:approximationtogeometricside}
    \S_V := \frac{\vol(\Gamma\backslash \Hyp)\pi}{2V}\sum_{t > 2}\sum_{c\geq 1} a(c; t)f(t)\omega(c \ell/V),
\end{equation}
which approximates the hyperbolic contribution. If we argue as in Section \ref{subsec:shortintervals}, we can execute the sum over $c$, apply Mellin inversion, move the contour to the left and pick up the residue at $s = 1$, obtaining
\begin{equation*}
    \S_V = \Lambda_{\text{hyp}}(h) + O\left(V^{-1/2}\int_{(1/2)} \sum_{t > 2}\left|f(t) Z(s; t) \widehat{\omega}(s)\right| \, |ds|\right).
\end{equation*}
If we estimate $Z(s; t)$ by Proposition \ref{prop:boundzagierzetafunction}, and bound $f(t)$ directly from the explicit formula \eqref{eq:explicitformulaf} we obtain
\begin{equation}\label{eq:approxtogeometricsidesumoverc}
    \S_V = \Lambda_{\text{hyp}}(h) + O_\varepsilon(V^{-1/2} X^{1/2 + \theta + \varepsilon}).
\end{equation}
On the other hand, we can deal with \eqref{eq:approximationtogeometricside} differently by first summing over $t$. Using Lemma \ref{lem:sumofcoefficients} and summation by parts we arrive at 
\begin{equation*}
\begin{aligned}
    \sum_{t > 2}a(c; t)f(t) & = -(1 + \chi(-I))\delta_{\alpha = 0}\frac{S_{(\Gamma, \chi)}(0, 0; c)}{c \ell}\left(\int_2^\infty f'(t) t\, dt\right) + O_\varepsilon\left(c^{1/2 + \varepsilon} \int_2^\infty |f'(t)|\, dt\right)\\
    & = (1 + \chi(-I))\delta_{\alpha = 0}\frac{S_{(\Gamma, \chi)}(0, 0; c)}{c \ell}\left(\int_2^\infty f(t)\, dt\right) + O(1) +  O_\varepsilon\left(c^{1/2 + \varepsilon} \int_2^\infty |f'(t)|\, dt\right).
\end{aligned}
\end{equation*}
From the definition of $f$ it is easy to check that $\int_2^\infty |f'(t)|\, dt \ll T$. It follows that
\begin{equation*}
    \sum_{t > 2}a(c; t)f(t) = (1 + \chi(-I))\delta_{\alpha = 0}\frac{S_{(\Gamma, \chi)}(0, 0; c)}{2c \ell} h(i/2) + O_\varepsilon(c^{1/2 + \varepsilon}T),
\end{equation*}
where we have used \eqref{eq:preciseevaluationfintegral}. After summing over $c$ we get
\begin{equation}\label{eq:intermediateapproxtoSv}
    \S_V = \delta_{\alpha = 0}(1 + \chi(-I))\frac{\vol(\Gamma\backslash \Hyp)\pi}{2V}\sum_c \frac{S_{(\Gamma, \chi)}(0, 0; c)}{2c \ell}\omega(c \ell/V) h(i/2)+ O_\varepsilon(V^{1/2 + \varepsilon} T).
\end{equation}
Finally, applying Lemma \ref{lem:zerothsumofkloostermansumevaluation} we obtain
\begin{equation}\label{eq:approxtogeometricsidesumovert}
    \S_V = \delta_{\chi = 1}h(i/2) + O_\varepsilon(X^{1/2}V^{-1} + V^{1/2 + \varepsilon}T).
\end{equation}
If we put together \eqref{eq:approxtogeometricsidesumoverc} and \eqref{eq:approxtogeometricsidesumovert} we see that
\begin{equation*}
    \Lambda_{\text{Hyp}}(h) = \delta_{\chi = 1} h(i/2) + O_\varepsilon(X^{1/2 + \theta + \varepsilon}V^{-1/2} + V^{1/2 + \varepsilon}T).
\end{equation*}
We can optimize the error term by taking $V = X^{1/2 + \theta}T^{-1}$, obtaining
\begin{equation*}
    \Lambda_{\text{Hyp}}(h) = \delta_{\chi = 1} h(i/2) + O_\varepsilon(X^{1/4 + \theta/2+ \varepsilon}T^{1/2}).
\end{equation*}
Combining this formula with \eqref{eq:evaluationselbergtraceformulav1} we arrive at the estimate
\begin{equation*}
    \S_0(T, X) =  O_\varepsilon(X^{1/4 + \theta/2+ \varepsilon}T^{1/2}).
\end{equation*}
By Proposition \ref{prop:conclusionsmoothingsum}, this bound implies Theorem \ref{thm:balkanovafrolenkovbound}.

\printbibliography

\footnotesize
\textit{Email address}: \, \texttt{alberto.reche.23@ucl.ac.uk}\par\nopagebreak
\textsc{Department of Mathematics, University College London, 25 Gordon Street, London WC1H 0AY, United Kingdom}
  
\end{document}